\documentclass[a4paper,11pt]{article}
\usepackage{amsmath,amsthm}
\usepackage{authblk}
\usepackage[style=numeric-comp,maxbibnames=99,bibencoding=utf8,firstinits=true]{biblatex}
\usepackage{booktabs}
\usepackage{cancel}
\usepackage{cases}
\usepackage{colortbl}
\usepackage[hmargin={26mm,26mm},vmargin={30mm,35mm}]{geometry}
\usepackage{nicefrac}
\usepackage[colorlinks,allcolors={blue}]{hyperref}
\usepackage{paralist}
\usepackage{subcaption}
\usepackage[compact,small]{titlesec}
\usepackage{tikz-cd}

\usepackage{newtxtext}
\usepackage{newtxmath}

\bibliography{presto,minivem}
\AtBeginBibliography{\footnotesize}

\newcommand{\email}[1]{\href{mailto:#1}{#1}}

\numberwithin{equation}{section}

\newtheorem{theorem}{Theorem}

\newtheorem{lemma}[theorem]{Lemma}
\newtheorem{corollary}[theorem]{Corollary}
\theoremstyle{remark}
\newtheorem{remark}[theorem]{Remark}
\theoremstyle{definition}
\newtheorem{assumption}[theorem]{Assumption}

\newcommand{\st}{\,:\,}
\newcommand{\Real}{\mathbb{R}}

\DeclareRobustCommand{\bvec}[1]{\boldsymbol{#1}}
\newcommand{\uvec}[1]{\underline{\bvec{#1}}}
\newcommand{\cvec}[1]{\bvec{\mathcal{#1}}}

\newcommand{\rotation}[1]{\varrho_{#1}}

\DeclareMathOperator{\GRAD}{\bf grad}
\DeclareMathOperator{\CURL}{\bf curl}
\DeclareMathOperator{\DIV}{div}
\DeclareMathOperator{\ROT}{rot}
\DeclareMathOperator{\VROT}{\bf rot}
\DeclareMathOperator{\LAPL}{\Delta}

\newcommand{\compl}{{\rm c}}
 
\newcommand{\Hcurl}[1]{\bvec{H}(\CURL;#1)}
\newcommand{\Hscurl}[2]{\bvec{H}^{#1}(\CURL;#2)}

\newcommand{\Hdiv}[1]{\bvec{H}(\DIV;#1)}

\newcommand{\Xgrad}[2][k]{\underline{X}_{\GRAD,#2}^{#1}}
\newcommand{\XgradO}[1]{\underline{X}_{\GRAD,#1,0}^k}
\newcommand{\Xcurl}[2][k]{\underline{\bvec{X}}_{\CURL,#2}^{#1}}
\newcommand{\Xdiv}[2][k]{\underline{\bvec{X}}_{\DIV,#2}^{#1}}
\newcommand{\Xbullet}[1]{\underline{X}_{\bullet,#1}^k}

\newcommand{\Igrad}[1]{\underline{I}_{\GRAD,#1}^k}
\newcommand{\Icurl}[1]{\uvec{I}_{\CURL,#1}^k}
\newcommand{\Idiv}[1]{\uvec{I}_{\DIV,#1}^{k}}

\newcommand{\RT}[1]{\boldsymbol{\mathcal{RT}}^{#1}}

\newcommand{\lproj}[2]{\Pi_{\mathcal{P},#2}^{#1}}
\newcommand{\vlproj}[2]{\boldsymbol{\Pi}_{\cvec{P},#2}^{#1}}

\newcommand{\Rproj}[2]{\bvec{\Pi}_{\cvec{R},#2}^{#1}}
\newcommand{\Rcproj}[2]{\bvec{\Pi}_{\cvec{R},#2}^{\compl,#1}}

\newcommand{\Gproj}[2]{\bvec{\Pi}_{\cvec{G},#2}^{#1}}
\newcommand{\Gcproj}[2]{\bvec{\Pi}_{\cvec{G},#2}^{\compl,#1}}

\newcommand{\uGT}[1][k]{\uvec{G}_T^{#1}}

\newcommand{\uCT}[1][k]{\uvec{C}_T^{#1}}

\newcommand{\uGh}[1][k]{\uvec{G}_h^{#1}}
\newcommand{\uCh}[1][]{\uvec{C}_h^k}
\newcommand{\Dh}[1][]{D_h^k}

\newcommand{\GE}[1][k]{G_E^{#1}}
\newcommand{\cGF}[1][k]{\boldsymbol{\mathsf{G}}_F^{#1}}
\newcommand{\cGT}[1][k]{\boldsymbol{\mathsf{G}}_T^{#1}}
\newcommand{\cGh}[1][k]{\boldsymbol{\mathsf{G}}_h^{#1}}

\newcommand{\CF}{C_F^k}
\newcommand{\cCT}{\boldsymbol{\mathsf{C}}_T^k} 
\newcommand{\cCh}{\boldsymbol{\mathsf{C}}_h^k} 

\newcommand{\DT}{D_T^k}

\newcommand{\trF}{\gamma_F^{k+1}}
\newcommand{\trFt}{\bvec{\gamma}_{{\rm t},F}^k}

\newcommand{\faces}[1]{\mathcal{F}_{#1}}
\newcommand{\edges}[1]{\mathcal{E}_{#1}}
\newcommand{\vertices}[1]{\mathcal{V}_{#1}}

\newcommand{\FT}{\faces{T}}
\newcommand{\ET}{\edges{T}}
\newcommand{\EF}{\edges{F}}
\newcommand{\VT}{\vertices{T}}
\newcommand{\VF}{\vertices{F}}

\newcommand{\normal}{\bvec{n}}
\newcommand{\tangent}{\bvec{t}}

\newcommand{\Poly}[2][]{\mathcal{P}_{#1}^{#2}}

\newcommand{\vPoly}[2][]{\cvec{P}_{#1}^{#2}}
\newcommand{\Roly}[1]{\cvec{R}^{#1}}
\newcommand{\Goly}[1]{\cvec{G}^{#1}}
\newcommand{\cRoly}[1]{\cvec{R}^{\compl,#1}}
\newcommand{\cGoly}[1]{\cvec{G}^{\compl,#1}}

\newcommand{\norm}[2][]{\|#2\|_{#1}}
\newcommand{\seminorm}[2][]{|#2|_{#1}}
\newcommand{\vvvert}{\vert\kern-0.25ex\vert\kern-0.25ex\vert}
\newcommand{\tnorm}[2][]{\vvvert #2\vvvert_{#1}}

\newcommand{\term}{\mathfrak{T}}

\DeclareMathOperator{\Ker}{Ker}

\newcommand{\Mh}[1][h]{\mathcal{M}_{#1}}
\newcommand{\Th}[1][h]{\mathcal{T}_{#1}}
\newcommand{\Fh}[1][h]{\mathcal{F}_{#1}}
\newcommand{\Eh}[1][h]{\mathcal{E}_{#1}}
\newcommand{\Vh}{\mathcal{V}_h}

\newcommand{\Pgrad}[1][k+1]{P_{\GRAD,T}^{#1}}
\newcommand{\Pcurl}[1][T]{\bvec{P}_{\CURL,#1}^k}
\newcommand{\Pdiv}[1][T]{\bvec{P}_{\DIV,#1}^k}
\newcommand{\Pbullet}[1][l]{P_{\bullet,T}^{#1}}

\newcommand{\dE}{\: }
\newcommand{\dPP}{\: }
\newcommand{\ds}{\: }

\newcommand{\df}{\: }

\newcommand\xxE{{\boldsymbol x}_{T}}

\newcommand\xxf{{\boldsymbol x}_{F}}

\renewcommand\vv{\boldsymbol v}
\newcommand\ff{\boldsymbol f}
\newcommand\uu{\boldsymbol u}

\newcommand\xx{\boldsymbol x}
\newcommand\var{\boldsymbol \varphi}

\newcommand\ww{\boldsymbol w}

\newcommand{\kdP}{k-1}
\newcommand{\node}{{\rm n}}
\newcommand{\edge}{{\rm e}}
\newcommand{\face}{{\rm f}}
\newcommand{\vol}{{\rm v}}

\newcommand{\kd}{k}
\newcommand{\kr}{k-1}

\newcommand\nn{\boldsymbol n}     
\renewcommand\tt{\boldsymbol t} 

\usepackage{pgfplots,pgfplotstable}
\usepackage{tikz-cd}

\graphicspath{{figures/}}

\newcommand{\logLogSlopeTriangle}[5]
{
    \pgfplotsextra
    {
        \pgfkeysgetvalue{/pgfplots/xmin}{\xmin}
        \pgfkeysgetvalue{/pgfplots/xmax}{\xmax}
        \pgfkeysgetvalue{/pgfplots/ymin}{\ymin}
        \pgfkeysgetvalue{/pgfplots/ymax}{\ymax}

        \pgfmathsetmacro{\xArel}{#1}
        \pgfmathsetmacro{\yArel}{#3}
        \pgfmathsetmacro{\xBrel}{#1-#2}
        \pgfmathsetmacro{\yBrel}{\yArel}
        \pgfmathsetmacro{\xCrel}{\xArel}

        \pgfmathsetmacro{\lnxB}{\xmin*(1-(#1-#2))+\xmax*(#1-#2)} 
        \pgfmathsetmacro{\lnxA}{\xmin*(1-#1)+\xmax*#1} 
        \pgfmathsetmacro{\lnyA}{\ymin*(1-#3)+\ymax*#3} 
        \pgfmathsetmacro{\lnyC}{\lnyA+#4*(\lnxA-\lnxB)}
        \pgfmathsetmacro{\yCrel}{\lnyC-\ymin)/(\ymax-\ymin)}

        \coordinate (A) at (rel axis cs:\xArel,\yArel);
        \coordinate (B) at (rel axis cs:\xBrel,\yBrel);
        \coordinate (C) at (rel axis cs:\xCrel,\yCrel);

        \draw[#5]   (A)-- node[pos=0.5,anchor=north] {\scriptsize{1}}
                    (B)-- 
                    (C)-- node[pos=0.,anchor=west] {\scriptsize{#4}} 
                    cycle;
    }
}

\begin{document}

\title{Arbitrary-order pressure-robust DDR and VEM methods for the Stokes problem on polyhedral meshes}

\author[1]{Louren\c{c}o Beir\~{a}o da Veiga}
\author[1]{Franco Dassi}
\author[2]{Daniele A. Di Pietro}
\author[3]{J\'er\^ome Droniou}
\affil[1]{Dipartimento di Matematica e Applicazioni, Universit\`{a} di Milano Bicocca, Italy, \email{lourenco.beirao@unimib.it}, \email{franco.dassi@unimib.it}}
\affil[2]{IMAG, Univ Montpellier, CNRS, Montpellier, France, \email{daniele.di-pietro@umontpellier.fr}}
\affil[3]{School of Mathematics, Monash University, Melbourne, Australia, \email{jerome.droniou@monash.edu}}

\maketitle

\begin{abstract}
  This paper contains two major contributions.
  First we derive, following the discrete de Rham (DDR) and Virtual Element (VEM) paradigms, pressure-robust methods for the Stokes equations that support arbitrary orders and polyhedral meshes.
  Unlike other methods presented in the literature, pressure-robustness is achieved here without resorting to an $\bvec{H}(\DIV)$-conforming construction on a submesh, but rather projecting the volumetric force onto the discrete $\bvec{H}(\CURL)$ space.
  The cancellation of the pressure error contribution stems from key commutation properties of the underlying DDR and VEM complexes.
  The pressure-robust error estimates in $h^{k+1}$ (with $h$ denoting the meshsize and $k\ge 0$ the polynomial degree of the DDR or VEM complex) are proven theoretically and supported by a panel of three-dimensional numerical tests.
  The second major contribution of the paper is an in-depth study of the relations between the DDR and VEM approaches.
  We show, in particular, that a complex developed following one paradigm admits a reformulation in the other, and that couples of related DDR and VEM complexes satisfy commuting diagram properties with the degrees of freedom maps.\medskip
\end{abstract}



\section{Introduction}

Denote by $\Omega\subset\Real^3$ an open connected polyhedral domain.
For the sake of simplicity, we assume that $\Omega$ has trivial topology, i.e., there is no tunnel crossing it and it does not enclose any void.
Given a volumetric force $\bvec{f}:\Omega\to\Real^3$, the Stokes problem for a homogeneous Newtonian fluid with unit viscosity reads:
\begin{equation}\label{eq:strong}
  \left\{~
  \begin{aligned}
    &\text{Find the velocity $\bvec{u}:\Omega\to\Real^3$ and the pressure $p:\Omega\to\Real$ such that}
    \\
    &\CURL(\CURL\bvec{u}) + \GRAD p = \bvec{f}\quad\text{in $\Omega$},
    \\
    &\DIV\bvec{u} = 0 \quad\text{in $\Omega$},
    \\
    &\CURL\bvec{u}\times\normal = \bvec{0} \quad\text{on $\partial\Omega$},
    \\
    &\bvec{u}\cdot\normal = 0 \quad\text{on $\partial\Omega$},
    \\
    &\int_\Omega p = 0.
  \end{aligned}
  \right.
\end{equation}
Notice that, in the momentum balance equation, we have used the vector calculus identity $-\LAPL\bvec{u} = \CURL(\CURL\bvec{u}) - \GRAD(\DIV\bvec{u})$ along with the fact that $\DIV\bvec{u}=0$ to reformulate the viscous term as the curl of the vorticity.
The trivial topology assumption is made to simplify the exposition: for domains crossed by tunnels (i.e., for which the first Betti number is non-zero), one additionally has to enforce the orthogonality of the velocity to 1-harmonic forms for well-posedness.
We do not delve further into this topic here and refer to, e.g., \cite[Chapter 4]{Arnold:18} for additional details.
We also consider homogeneous natural boundary conditions only for the sake of simplicity: the extension to non-homogeneous and/or essential boundary conditions is possible (see, e.g., the discussion in \cite{Bonelle.Ern:15} concerning the formulation corresponding to Eq.~(1) therein).
We also notice that the extension of the method to more standard boundary conditions is possible by modifying the space for the velocity (e.g., for wall boundary conditions, it suffices to consider the subspace of velocities with vanishing tangential components on boundary edges and faces -- these components being naturally available in our discrete spaces).
The details are postponed to a future work.
We finally point out that the two-dimensional case can be recovered as described in \cite[Remark 10]{Castanon-Quiroz.Di-Pietro:20}, and the resulting scheme has analogous robustness properties as the ones discussed below for the three-dimensional case.

We are interested in the weak formulation of problem \eqref{eq:strong} described hereafter.
Assume $\bvec{f}\in\bvec{L}^2(\Omega)$ and denote by $H^1(\Omega)$ and $\Hcurl{\Omega}$ the spaces of functions that are square-integrable over $\Omega$ along with their gradient and curl, respectively.
Additionally letting $L^2_0(\Omega)\coloneq\left\{q\in L^2(\Omega)\st\int_\Omega q = 0\right\}$, the weak formulation of problem \eqref{eq:strong} reads:
\begin{equation}\label{eq:variational}
  \left\{~
  \begin{aligned}
    & \text{Find $\uu \in \Hcurl{\Omega}$ and  $p \in H^1(\Omega)\cap L^2_0(\Omega)$ such that} \\
    & \int_\Omega \CURL \uu \cdot \CURL \vv + \int_\Omega \GRAD p\cdot\vv = \int_\Omega \ff \cdot \vv
    \quad \forall \vv \in \Hcurl{\Omega} \\
    & \int_\Omega \GRAD q\cdot \uu  = 0 \quad \forall q \in H^1(\Omega)\cap L^2_0(\Omega).
  \end{aligned}
  \right.
\end{equation}
Problem \eqref{eq:variational} admits a unique solution which, if regular enough, satisfies \eqref{eq:strong} almost everywhere.
It is a simple matter to check that changing the irrotational component of the body force $\bvec{f}$ only affects the pressure $p$, leaving the velocity $\bvec{u}$ unaltered.
When considering numerical approximations, the failure to reproduce this property at the discrete level can have a sizeable impact on the quality of the numerical solution \cite{Linke:14,john-linke-merdon-neilan-rebholz:2017}.
This can happen, e.g., when the Coriolis force is taken into account (in two dimensions, this force is always irrotational).
A numerical example where large unphysical oscillations occur due to the lack of pressure-robustness is provided in \cite{Castanon-Quiroz.Di-Pietro:20}; see, in particular, Fig.~4 therein.
Numerical schemes that behave robustly with respect to the magnitude of the irrotational part of the body force are often referred to as \emph{pressure-robust}.
From the analysis standpoint, such methods guarantee velocity error estimates that are independent of the pressure. 

The issue of pressure-robustness for finite element discretizations on standard (conforming) meshes has been addressed in several works.
A two-dimensional finite element pair on standard triangular meshes which is conforming, inf-sup stable, and delivers $\Hdiv{\Omega}$-conforming approximations of the velocity has been developed in \cite{Falk.Neilan:13} using as a starting point the Stokes complex;  see also \cite{Zhang:16} for an extension to quadrilateral elements.
$\Hdiv{\Omega}$-conforming velocity approximations naturally lead to pressure-robustness.
A related strategy to recover this property for a variety of numerical schemes is outlined in \cite{Linke.Merdon:16} (see also \cite{Wang.Mu.ea:21}), where the authors suggest a modification of the right-hand side involving the projection of the test function onto an $\Hdiv{\Omega}$-conforming space.
This strategy has been applied to the design of pressure-robust Hybrid High-Order (HHO) discretizations of the Stokes problem on conforming simplicial meshes in \cite{Di-Pietro.Ern.ea:16*1}; see also \cite[Section 8.6]{Di-Pietro.Droniou:20} and \cite{Castanon-Quiroz.Di-Pietro:20} (along with the precursor works \cite{Di-Pietro.Krell:18,Botti.Di-Pietro.ea:19*1}) concerning the extension to the full Navier--Stokes equations.
In \cite{Beirao-da-Veiga.Lovadina.ea:17,BLV:2018,BDV:3D:ST}, the authors proposed a family of Virtual Element schemes for general polytopal meshes such that the virtual velocity is divergence-free and enjoys error bounds that do not depend directly on the pressure; although this represents an improvement with respect to standard inf-sup stable methods, the scheme is not fully pressure-robust since the velocity error depends indirectly on the pressure through a higher order loading term. We can designate such schemes as \emph{asymptotically pressure robust}, meaning that the terms involving the pressure in the right-hand side of the error estimates are of higher-order than the dominating error component.
 
Adapting the strategy of \cite{Linke.Merdon:16} to general polytopal meshes can be problematic owing to the difficulty of devising discrete spaces that are both \emph{$\Hdiv{\Omega}$-conforming} and \emph{fully computable}.
$\Hdiv{\Omega}$-conforming virtual spaces, e.g., fulfill the first requirement but not the second; as a result, when used in the design of numerical schemes, they only lead to \emph{asymptotic pressure robustness}.
One possibility then consists in constructing $\Hdiv{\Omega}$-conforming spaces based on a matching and shape regular simplicial submesh, as recently proposed in \cite{Frerichs.Merdon:20}; see also \cite{Castanon-Quiroz.Di-Pietro:22} concerning the application of a similar strategy to HHO methods.
While this approach leads to fully pressure-robust methods, it hinges on a construction that can be computationally expensive, particularly in dimension 3 and/or in the presence of faces and edges that are orders of magnitude smaller than the parent element.

In this work we explore a different strategy based on a \emph{compatible} approach, that is, we replace the spaces that appear in the weak formulation \eqref{eq:variational} with finite-dimensional counterparts that form an exact complex when connected by (discrete counterparts of) the usual vector calculus operators.
Pressure-robustness is then obtained projecting the body force $\bvec{f}$ onto the discrete $\Hcurl{\Omega}$ space and leveraging a commutativity property involving the interpolators on the discrete counterparts of $\Hcurl{\Omega}$ and $H^1(\Omega)$ and the (discrete analog of) the gradient operator.
A similar strategy has been considered in \cite{Bonelle.Ern:15} in the context of Compatible Discrete Operators, leading to a pressure-robust, low-order method on general polytopal meshes.
We also mention here \cite[Remark 23]{Chave.Di-Pietro.ea:20} on a related approach for HHO methods on standard meshes.
Two different design paradigms are considered: the discrete de Rham (DDR) approach of \cite{Di-Pietro.Droniou:21*1} (see also \cite{Di-Pietro.Droniou.ea:20,Di-Pietro.Droniou:21}), where both the spaces and differential operators are replaced by discrete analogs, and the Virtual Element Method (VEM) of \cite{2Dmagneto,Beirao-da-Veiga.Brezzi.ea:18*2} (see also \cite{Beirao-da-Veiga.Brezzi.ea:16,Beirao-da-Veiga.Brezzi.ea:18*1}), where compatible and conforming (but not fully computable) spaces are exploited to design a numerical scheme through computable projections.
In both cases, we obtain fully pressure-robust schemes that, when complexes of degree $k\ge 0$ are used as starting points, converge as $h^{k+1}$ (with $h$ denoting, as usual, the meshsize) in the graph norm.
The key feature of both schemes is that they achieve pressure-robustness on general polytopal meshes \emph{without} resorting to a matching simplicial submesh.

This work also contains a second important contribution, namely the construction of bridges between the DDR and VEM approaches.
Specifically, we recast the spaces and local constructions of each paradigm into the other, thus enabling an in-depth comparison.
On one hand, this shows differences in the choice (and polynomial degree) of certain degrees of freedom;
on the other hand, it reveals that the reduction of the number of unknowns is obtained through different strategies in the two methods (serendipity for VEM, a systematic use of enhancement for DDR).
The links established in the present work can serve as a starting point for cross-fertilization of these approaches.

The rest of the paper is organized as follows.
In Section \ref{sec:setting} we establish the discrete setting.
In Sections \ref{sec:ddr} and \ref{sec:discre1} we state, respectively, the DDR and VEM schemes along with the corresponding pressure-robust error estimates.
A numerical study of the methods is performed in Section \ref{sec:numerical.examples}, where we also verify in practice the pressure-robustness property.
Bridges between the DDR and VEM schemes are built in Section \ref{sec:bridge}.
Finally, Section \ref{sec:theoretical} contains the proofs of the main results.


\section{Setting}\label{sec:setting}

\subsection{Mesh and orientation of mesh entities}

For any measurable set $Y\subset\Real^3$, we denote by $h_Y\coloneq\sup\{|\bvec{x}-\bvec{y}|\st \bvec{x},\bvec{y}\in Y\}$ its diameter and by $|Y|$ its Hausdorff measure.
We consider meshes $\Mh\coloneq\Th\cup\Fh\cup\Eh\cup\Vh$ of the domain $\Omega$, where:
$\Th$ is a finite collection of open disjoint polyhedral elements such that $\overline{\Omega} = \bigcup_{T\in\Th}\overline{T}$ and $h=\max_{T\in\Th}h_T>0$;
$\Fh$ is a finite collection of open planar faces;
$\Eh$ is the set collecting the open polygonal edges (line segments) of the faces;
$\Vh$ is the set collecting the edge endpoints.
It is assumed, in what follows, that $(\Th,\Fh)$ matches the conditions in \cite[Definition 1.4]{Di-Pietro.Droniou:20}.
We additionally assume that the polytopes in $\Th\cup\Fh$ are simply connected and have connected Lipschitz-continuous boundaries.
The set collecting the mesh faces that lie on the boundary of a mesh element $T\in\Th$ is denoted by $\FT$.
For any mesh element or face $Y\in\Th\cup\Fh$, we denote, respectively, by $\edges{Y}$ and $\vertices{Y}$ the set of edges and vertices of $Y$.

For any face $F\in\Fh$, an orientation is set by prescribing a unit normal vector $\normal_F$ and, for any mesh element $T\in\Th$ sharing $F$, we denote by $\omega_{TF}\in\{-1,1\}$ the orientation of $F$ relative to $T$ such that $\omega_{TF}\normal_F$ points out of $T$.
For any edge $E\in\Eh$, an orientation is set by prescribing the unit tangent vector $\tangent_E$.
Denote by $F\in\Fh$ a face such that $E\in\EF$ and let $\normal_{FE}$ be the unit vector normal to $E$ lying in the plane of $F$ such that $(\tangent_E,\normal_{FE})$ forms a system of right-handed coordinates.
We let $\omega_{FE}\in\{-1,1\}$ be the orientation of $E$ relative to $F$ such that $\omega_{FE}\normal_{FE}$ points out of $F$.

\subsection{Differential operators on faces and tangential trace}

For any mesh face $F\in\Fh$, we denote by $\GRAD_F$ and $\DIV_F$ the tangent gradient and divergence operators acting on smooth enough functions over $F$.
Moreover, for any $r:F\to\Real$ and $\bvec{z}:F\to\Real^2$ smooth enough, we define the two-dimensional vector and scalar curl operators such that
\[
  \VROT_F r\coloneq \rotation{-\nicefrac\pi2}(\GRAD_F r)\quad\mbox{ and }\quad \ROT_F\bvec{z}=\DIV_F(\rotation{-\nicefrac\pi2}\bvec{z}),
\]
where $\rotation{-\nicefrac\pi2}$ is the rotation of angle $-\frac\pi2$ in the oriented tangent space to $F$.
When considering the face $F\in\Fh$ as immersed in $\Real^3$, both $\GRAD_F$ and $\VROT_F$ act on the restrictions to $F$ of scalar-valued functions of the three-dimensional space coordinate.
Similarly, $\DIV_F$ and $\ROT_F$ act on the tangential trace on $F$ of vector-valued functions of the three-dimensional space coordinate.
The tangential trace is hereafter denoted appending the index ``${\rm t},\! F$'' to the name of the function so that, e.g., given $\bvec{v}:\Omega\to\Real^3$ smooth enough and $F\in\Fh$, $\bvec{v}_{{\rm t},F}\coloneq\normal_F\times(\bvec{v}_{|F}\times\normal_F)$.

\subsection{Lebesgue and Hilbert spaces}

For $Y$ measured subset of $\Real^3$, we denote by $L^2(Y)$ the Lebesgue space spanned by functions that are square-integrable over $Y$.
When $Y$ is an $n$-dimensional set (typically a mesh element or face), we will use the boldface notation $\bvec{L}^2(Y)\coloneq L^2(Y)^n$ for the space of vector-valued fields over $Y$ with square-integrable components.
Given $s>0$ and $Y\in\{\Omega\}\cup\Th\cup\Fh$, $H^s(Y)$ will denote the usual Hilbert space of index $s$ on $Y$, and we additionally let $\bvec{H}^s(Y)\coloneq H^s(Y)^n$ and $\bvec{C}^s(\overline{Y})\coloneq C^s(\overline{Y})^n$ with $C^s(\overline{Y})$ spanned by functions that are continuous, along with their derivatives up to order $s$, up to the boundary of $Y$.

For all $Y\in\{\Omega\}\cup\Th$, $\Hcurl{Y}$ and $\Hdiv{Y}$ denote the spaces of vector-valued functions that are square integrable along with their curl or divergence, respectively.
We additionally let, for any $s>0$, $\Hscurl{s}{Y} \coloneq\{\vv\in\bvec{H}^s(Y)\st{\CURL\vv\in\bvec{H}^s(Y)}\}$.
We notice, in passing, that trace theorems and integration by parts formulas for $\Hcurl{Y}$ and $\Hdiv{Y}$ in polyhedral domains involve subtleties that are out of the scope of the present exposition (and not directly useful to us as we will only require traces of functions that are smooth enough); we refer to \cite{Buffa.Ciarlet:01,Buffa.Costabel.ea:02} and references therein for a rigorous study of this subject.

The regularity requirements in the error estimates will be expressed in terms of the broken Hilbert spaces $H^s(\Th)\coloneq\left\{q\in L^2(\Omega)\st{q_{|T}\in H^s(T)\quad\forall T\in\Th}\right\}$.
According to the previously established conventions, the corresponding vector-valued version is denoted by $\bvec{H}^s(\Th)$, and we additionally let $\Hscurl{s}{\Th} \coloneq\left\{\vv\in\bvec{L}^2(\Omega)\st{\vv\in\Hscurl{s}{T}\quad\forall T\in\Th}\right\}$.

\subsection{Polynomial spaces and decompositions}\label{sec:polynomial.spaces}

For a given integer $l\ge 0$, $\mathbb{P}_n^l$ denotes the space of $n$-variate polynomials of total degree $\le l$, with the convention that $\mathbb{P}_n^{-1} \coloneq \{ 0 \}$ for any $n$.
For any $Y\in\Th\cup\Fh\cup\Eh$, we denote by $\Poly{l}(Y)$ the space spanned by the restriction to $Y$ of the functions in $\mathbb{P}_3^l$.
Denoting by $1\le n\le 3$ the dimension of $Y$, $\Poly{l}(Y)$ is isomorphic to $\mathbb{P}_n^l$ (see \cite[Proposition 1.23]{Di-Pietro.Droniou:20}).
In what follows, with a small abuse of notation, both spaces are denoted by $\Poly{l}(Y)$.
We also denote by
  \begin{equation}\label{defPh}
    \Poly[0]{l}(Y)\coloneq\left\{q\in\Poly{l}(Y)\st\int_Y q = 0\right\}
  \end{equation}
the subspace of $\Poly{l}(Y)$ spanned by functions in $\Poly{l}(Y)$ with zero mean value over $Y$.
For the sake of brevity, we also introduce the boldface notations $\vPoly{l}(T)\coloneq\Poly{l}(T)^3$ for all $T\in\Th$ and $\vPoly{l}(F)\coloneq\Poly{l}(F)^2$ for all $F\in\Fh$.
For $Y$ as above, we additionally denote by $\lproj{l}{Y}$ (resp.\ $\vlproj{l}{Y}$) the $L^2$-orthogonal projector on $\Poly{l}(Y)$ (resp.\ $\vPoly{l}(Y)$).

For all $Y\in\Th\cup\Fh$ and $\bvec{x} \in Y$, we introduce the translated coordinate vector $\bvec{x}_Y = \bvec{x} - \bvec{c}_Y$ where, for each given $Y$, we have fixed a point $\bvec{c}_Y \in Y$ such that $Y$ contains a ball centered at $\bvec{c}_Y$ of radius $\rho h_Y$, with $\rho$ denoting the mesh regularity parameter (see \cite[Definition 1.9]{Di-Pietro.Droniou:20} and also Assumption \ref{ass:star-shaped} in Section \ref{sec:prel} concerning the VEM scheme).
For any mesh face $F\in\Fh$ and any integer $l\ge 0$, we define the following relevant subspaces of $\vPoly{l}(F)$:
\begin{equation}\label{eq:spaces.F}
  \begin{alignedat}{2}
    \Goly{l}(F)&\coloneq\GRAD_F\Poly{l+1}(F),
    &\qquad
    \cGoly{l}(F)&\coloneq\bvec{x}_F^\perp\,\Poly{l-1}(F),
    \\ 
    \Roly{l}(F)&\coloneq\VROT_F\Poly{l+1}(F),
    &\qquad
    \cRoly{l}(F)&\coloneq\bvec{x}_F\,\Poly{l-1}(F),
  \end{alignedat}
\end{equation}
where $\bvec{y}^\perp$ is a shorthand for the vector $\bvec{y}$ rotated by $-\frac{\pi}2$ in $F$, so that
\[
\vPoly{l}(F)
= \Goly{l}(F) \oplus \cGoly{l}(F)
= \Roly{l}(F) \oplus \cRoly{l}(F).
\]
The $L^2$-orthogonal projectors on the spaces \eqref{eq:spaces.F} are, with obvious notation, $\Gproj{l}{F}$, $\Gcproj{l}{F}$, $\Rproj{l}{F}$, and $\Rcproj{l}{F}$.
Similarly, for any mesh element $T\in\Th$ and any integer $l\ge 0$, we introduce the following subspaces of $\vPoly{l}(T)$:
\begin{equation}\label{eq:spaces.T}
  \begin{alignedat}{2}
    \Goly{l}(T)&\coloneq\GRAD\Poly{l+1}(T),
    &\qquad 
    \cGoly{l}(T)&\coloneq\bvec{x}_T\times \vPoly{l-1}(T),
    \\
    \Roly{l}(T)&\coloneq\CURL\vPoly{l+1}(T),
    &\qquad
    \cRoly{l}(T)&\coloneq\bvec{x}_T\,\Poly{l-1}(T).
  \end{alignedat}
\end{equation}
The $L^2$-orthogonal projectors on the spaces \eqref{eq:spaces.T} are $\Gproj{l}{T}$, $\Gcproj{l}{T}$, $\Rproj{l}{T}$, and $\Rcproj{l}{T}$, and we have
\[
\vPoly{l}(T)
= \Goly{l}(T) \oplus \cGoly{l}(T)
= \Roly{l}(T) \oplus \cRoly{l}(T).
\]


\section{DDR scheme}\label{sec:ddr}

We present a scheme based on the DDR sequence of \cite{Di-Pietro.Droniou:21*1}.
Throughout this section, we let an integer $k\ge 0$ be fixed, corresponding to the polynomial degree of the sequence.

\subsection{Spaces}

We define the following spaces for the velocity and the pressure:
\begin{align*}
  \Xcurl{h}&\coloneq\Big\{
  \begin{aligned}[t]
    \uvec{v}_h
    &=\big(
    (\bvec{v}_{\cvec{R},T},\bvec{v}_{\cvec{R},T}^\compl)_{T\in\Th},(\bvec{v}_{\cvec{R},F},\bvec{v}_{\cvec{R},F}^\compl)_{F\in\Fh}, (v_E)_{E\in\Eh}
    \big)\st
    \\
    &\qquad\text{$\bvec{v}_{\cvec{R},T}\in\Roly{k-1}(T)$ and $\bvec{v}_{\cvec{R},T}^\compl\in\cRoly{k}(T)$ for all $T\in\Th$,}
    \\
    &\qquad\text{$\bvec{v}_{\cvec{R},F}\in\Roly{k-1}(F)$ and $\bvec{v}_{\cvec{R},F}^\compl\in\cRoly{k}(F)$ for all $F\in\Fh$,}
    \\
    &\qquad\text{and $v_E\in\Poly{k}(E)$ for all $E\in\Eh$}\Big\},
  \end{aligned}
  \\ 
  \Xgrad{h}&\coloneq\Big\{
  \begin{aligned}[t]
    \underline{q}_h=\big((q_T)_{T\in\Th},(q_F)_{F\in\Fh},q_{\Eh}\big)\st
    &\text{$q_T\in \Poly{k-1}(T)$ for all $T\in\Th$,}
    \\
    &\text{$q_F\in\Poly{k-1}(F)$ for all $F\in\Fh$,}
    \\
    &\text{and $q_{\Eh}\in\Poly[\rm c]{k+1}(\Eh)$}\Big\},
  \end{aligned}
\end{align*}
where $\Poly[\rm c]{k+1}(\Eh)$ is spanned by functions that are continuous on the edge skeleton of the mesh and whose restriction to every edge $E\in\Eh$ is in $\Poly{k+1}(E)$.
Given $\bullet\in\{\GRAD,\CURL\}$ and a mesh entity $Y$ appearing in the definition of $\underline{X}_{\bullet,h}^k$, we denote by $\underline{X}_{\bullet,Y}^k$ the restriction of this space to $Y$, gathering the polynomial components attached to $Y$ and to the mesh entities on the boundary $\partial Y$ of $Y$.
Similarly, the restriction to $Y$ of an element $\underline{\xi}_h\in\underline{X}_{\bullet,h}^k$ is denoted by $\underline{\xi}_Y$ and is obtained collecting the polynomial components of $\underline{\xi}_h$ attached to $Y$ and to the mesh entities on $\partial Y$.
If $q_{\Eh}\in\Poly[c]{k+1}(\Eh)$ and $F\in\Fh$, we similarly let $q_{\EF}$ be the restriction of $q_{\Eh}$ to $\partial F=\cup_{E\in\EF}\overline{E}$.

The interpolators on the DDR spaces are defined as follows:
$\Icurl{h}: \bvec{C}^0(\overline{\Omega})\to\Xcurl{h}$ is obtained setting, for all $\bvec{v}\in\bvec{C}^0(\overline{\Omega})$,
\[
\Icurl{h}\bvec{v}\coloneq\big(
(\Rproj{k-1}{T}\bvec{v}_{|T},\Rcproj{k}{T}\bvec{v}_{|T})_{T\in\Th},
(\Rproj{k-1}{F}\bvec{v}_{{\rm t},F},\Rcproj{k}{F}\bvec{v}_{{\rm t},F})_{F\in\Fh},
(\lproj{k}{E}(\bvec{v}_{|E}\cdot\tangent_E))_{E\in\Eh}
\big),
\]
where we remind the reader that $\bvec{v}_{{\rm t},F}$ denotes the tangential trace of $\bvec{v}$ over $F$, while $\Igrad{h}:C^0(\overline{\Omega})\to\Xgrad{h}$ is such that, for all $q\in C^0(\overline{\Omega})$,
\[
\begin{gathered}
  \Igrad{h} q \coloneq \big((\lproj{k-1}{T} q_{|T})_{T\in\Th},(\lproj{k-1}{F} q_{|F})_{F\in\Fh},q_{\Eh}\big)\in\Xgrad{h}
  \\
  \text{
    where $\lproj{k-1}{E} q_{\Eh|E}=\lproj{k-1}{E} q_{|E}$ for all $E\in\Eh$
    and $q_{\Eh}(\bvec{c}_\nu)=q(\bvec{c}_\nu)$ for all $\nu\in\Vh$,
  }
\end{gathered}
\]
with $\bvec{c}_\nu$ denoting the coordinate vector of the vertex $\nu\in\Vh$.
\begin{remark}[Domain of {$\Icurl{h}$}]
  The domain of the interpolator on $\Xcurl{h}$ could also be chosen as $\left\{\bvec{v}\in\bvec{H}^s(\Omega)\st\CURL\bvec{v}\in\bvec{L}^p(\Omega)\right\}$ with $s>\frac12$ and $p>2$; see, e.g., \cite{Amrouche.Bernardi.ea:98}.
  An in-depth study of the domain of the N\'ed\'elec interpolator in the context of classical finite elements on standard meshes can be found in \cite[Chapter 16]{Ern.Guermond:21}.
  We notice, in passing, that the regularity $\bvec{u}\in\bvec{H}^2(\Omega)$ in Theorem \ref{thm:ddr:convergence} below is sufficient for $\Icurl{h}\bvec{u}$ to be well-defined, as $\bvec{H}^2(\Omega)$ is embedded into $\bvec{C}^0(\overline{\Omega})$.
\end{remark}

\subsection{Discrete vector calculus operators}

Discrete vector calculus operators are built emulating integration by parts formulas.
We recall here their definitions and refer to \cite{Di-Pietro.Droniou.ea:20,Di-Pietro.Droniou:21*1} for further details.
Following standard DDR notations, full operators that only appear in the discrete complex through projections (i.e., $\cCT$, $\cGF$, and $\cGT$ respectively defined by \eqref{eq:cCT}, \eqref{eq:cGF}, and \eqref{new:X1} below) are denoted in sans serif font.

\subsubsection{Curl}

For all $F\in\Fh$, the \emph{face curl} $\CF:\Xcurl{F}\to\Poly{k}(F)$ is such that, for all $\uvec{v}_F\in\Xcurl{F}$,
\begin{equation} \label{eq:CF}
  \int_F\CF\uvec{v}_F~r_F
  = \int_F\bvec{v}_{\cvec{R},F}\cdot\VROT_F r_F
  - \sum_{E\in\EF}\omega_{FE}\int_E v_Er_F\qquad
  \forall r_F\in\Poly{k}(F).
\end{equation}
The \emph{tangential trace} $\trFt:\Xcurl{F}\to\vPoly{k}(F)$ is such that, for all $\uvec{v}_F\in\Xcurl{F}$ and all $(r_F,\bvec{w}_F)\in\Poly{k+1}(F)\times\cRoly{k}(F)$,
\begin{equation}\label{eq:trFt}
  \int_F\trFt\uvec{v}_F\cdot(\VROT_F r_F + \bvec{w}_F)
  = \int_F\CF\uvec{v}_F~r_F
  + \sum_{E\in\EF}\omega_{FE}\int_E v_E r_F
  + \int_F\bvec{v}_{\cvec{R},F}^\compl\cdot\bvec{w}_F.
\end{equation}
For all $T\in\Th$, the \emph{element curl} $\cCT:\Xcurl{T}\to\vPoly{k}(T)$ is obtained, for all $\uvec{v}_T\in\Xcurl{T}$, by enforcing
\begin{equation} \label{eq:cCT}
  \int_T\cCT\uvec{v}_T\cdot\bvec{w}_T
  = \int_T\bvec{v}_{\cvec{R},T}\cdot\CURL\bvec{w}_T
  + \sum_{F\in\FT}\omega_{TF}\int_F\trFt\uvec{v}_F\cdot(\bvec{w}_T\times\normal_F)
  \qquad\forall\bvec{w}_T\in\vPoly{k}(T).
\end{equation}
The \emph{discrete curl} $\uCh$ maps on the following discrete counterpart of the space $\Hdiv{\Omega}$:
\begin{equation*}
  \Xdiv{h}\coloneq\Big\{
  \begin{aligned}[t]
    \uvec{w}_h
    &=\big((\bvec{w}_{\cvec{G},T},\bvec{w}_{\cvec{G},T}^\compl)_{T\in\Th}, (w_F)_{F\in\Fh}\big)\st
    \\
    &\qquad\text{$\bvec{w}_{\cvec{G},T}\in\Goly{k-1}(T)$ and $\bvec{w}_{\cvec{G},T}^\compl\in\cGoly{k}(T)$ for all $T\in\Th$,}
    \\
    &\qquad\text{and $w_F\in\Poly{k}(F)$ for all $F\in\Fh$}
    \Big\},
  \end{aligned}
\end{equation*}
The polynomial components of $\Xdiv{h}$ can be interpreted according to the interpolator $\Idiv{h}:\bvec{H}^1(\Omega)\to\Xdiv{h}$ such that, for all $\bvec{w}\in\bvec{H}^1(\Omega)$,
\[
\Idiv{h}\bvec{w}\coloneq\big(
(\Gproj{k-1}{T}\bvec{w}_{|T},\Gcproj{k}{T}\bvec{w}_{|T})_{T\in\Th},
(\lproj{k}{F}(\bvec{w}_{|F}\cdot\normal_F))_{F\in\Fh}
\big).
\]
We let $\uCh:\Xcurl{h}\to\Xdiv{h}$ be such that, for all $\uvec{v}_h\in\Xcurl{h}$,
\begin{equation}\label{eq:uCh}
  \uCh\uvec{v}_h\coloneq\big(
  \big( \Gproj{k-1}{T}\big(\cCT\uvec{v}_T\big),\Gcproj{k}{T}\big(\cCT\uvec{v}_T\big) \big)_{T\in\Th},
  ( \CF\uvec{v}_F )_{F\in\Fh}
  \big).
\end{equation}

\subsubsection{Gradient}

For any $E\in\Eh$, the \emph{edge gradient} $\GE:\Xgrad{E}\to\Poly{k}(E)$ is defined as follows:
For all $q_E\in\Xgrad{E}=\Poly{k+1}(E)$,
\begin{equation*}
  \GE q_E \coloneq q_E',
\end{equation*}
where the derivative is taken along $E$ according to the orientation of $\tangent_E$.
For any $F\in\Fh$, the \emph{face gradient} $\cGF:\Xgrad{F}\to\vPoly{k}(F)$ is such that, for all $\underline{q}_F\in\Xgrad{F}$, 
\begin{equation} \label{eq:cGF}
  \int_F\cGF\underline{q}_F\cdot\bvec{w}_F
  = -\int_F q_F\DIV_F\bvec{w}_F
  + \sum_{E\in\EF}\omega_{FE}\int_E q_{\EF}(\bvec{w}_F\cdot\normal_{FE})
  \qquad\forall\bvec{w}_F\in\vPoly{k}(F).
\end{equation}
The \emph{scalar trace} $\trF:\Xgrad{F}\to\Poly{k+1}(F)$ is such that, for all $\underline{q}_F\in\Xgrad{F}$,
\begin{equation}\label{eq:trF}
\int_F\trF\underline{q}_F\DIV_F\bvec{v}_F
= -\int_F\cGF\underline{q}_F\cdot\bvec{v}_F
+ \sum_{E\in\EF}\omega_{FE}\int_E q_{\EF}(\bvec{v}_F\cdot\normal_{FE})
\qquad\forall\bvec{v}_F\in\cRoly{k+2}(F).
\end{equation}
For all $T\in\Th$, the \emph{element gradient} $\cGT:\Xgrad{T}\to\vPoly{k}(T)$ is defined such that, for all $\underline{q}_T\in\Xgrad{T}$,
\begin{equation}\label{new:X1}
  \int_T\cGT\underline{q}_T\cdot\bvec{w}_T
  = -\int_T q_T\DIV\bvec{w}_T
  + \sum_{F\in\FT}\omega_{TF}\int_F\trF\underline{q}_F(\bvec{w}_T\cdot\normal_F)\qquad
  \forall\bvec{w}_T\in\vPoly{k}(T).
\end{equation}
Finally, the \emph{discrete gradient} $\uGh:\Xgrad{h}\to\Xcurl{h}$ is obtained collecting the projections of each local gradient on the space(s) attached to the corresponding mesh entity: For all $\underline{q}_h\in\Xgrad{h}$,
\begin{equation}\label{eq:uGh}
  \uGh\underline{q}_h\coloneq
  \begin{aligned}[t]
    \Big(    
    &\big( \Rproj{k-1}{T}\big(\cGT\underline{q}_T\big),\Rcproj{k}{T}\big(\cGT\underline{q}_T\big) \big)_{T\in\Th},
    \\
    &\big( \Rproj{k-1}{F}\big(\cGF\underline{q}_F\big),\Rcproj{k}{F}\big(\cGF\underline{q}_F\big) \big)_{F\in\Fh},
    \\
    &( \GE q_E )_{E\in\Eh}
    \Big).
    \end{aligned} 
\end{equation}
The following discrete counterpart of the property $\CURL\GRAD = \bvec{0}$ is proved in \cite[Theorem 1]{Di-Pietro.Droniou:21*1}:
\begin{equation}\label{eq:Im.uGh.subset.Ker.uCh}
  \uCh(\uGh\underline{q}_h) = \uvec{0}\qquad\forall\underline{q}_h\in\Xgrad{h}.
\end{equation}

\subsection{Discrete potentials and $L^2$-products}
\label{sec:DDRpots}

We next equip the DDR spaces with discrete $L^2$-products composed of a consistent term (equal to the $L^2$-product of discrete scalar or vector potentials) and a stabilisation term involving least-square penalisations of boundary differences.

Let $T\in\Th$.
The discrete scalar potential $\Pgrad:\Xgrad{T}\to\Poly{k+1}(T)$ is such that, for all $\underline{q}_T\in\Xgrad{T}$,
\begin{equation} \label{eq:PgradT}
\int_T\Pgrad\underline{q}_T\DIV\bvec{v}_T
= -\int_T\cGT\underline{q}_T\cdot\bvec{v}_T
+ \sum_{F\in\FT}\omega_{TF}\int_F\trF\underline{q}_F(\bvec{v}_T\cdot\normal_F)\qquad\forall\bvec{v}_T\in\cRoly{k+2}(T),
\end{equation}
with $\trF$ defined by \eqref{eq:trF}.
The discrete vector potential $\Pcurl:\Xcurl{T}\to\vPoly{k}(T)$ is such that, for all $\uvec{v}_T\in\Xcurl{T}$ and all $(\bvec{w}_T,\bvec{z}_T)\in\cGoly{k+1}(T)\times\cRoly{k}(T)$,
\begin{equation} \label{eq:PcurlT}
\int_T\Pcurl\uvec{v}_T\cdot(\CURL\bvec{w}_T + \bvec{z}_T)
= \int_T\cCT\uvec{v}_T\cdot\bvec{w}_T
- \sum_{F\in\FT}\omega_{TF}\int_F\trFt\uvec{v}_F\cdot(\bvec{w}_T\times\normal_F)
+ \int_T\bvec{v}_{\cvec{R},T}^\compl\cdot\bvec{z}_T.
\end{equation}
Finally, the discrete vector potential $\Pdiv:\Xdiv{T}\to\vPoly{k}(T)$ satisfies, for all $\uvec{w}_T\in\Xdiv{T}$ and all $(r_T,\bvec{z}_T)\in\Poly{k+1}(T)\times\cGoly{k}(T)$,
\begin{equation} \label{eq:PdivT}
\int_T\Pdiv\uvec{w}_T\cdot(\GRAD r_T + \bvec{z}_T)
= -\int_T\DT\uvec{w}_T~r_T + \sum_{F\in\FT}\omega_{TF}\int_Fw_F~r_T
+ \int_T\bvec{w}_{\cvec{G},T}^\compl\cdot\bvec{z}_T,
\end{equation}
with discrete divergence $\DT:\Xdiv{T}\to\Poly{k}(T)$ such that
\begin{equation} \label{eq:DT}
\int_T\DT\uvec{w}_T~q_T = -\int_T\bvec{v}_{\cvec{G},T}\cdot\GRAD q_T
+ \sum_{F\in\FT}\omega_{TF}\int_Fv_F~q_T\qquad\forall q_T\in\Poly{k}(T).
\end{equation}

For $(\bullet,l)\in\left\{(\GRAD,k+1),(\CURL,k),(\DIV,k)\right\}$, the discrete $L^2$-product $(\cdot,\cdot)_{\bullet,h}:\Xbullet{h}\times\Xbullet{h}\to\Real$ is such that, for all $\underline{\xi}_h,\underline{\zeta}_h\in\Xbullet{h}$,
\[
(\underline{\xi}_h,\underline{\zeta}_h)_{\bullet,h}
\coloneq\sum_{T\in\Th}\left[
\int_T\Pbullet\underline{\xi}_T\cdot\Pbullet\underline{\zeta}_T
+ \mathrm{s}_{\bullet,T}(\underline{\xi}_T,\underline{\zeta}_T)
\right],
\]
with local stabilization bilinear forms such that, for all $T\in\Th$:
For all  $(\underline{r}_T,\underline{q}_T)\in\Xgrad{T}\times\Xgrad{T}$,
\[
\begin{aligned}
  \mathrm{s}_{\GRAD,T}(\underline{r}_T,\underline{q}_T)
  &\coloneq
  \sum_{F\in\FT}h_T\int_F(\Pgrad\underline{r}_T-\trF\underline{r}_F)(\Pgrad\underline{q}_T-\trF\underline{q}_F)
  \\
  &\quad + \sum_{E\in\ET} h_T^2\int_E(\Pgrad\underline{r}_T - r_E)(\Pgrad\underline{q}_T - q_E),
\end{aligned}
\]
for all $(\uvec{w}_T,\uvec{v}_T)\in\Xcurl{T}\times\Xcurl{T}$,
\[
\begin{aligned}
  \mathrm{s}_{\CURL,T}(\uvec{w}_T,\uvec{v}_T)
  &\coloneq
  \sum_{F\in\FT}h_T\int_F \big[ (\Pcurl\uvec{w}_T)_{{\rm t},F}-\trFt\uvec{w}_F\big]\cdot\big[ (\Pcurl\uvec{v}_T)_{{\rm t},F}-\trFt\uvec{v}_F\big]
  \\
  &\quad + \sum_{E\in\ET}h_T^2\int_E\big(\Pcurl\uvec{w}_T\cdot\tangent_E-w_E\big)\big(\Pcurl\uvec{v}_T\cdot\tangent_E-v_E\big)
\end{aligned}
\]
(recall that the subscript ``${\rm t},F$'' denotes the tangential trace on $F$) and, for all $(\uvec{w}_T,\uvec{v}_T)\in\Xdiv{T}\times\Xdiv{T}$,
\[
\mathrm{s}_{\DIV,T}(\uvec{w}_T,\uvec{v}_T)
\coloneq
\sum_{F\in\FT}h_T\int_F\big(\Pdiv\uvec{w}_T\cdot\normal_F-w_F\big)\big(\Pdiv\uvec{v}_T\cdot\normal_F-v_F\big).
\]

\begin{remark}[Stabilisation]
  Other choices are possible for the DDR stabilisation bilinear forms, the main requirements being: (a) the positive-definiteness (coercivity) of the discrete $L^2$-product and (b) polynomial consistency, namely the fact that the stabilisation vanishes whenever one of its arguments is the interpolant of a polynomial of the appropriate degree ($k+1$ for $\Xgrad{h}$, $k$ for $\Xcurl{h}$ and $\Xdiv{h}$).
  It is not very difficult to derive abstract assumptions on the stabilisation along the lines of what has been done for HHO (see, e.g., \cite[Assumption 2.4]{Di-Pietro.Droniou:20}); we leave that as an exercise to the reader.
\end{remark}

\subsection{Discrete problem and convergence}\label{sec:ddr.problem}

Define the following subspace of $\Xgrad{h}$ incorporating the zero-mean value condition:
\[
  \XgradO{h}\coloneq\left\{
  \underline{q}_h\in\Xgrad{h}\st (\underline{q}_h,\Igrad{h} 1)_{\GRAD,h} = 0
  \right\}.
\]
Assuming the additional regularity $\bvec{f}\in\bvec{C}^0(\overline{\Omega})$, the DDR scheme reads:
\begin{equation}\label{eq:discrete}
  \left\{~
  \begin{aligned}
    &\text{Find $\uvec{u}_h\in\Xcurl{h}$ and $\underline{p}_h\in\XgradO{h}$ such that} \\
    &\begin{array}{r@{{}={}}ll}
    \mathrm{a}_h(\uvec{u}_h,\uvec{v}_h) + \mathrm{b}_h(\underline{p}_h,\uvec{v}_h)
    & \ell_h(\bvec{f},\uvec{v}_h)
    &\quad\forall\uvec{v}_h\in\Xcurl{h},
    \\
    -\mathrm{b}_h(\underline{q}_h,\uvec{u}_h)
    & 0
    &\quad\forall\underline{q}_h\in\XgradO{h},
    \end{array}
  \end{aligned}
  \right.
\end{equation}
where the bilinear forms $\mathrm{a}_h:\Xcurl{h}\times\Xcurl{h}\to\Real$,
$\mathrm{b}_h:\Xgrad{h}\times\Xcurl{h}\to\Real$,
and $\ell_h:\bvec{C}^0(\overline{\Omega})\times\Xcurl{h}\to\Real$ are such that,
for all $\uvec{v}_h,\uvec{w}_h\in\Xcurl{h}$, all $\underline{q}_h\in\Xgrad{h}$, and all $\bvec{g}\in\bvec{C}^0(\overline{\Omega})$,
\begin{equation}\label{eq:ah.bh.lh}
  \begin{gathered}
    \mathrm{a}_h(\uvec{w}_h,\uvec{v}_h)\coloneq(\uCh\uvec{w}_h,\uCh\uvec{v}_h)_{\DIV,h},\qquad
    \mathrm{b}_h(\underline{q}_h,\uvec{v}_h)\coloneq(\uGh\underline{q}_h,\uvec{v}_h)_{\CURL,h},
    \\
    \ell_h(\bvec{g},\uvec{v}_h)\coloneq(\Icurl{h}\bvec{g},\uvec{v}_h)_{\CURL,h}.
  \end{gathered}
\end{equation}

\begin{remark}[Discretisation of the volumetric force term]
  The following commutation property is an immediate consequence of the corresponding local version proved in \cite[Lemma 4]{Di-Pietro.Droniou:21*1}:
  \begin{equation*}
    \uGh(\Igrad{h} q) = \Icurl{h}(\GRAD q)
    \qquad\forall q\in C^1(\overline{\Omega}).
  \end{equation*}
  It follows from this relation that, for all $(\psi,\uvec{v}_h)\in C^1(\overline{\Omega})\times\Xcurl{h}$,
  \begin{equation}\label{eq:rhs:irrotational.source}
    \ell_h(\GRAD\psi,\uvec{v}_h)
    = \mathrm{b}_h(\Igrad{h}\psi,\uvec{v}_h).
  \end{equation}
\end{remark}

Since we are interested in $h$-convergence, we assume in what follows that $\Mh$ belongs to a mesh sequence that is regular in the sense of \cite[Definition 1.9]{Di-Pietro.Droniou:20}.
For $\bullet\in\{\GRAD,\CURL,\DIV\}$, we denote by $\norm[\bullet,h]{{\cdot}}$ the norm induced by the inner product $(\cdot,\cdot)_{\bullet,h}$ on the space $\Xbullet{h}$, and we additionally set, for all $(\uvec{v}_h,\underline{q}_h)\in\Xcurl{h}\times\Xgrad{h}$,
\begin{equation}\label{eq:tnorm.h}
\tnorm[h]{(\uvec{v}_h,\underline{q}_h)}
\coloneq\left(
\tnorm[\CURL,h]{\uvec{v}_h}^2
+ \tnorm[\GRAD,h]{\underline{q}_h}^2
\right)^{\frac12},
\end{equation}
with graph norms on $\Xcurl{h}$ and $\Xgrad{h}$ given by, respectively,
\begin{equation}\label{eq:tnorm.CURL.GRAD.h}
  \begin{gathered}
    \tnorm[\CURL,h]{\uvec{v}_h}\coloneq\left(
    \norm[\CURL,h]{\uvec{v}_h}^2
    + \norm[\DIV,h]{\uCh\uvec{v}_h}^2
    \right)^{\frac12},\quad  
    \tnorm[\GRAD,h]{\underline{q}_h}\coloneq\left(
    \norm[\GRAD,h]{\underline{q}_h}^2
    + \norm[\CURL,h]{\uGh\underline{q}_h}^2
    \right)^{\frac12}.
  \end{gathered}
\end{equation}
In Theorem \ref{thm:ddr:convergence} below, we compare the solution of the discrete problem \eqref{eq:discrete} with the interpolate of the solution to the continuous problem \eqref{eq:variational}.
For all $T\in\Th$, in order to account for the additional regularity required by the interpolator on $\Xcurl{T}$, we set:
For all $1\le s\le k+1$ and all $\bvec{v}\in\bvec{H}^{\max(s,2)}(T)$,
\begin{equation*}
  \seminorm[\bvec{H}^{(s,2)}(T)]{\bvec{v}}\coloneq
  \begin{cases}
    \seminorm[\bvec{H}^{1}(T)]{\bvec{v}} + h_T\seminorm[\bvec{H}^{2}(T)]{\bvec{v}} & \text{if $s=1$},
    \\
    \seminorm[\bvec{H}^{s}(T)]{\bvec{v}} & \text{if $s\ge 2$}.
  \end{cases}
\end{equation*}
Correspondingly we set, for all $\bvec{v}\in\bvec{H}^{\max(s,2)}(\Th)$,
$
\seminorm[\bvec{H}^{(s,2)}(\Th)]{\bvec{v}}\coloneq\Big(
\sum_{T\in\Th}\seminorm[\bvec{H}^{(s,2)}(T)]{\bvec{v}}^2
\Big)^{\frac12}.
$
Throughout the rest of the paper, we write $a\lesssim b$ in place of $a\le Cb$ with $C$ depending only on, and possibly not all of them, the domain $\Omega$, the polynomial degree $k$ and the mesh regularity parameters (see \cite[Definition 1.9]{Di-Pietro.Droniou:20} for the DDR method and Assumption \ref{ass:star-shaped} in Section \ref{sec:theo:VEM} for the VEM).

\begin{theorem}[Error estimate for the DDR scheme \eqref{eq:discrete}]\label{thm:ddr:convergence}
  Denote by $\bvec{u}\in\Hcurl{\Omega}\cap\Hdiv{\Omega}$ and $p\in H^1(\Omega)\cap L^2_0(\Omega)$, respectively, the velocity and pressure fields solution of the weak formulation \eqref{eq:variational}, and by $\uvec{u}_h\in\Xcurl{h}$ and $\underline{p}_h\in\XgradO{h}$ the corresponding discrete counterparts solving the DDR scheme \eqref{eq:discrete}.
  Let $1\le s \le k+1$ and assume the additional regularity
  $\bvec{u}\in\bvec{H}^2(\Omega)\cap\bvec{H}^s(\Th)$,
  $\CURL\bvec{u}\in\bvec{H}^{s+1}(\Th)$,
  $\CURL\CURL\bvec{u}\in\bvec{C}^0(\overline{\Omega})\cap\bvec{H}^{\max(s,2)}(\Th)$,
  and $p\in C^1(\overline{\Omega})$.
  Then, it holds
  \begin{multline}\label{eq:err.est}
    \tnorm[h]{(\uvec{u}_h - \Icurl{h}\bvec{u}, \underline{p}_h - \Igrad{h} p)}
    \\
    \lesssim h^s\left(
    \seminorm[\bvec{H}^{(s,2)}(\Th)]{\bvec{u}}
    + \seminorm[\bvec{H}^{s}(\Th)]{\CURL\bvec{u}}
    + \seminorm[\bvec{H}^{s+1}(\Th)]{\CURL\bvec{u}}
    + \seminorm[\bvec{H}^{(s,2)}(\Th)]{\CURL\CURL\bvec{u}}
    \right).
  \end{multline}
\end{theorem}

\begin{proof}
  See Section \ref{sec:theo-DDR:convergence}.
\end{proof}

\begin{remark}[Pressure robustness]\label{rem:ddr:pressure.robustness}
  It can be easily checked that the substitution $\bvec{f}\gets\bvec{f}+\GRAD\psi$ with $\psi\in H^1(\Omega)$ in \eqref{eq:strong} results in $(\bvec{u},p)\gets(\bvec{u},p+\psi)$, showing that the velocity field is not affected by the irrotational component of the source term at the continuous level.
  By \eqref{eq:rhs:irrotational.source}, a similar property holds at the discrete level: the substitution $\bvec{f}\gets\bvec{f}+\GRAD\psi$ with $\psi\in C^1(\overline{\Omega})$ (the additional regularity being required by the presence of the interpolator in the definition \eqref{eq:ah.bh.lh} of $\ell_h$) results in $(\uvec{u}_h,\underline{p}_h)\gets(\bvec{u}_h,\underline{p}_h+\Igrad{h}(\psi+C))$ where $C$ is a constant which ensures that $(\Igrad{h}(\psi+C),\Igrad{h}1)_{\GRAD,h}=0$.
  This property has the important consequence that the right-hand side of the error estimate \eqref{eq:err.est} is independent of the pressure, and is therefore not affected by the substitution $\bvec{f}\gets\bvec{f}+\GRAD\psi$, showing that the DDR scheme \eqref{eq:discrete} is \emph{pressure robust} \cite{Linke.Merdon:16}.
  This property is obtained here on general meshes, for arbitrary polynomial degrees, and without resorting to submeshing.
\end{remark}

\section{VEM scheme}
\label{sec:discre1}

The second scheme we present in this paper is based on a set of Virtual Element spaces that form an exact sequence. Since the spaces are a simple modification of those presented in 
\cite{Beirao-da-Veiga.Brezzi.ea:18*2}, we will provide only a brief description and refer to the above paper for a deeper overview.
Throughout the rest of this section, the integer $k\ge 0$ will denote the polynomial degree of the sequence.

\subsection{Local spaces on faces}\label{sec:local}

We first introduce the local \emph{edge} and \emph{nodal} spaces on faces, minimal modifications of those introduced in \cite{2Dmagneto,Beirao-da-Veiga.Brezzi.ea:18*2}. These spaces can be seen as a generalization to polygons of {N\'ed\'elec elements of the first kind}. Let $F\in\Fh$ denote a mesh face.
In order to describe directly the serendipity version of the spaces (which, when compared to the standard version, requires a more complex construction but is more computationally efficient), we let
\begin{equation}\label{defbeta}
\beta_F \coloneq k+1-\eta_F \, ,
\end{equation}
where $\eta_F$ is an integer, equal to or smaller than the number of straight lines necessary to cover the boundary of $F$. A safe (but possibly not optimal) choice is $\eta_F=3$, yielding $\beta_F=k-2$. Higher values of $\beta_F$ will lead to a more efficient scheme but may not always be feasible depending on the face geometry. 
In what follows, we always assume $\eta_F\ge 1$ so that $\beta_F\le k$; note that in the case $\beta_F=k$, which represents the plain non-serendipity VEM, the conditions in the spaces \eqref{defVeS}, \eqref{defVnS} vanish and one does not need to define the serendipity projectors \eqref{newnewL}, \eqref{PiS-nod2}.
In the case $\beta_F \ge 0$ we assume that faces are convex (such condition is not needed if $\beta_F < 0$). This convexity condition simplifies the development of serendipity spaces; for the treatment of non-convex faces also in the case $\beta_F \ge 0$ we refer to \cite{Serendipity-2}. An alternative option would be instead to use an ``enhancement'', in the spirit of \cite{Projectors}, that is a simpler approach but leads to a less significant reduction in the number of degrees of freedom.

\subsubsection{Edge space on faces}

We start by defining a projector ${\mathbf \Pi}_{S,F}^{\edge}: {\cal S}^{\edge} \rightarrow\vPoly{k}(F)$, where ${\cal S}^{\edge}$ is a set of sufficiently regular functions $F\to\Real^2$, as follows:
For all $\vv \in {\cal S}^{\edge}$,
\begin{subequations}
\label{newnewL}
  \begin{align}
    &\mbox{$\int_{\partial F}[(\vv - {\mathbf \Pi}_{S,F}^{\edge} \vv)\cdot\tt_{\partial F}] [\GRAD_F p\cdot\tt_{\partial F}]  \ds= 0 \quad \forall { p} \in \Poly{k+1}(F)$},\label{defPieS1}\\
    &\mbox{$\int_{\partial F}(\vv - {\mathbf \Pi}_{S,F}^{\edge} \vv)\cdot \tt_{\partial F} \ds=0$},\label{defPieS2}\\
    &\mbox{if $k>1$: \ $\int_F \ROT_F(\vv-{\mathbf \Pi}^{\edge}_{S,F}\vv)p \df=0\quad \forall p \in \Poly[0]{\kr}(F)$} ,\label{defPieS3}
    \\
    &\mbox{if $\beta_F\ge 0$: \  $\int_{F}(\vv-{\mathbf \Pi}^{\edge}_{S,F} \vv)\cdot\xxf \: p\df =0 \quad \forall p \in\Poly{\beta_F}(F)$} ,\label{defPieS4}
  \end{align}
\end{subequations}
where $(\tt_{\partial F})_{|E} \coloneq \omega_{FE}\tt_E$ for all $E \in \EF$. Note that, if $\eta_F$ is chosen smaller than the number of straight lines necessary to cover the boundary of $F$, the above conditions are either to be intended in the least square sense or the integral in the first condition to be taken on a suitable subset of $\partial F$.
The (serendipity) edge space on $F$ is then defined as: 
\begin{multline}\label{defVeS}
  SV_{k}^{\edge}(F) \coloneq \Big\{
  \vv\in \bvec{L}^2(F)\st
  \DIV_F\vv \in \Poly{k}(F), \, \ROT_F\vv \in \Poly{k}(F),\,
  \vv \cdot \tt_E \in \Poly{k}(E)\quad\forall E \in \EF,  \\ 
  \int_{F} (\vv-{\mathbf \Pi}_{S,F}^{\edge} \vv )  \cdot \xxf\,p \df=0 \quad\forall p\in \Poly{\beta_F |\kd}(F)
  \Big\} ,
\end{multline}
where $\Poly{\beta_F |\kd}(F)$ is any space such that
$$
\Poly{\kd}(F) = \Poly{\beta_F}(F) \oplus  \Poly{\beta_F|\kd}(F) \, .
$$
The following operators constitute (once bases for the corresponding polynomial test spaces are chosen) a unisolvent set of degrees of freedom (DoFs) for $SV_{k}^{\edge}(F)$:
\begin{align}
&\bullet
\mbox{on each edge $E \in \EF$, $\int_E (\vv\cdot\tt_E) p \ds \quad \forall p \in \Poly{k}(E) $,}
\label{k-2dofe1}\\
&\bullet
\mbox{if $\beta_F\ge 0$: \ }
\int_{F}(\vv\cdot\xxf) \: p\df \quad \forall p \in\Poly{\beta_F}(F) ,
\label{k-2dofe3-sere}\\
&\bullet \mbox{if $k>0$: \ }\int_{F} \ROT_F\vv ~ p \df \quad \forall p \in \Poly[0]{k}(F) , \label{k-2dofe3}
\end{align}
where we recall that $\xxf=\xx-\bvec{c}_{F}$, and that $\Poly[0]{k}(F)$ was defined in \eqref{defPh}.
Following \cite[Eq.~(3.6)]{Beirao-da-Veiga.Brezzi.ea:18*2}, we can compute the $L^2$-orthogonal projector $\vlproj{k+1}{F}{}:SV^{\edge}_{k}(F) \rightarrow \vPoly{k+1}(F)$ using only the DoFs on $V^{\edge}_k(F)$ (that is, without the need of actually reconstructing the functions of $V^{\edge}_k(F)$).

\begin{remark}[Edge serendipity operator]
  One could modify the serendipity operator into 
  ${\mathbf \Pi}_{S,F}^{\edge}: {\cal S}^{\edge} \rightarrow \vPoly{k}(F) + \xxf^\perp\,\Poly[]{k}(F)$, 
  by increasing the test functions in \eqref{defPieS3} to $p \in \Poly[0]{k}(F)$. This would guarantee
  that the N\'ed\'elec space $\vPoly{k}(F) + \xxf^\perp\,\Poly[]{k}(F)$ is contained in $SV_{k}^{\edge}(F)$,
  and thus that the full N\'ed\'elec space of the first kind is contained 
  in ${V}^{\edge}_{k}(T)$, see \eqref{edge3} below. 
Such a change would however not improve the interpolation properties of the space with respect to Lemma \ref{prop:int-edge} below.
\end{remark}

\subsubsection{Nodal space on faces}

For the construction of the nodal serendipity space on faces we proceed as before. 
Let $\Pi^{\node}_{S,F}: {\cal S}^{\node} \rightarrow \Poly{k+1}(F)$, with ${\cal S}^{\node}$ a space of sufficiently regular functions $F\to\Real$, 
be a projector defined by: For all $q \in {\cal S}^{\node}$,
\begin{equation}\label{PiS-nod2}
\begin{aligned}
&\mbox{$\int_{\partial F} [\GRAD_F(q-\Pi^{\node}_{S,F}q)\cdot\tt_{\partial F}][\GRAD_F p\cdot\tt_{\partial F}] \ds=0\quad \forall p\in \Poly{k+1}(F)$},\\
&\mbox{$ \int_{\partial F} (q - \Pi_{S,F}^{\node} q) (\xxf\cdot\nn_{\partial F}) \ds = 0$} ,\\
&
\mbox{if $\beta_F \ge 0$: \, $\int_{F}\GRAD_F(q-\Pi^{\node}_{S,F} q)\cdot\xxf\, p\df=0\quad\forall p\in\Poly{\beta_F}(F)$},
\end{aligned}
\end{equation}
where, for each $E \in \EF$, $(\nn_{\partial F})_{|E} \coloneq{\omega_{FE}} \nn_{FE}$. The same observation as in \eqref{defPieS4} applies.
The (serendipity) nodal space of order $k+1$ on the face $F$ is then defined as:  
\begin{equation}\label{defVnS}
\begin{aligned}
SV_{k+1}^{\node}(F) \coloneq \Big\{ q \in H^1(F) : \: & q_{|E} \in \Poly{k+1}(E) \quad \forall E \subset \EF, \: \Delta_F q \in \Poly{\kd}(F) , \\
& \int_{F} (\GRAD_F q - \GRAD_F\Pi_{S,F}^{\node} q )\cdot \xxf\,p\df=0 \quad\forall p\in \Poly{\beta_F| \kd}(F) \Big\} .
\end{aligned}
\end{equation}
Note that the above conditions easily imply that functions in $SV_{k+1}^{\node}(F)$ are continuous on the boundary of $F$.
The DoFs in $SV_{k+1}^{\node}(F)$ are 
\begin{align}
& \bullet \mbox{for each vertex $\nu \in \VF$, the value $q(\bvec{c}_\nu)$, } \label{k-2dofn0}\\
& \bullet\mbox{if $k\ge 1$: \ for each edge $E \in \EF$, $\int_E q\, p\ds\quad\forall p\in\Poly{k-1}(E)$},\label{k-2dofn1}\\
& \bullet\mbox{if $\beta_F\ge 0$:}\  
\int_{F}(\GRAD_F q \cdot \xxf)\: p \df \quad \forall p \in\Poly{\beta_F}(F).
\label{k-2dofn2-sere}
\end{align}
Notice that, if $F$ is a triangle, the space $SV_{k+1}^{\node}(F)$ corresponds to the standard polynomial Finite Element space of degree $k+1$. We do not discuss here projectors in the space $SV_{k+1}^{\node}(F)$, since these will not be needed in the following.

\subsection{Local spaces and $L^2$-products on polyhedra}\label{sec:local3}

Let $T$ denote a mesh element of $\Th$, which we assume to have convex faces.
We introduce the following nodal, edge, and face (local) three-dimensional spaces, which, again, are minimal modifications of those in \cite{Beirao-da-Veiga.Brezzi.ea:18*2} (to which we refer for the proofs of the properties hereafter stated):
\begin{align}\label{nod3}
 V^{\node}_{k+1}(T) &\coloneq \Big\{ q \in H^1(T) \: : q_{|F}\in SV^{\node}_{k+1}(F)\quad\forall F\in\FT , \:
  \,\Delta\,q\in\Poly{k-1}(T)\Big\} ,
  \\ \label{edge3}
  V^{\edge}_{k}(T) &\coloneq \Big\{
  \begin{aligned}[t]
    &\vv\in\bvec{L}^2(T)\,:\DIV\vv\in \Poly{k-1}(T), \: \CURL(\CURL\vv) \in \vPoly{k}(T),
    \\
    &\quad \vv_{{\rm t},F} \in SV^{\edge}_{k}(F)\quad\forall F\in\FT,\; \vv\cdot\tt_E \mbox{ single valued on each edge } E\in\ET\Big\},
  \end{aligned}
\end{align}
where, as before, $\vv_{{\rm t},F}$ denotes the tangential trace of $\vv$ over $F$, and
\begin{equation} \label{face3}
\! {V}^{\face}_{k}(T) \!\coloneq
\! \Big\{ \ww\in\bvec{L}^2(T)\st {\DIV\ww\in \Poly{k}}(T), \, \CURL\ww\in\vPoly{k}(T),\:   \ww_{|F}\cdot\nn_F\in\Poly{k}(F)\quad\forall F \in \FT \Big\}.
\end{equation}
The following linear maps form a set of DoFs for $V^{\node}_{k+1}(T)$:
\begin{align}
& \bullet \mbox{for each vertex $\nu\in\VT$, the nodal value $q(\bvec{c}_\nu),$ } \label{dof-3dnk-1}\\
& \bullet \mbox{if $k\ge 1$: \ for each edge $E\in\ET$, $ \int_E q\, p\ds\quad\forall p\in\Poly{k-1}(E),$}\label{dof-3dnk-2}\\
&\bullet
\mbox{for each face $F\in\FT$ with $\beta_F\ge 0$, $\int_{F}(\GRAD_F q\cdot \xx_F)\: p \df \quad \forall p \in\Poly{\beta_F}(F)$,}\label{dof-3dnk-3}\\
&\bullet\mbox{$\int_{T} (\GRAD q\cdot \xxE ) \: p \dPP \quad \forall p \in\Poly{\kdP}(T).$}
\label{dof-3dnk-4}
\end{align}
In $V^{\edge}_{k}(T)$, the DoFs are
\begin{align}
& \bullet\mbox{for each edge $E\in\ET$, $\int_E (\vv\cdot\tt_E) p\ds \quad \forall p \in \Poly{k}(E) ,$ } \label{dof-3dek-1}\\
& \bullet \mbox{for each face $F\in\FT$ with $\beta_F\ge 0$, $\int_{F}({\vv_{{\rm t},F}}\cdot\xxf) \: p\df \quad \forall p \in\Poly{\beta_F}(F),$} \label{dof-3dek-2}\\
& \bullet \mbox{if $k>0$: \ for each face $F\in\FT$, $\int_{F} \ROT_F{\vv_{{\rm t},F}} \, p \df \quad \forall p \in \Poly[0]{k}(F)$} , \label{dof-3dek-3}\\
& \bullet \mbox{$ \int_{T}
(\vv\cdot\xx_{T}) p\, \dPP \quad \forall p \in \Poly{\kdP}(T)$} ,\label{dof-3dek-4} \\
& \bullet\mbox{ $\int_{T} 
(\CURL\vv)\cdot(\xx_{T}\times \bvec{p})\,\dPP \quad \forall \bvec{p} \in\vPoly{k}(T)$},  \label{dof-3dek-5}
\end{align}
where we recall that $\bvec{x}_{T} = \bvec{x} - \bvec{c}_T$.
Finally, for $V^{\face}_{k}(T)$ we have the DoFs
\begin{align}
& \bullet\mbox{for any face $F\in\FT$, $\int_F
(\ww\cdot\nn_F) p \df \quad \forall p \in \Poly{k}(F), $} \label{dof-3dfk-1}\\
& \bullet \mbox{if $k>0$: \ $\int_{T}
(\DIV \ww) p \dPP \quad \forall p \in \Poly[0]{k}(T), 
$ }\label{dof-3dfk-2} \\
& \bullet \mbox{$\int_{T}
\ww\cdot (\xx_{T}\times \bvec{p}) \dPP \quad \forall \bvec{p}\in\vPoly{k}(T)$}.
 \label{dof-3dfk-3}
\end{align}
Many of the above DoFs could equivalently be written using the polynomial subspaces \eqref{eq:spaces.F}-\eqref{eq:spaces.T} instead of using an explicit expression. We prefer here to conform to the standard notation used in the VEM literature; a bridge between the two approaches will be built in Section \ref{sec:bridge}. 
We also notice that, on tetrahedra, the spaces above have a higher number of internal DoFs than in the corresponding finite element case; one could reduce such number by applying an enhancement approach, see Remark \ref{rem:enh}.
However, here we will make no effort to reduce this number, as it is assumed that, in practice, they could be eliminated by static condensation (since they are internal to the elements).

As shown in \cite[Proposition 3.7]{Beirao-da-Veiga.Brezzi.ea:18*2}, from the above DoFs we can compute (in particular) the following $L^2$-orthogonal projections of virtual functions on polynomial spaces: 
from $V^{\edge}_{k}(T)$ to $\vPoly{k}(T)$ and
from $V^{\face}_{k}(T)$ to $\vPoly{k}(T)$ (actually, notice that we could also compute richer projections from $V^{\face}_{k}(T)$ to $\vPoly{k+1}(T)$, but this will not be used in the following). 
These projections are used to define the following scalar products for edge and face spaces mimicking the $\bvec{L}^2(T)$ product:
For $\bullet\in\{\edge,\face\}$, we let
$$
[\vv , \ww]_{V^\bullet_k(T)} \coloneq \int_T{\vlproj{k}{T}}\vv \cdot {\vlproj{k}{T}}\ww
+ s^T((I-{\vlproj{k}{T}})\vv , (I-{\vlproj{k}{T}})\ww) 
\qquad \forall (\vv,\ww) \in V^\bullet_k(T)\times V^\bullet_k(T),
$$
where the symmetric and computable bilinear form $s^T(\cdot,\cdot)$ can be taken, for instance, as
\begin{equation}\label{eq:vem:sT}
  s^T(\vv , \ww) = {h_{T}^3} \sum_i ({\rm dof}_i\{\vv\}) \: ({\rm dof}_i\{\ww\}) \qquad \forall (\vv,\ww) \in V^\bullet_k(T)\times V^\bullet_k(T),
\end{equation}
where we assume that all DoFs ${\rm dof}_i$ are scaled in order to behave (with respect to element size changes) as nodal evaluations.
It is immediate to check that, by construction,
\begin{equation}\label{consiE3k}
[\vv,\bvec{p}]_{V^\bullet_k(T)}=\int_{T} \vv\cdot \bvec{p}\dE 
\qquad \forall \vv\in V^\bullet_k(T),
\; \forall \bvec{p}\in\vPoly{k}(T).
\end{equation}

\begin{remark}[Alternative stabilisation]\label{rem:stab.VEM}
Another choice of stabilisation in $V^\edge_k(T)$ is
\[
s^T(\bvec{v},\bvec{w})=\sum_{E\in\ET}h_T^2\int_E \lproj{k}{E}(\bvec{v}\cdot\tangent_E)\lproj{k}{E}(\bvec{w}\cdot\tangent_E)
+\sum_{F\in\FT}h_T\int_F \vlproj{k+1}{F}\bvec{v}_{{\rm t},F}\cdot\vlproj{k+1}{F}\bvec{v}_{{\rm t},F}
\]
while, in $V^\face_k(T)$, we can take
\[
s^T(\bvec{v},\bvec{w})=\sum_{F\in\FT}h_T\int_F \lproj{k}{F}(\bvec{v}\cdot\normal_F)\lproj{k}{F}(\bvec{w}\cdot\normal_F)
+\int_T \Gcproj{k+1}{T}\bvec{v}\cdot \Gcproj{k+1}{T}\bvec{w}.
\]
These stabilisations are inspired by the ones considered in the DDR setting and, following the ideas in \cite[Lemma 5]{Di-Pietro.Droniou:21*1} and using discrete inverse inequalities, it can easily be checked that they yield coercive and consistent $\bvec{L}^2(T)$-like inner products.
\end{remark}

\begin{remark}[Alternative choices of degrees of freedom]
An integration by parts easily shows that the DoFs \eqref{dof-3dnk-3} could be replaced by
$\int_F \! q \, p \df$ for all $p \in\Poly{\beta_F}(F)$, and analogously the set \eqref{dof-3dnk-4} by
$\int_{T} \! q  p \dPP$ for all $p \in\Poly{\kdP}(T)$.
Knowledge on one set of DoFs implies knowledge on the other set, and vice-versa. Similarly, the set of DoFs \eqref{dof-3dek-5} could be replaced by 
$\int_{T} \! \vv \cdot \CURL(\xx_{T}\times \bvec{p}) \dPP$ for all $ \bvec{p} \in\vPoly{k}(T)$ and the set \eqref{dof-3dfk-2} 
by $\int_{T} \ww\cdot \GRAD p \dPP$ for all $p \in \Poly[0]{k}(T)$.
The advantage of the current choice is a more direct expression of the differential operators in terms of DoFs. The advantage of the alternative choice would be the reduced regularity needed in order to compute the DoF-interpolant of a generic function.
\end{remark}

\begin{remark}[Enhancement]
  \label{rem:enh}
  Note that one could apply an ``enhancement'' approach (in the elements volume), in order to reduce the number of internal DoFs. The enhancement idea, first introduced in \cite{Projectors}, is to adopt a slightly different definition of the VEM spaces in order to reduce the number of DoFs without sacrificing accuracy or computability. 
  For the spaces presented here, a particularly interesting form of enhancement can be obtained in the spirit of the DDR approach, see Section \ref{sec:bridge:comparison}.
\end{remark}

\subsection{Global spaces}\label{sec:global}

The global spaces are constructed by standard degrees of freedom gluing, as in classical Finite Elements. The scalar nodal space is conforming in $H^1(\Omega)\cap L^2_0(\Omega)$, the edge space in $\Hcurl{\Omega}$, and the face space in $\Hdiv{\Omega}$.
We set
\begin{align}\label{glo-n}
  V^{\node}_{k+1,0} &\coloneq\Big\{ q \in H^1(\Omega) \cap L^2_0(\Omega) \mbox{ such that } q_{|T} \in V^{\node}_{k+1}(T)\quad \forall T \in \Th \Big\},
  \\ \label{glo-e}
  V^{\edge}_{k} &\coloneq\Big\{ \vv \in\Hcurl{\Omega} \mbox{ such that } \vv_{|T} \in V^{\edge}_{k}(T)\quad\forall T \in \Th \Big\},
  \\ \label{globf}
  V^{\face}_{k} &\coloneq\Big\{\ww\in\Hdiv{\Omega}\mbox{ such that }\ww_{|T}\in V^{\face}_{k}(T)\quad\forall T \in \Th\Big\}.
\end{align}
The global DoFs can be trivially derived from the local ones (for instance, the space $V^{\node}_{k+1,0}$ has one DoF per vertex of $\Th$, $k$ DoFs per edge, etc.).
For $\bullet\in\{\edge,\face\}$, the global scalar products are given by
$$
[\vv,\ww]_{V^\bullet_k} \coloneq \sum_{T \in \Th} [\vv_{|T},\ww_{|T}]_{V^\bullet_k(T)} 
\qquad \forall (\vv,\ww) \in V^\bullet_k\times V^\bullet_k.
$$

\begin{remark}[Virtual exact sequence]\label{rem:ES}
  The Virtual Element spaces defined above form an exact sequence:
  $$
  \begin{tikzcd}
    \Real \arrow{r}{{\rm i}}
    & V^{\node}_{k+1,0}\arrow{r}{\GRAD}
    & V^{\edge}_{k}\arrow{r}{\CURL}
    & V^{\face}_{k}\arrow{r}{\DIV}
    & V^{\vol}_{k}\arrow{r}{0}
    & 0,
  \end{tikzcd}
  $$
  where the space $V^{ \vol}_{k} = \Poly{k}(\Th)\coloneq\left\{ q \in L^2(\Omega) \st q_{|T} \in \Poly{k}(T) \quad \forall T \in \Th \right\}$ collects broken polynomial functions on $\Th$ of total degree $\le k$.
\end{remark}
It is important to point out that the inclusions above can be also computed in practice, in the following sense.
Given the DoFs of $q\in V^{\node}_{k+1,0}$, we can compute the DoFs of $\GRAD q$ in $V^{\edge}_{k}$; given the DoFs of $\vv\in V^{\edge}_{k}$, we can compute the DoFs of $\CURL\vv$ in $V^{\face}_{k}$; from the DoFs of $\ww\in V^{\face}_{k}$ we can compute its divergence in $V^{\vol}_{k}$.
This observation also entails that we have a commuting diagram property that involves the natural interpolation operators defined through the DoFs.
We refer to \cite{Beirao-da-Veiga.Brezzi.ea:18*2} for a deeper overview on such aspects.

\subsection{Discrete problem and convergence}\label{sec:disc-pbl}

We are now able to present the VEM discretization of problem \eqref{eq:strong}. Assuming that $\ff \in \bvec{H}^s(\Omega)$ with $\CURL \ff \in \bvec{H}^s(\Omega)$ for some $s > \frac12$ and that $\ff_{|E}\cdot\tangent_E$ is integrable on any edge $E\in\Eh$ (see \cite{edgeface,edgefacegeneral}), we denote its interpolant in $V^{\edge}_{k}$ by $\ff_I = {\cal I}^{\edge}_k (\ff)$.
The VEM problem reads:
\begin{equation}\label{eq:discr-pbl}
\left\{~
\begin{aligned}
& \text{Find $\uu_h \in V^{\edge}_{k}$ and $p_h \in V^{\node}_{k+1,0}$ such that } \\
& \begin{array}{r@{{}={}}ll}
[\CURL \uu_h , \CURL \vv_h]_{V^{\face}_{k}} + [\GRAD p_h,\vv_h]_{V^{\edge}_{k}} & [\ff_I,\vv_h]_{V^{\edge}_{k}}
&\quad \forall \vv_h \in V^{\edge}_{k}, \\
 -[\GRAD q_h,\uu_h]_{V^{\edge}_{k}} & 0 &\quad \forall q_h \in V^{\node}_{k+1,0}.
\end{array}
\end{aligned}
\right.
\end{equation}
Notice that all the above scalar products are well defined thanks to the inclusion properties stemming from Remark \ref{rem:ES}.
\begin{remark}[Pressure robustness]\label{rem:PBgrad}
Whenever the loading $\ff$ is a gradient, i.e.\ $\ff = \GRAD \psi$ for some sufficiently regular $\psi$, the commuting diagram property recalled at the end of the previous section gives
  $$
  \ff_I = {\cal I}^{\edge}_k (\GRAD \psi) = \GRAD {\cal I}^{\node}_{k+1}(\psi) 
  $$
  where the operator ${\cal I}^{\node}_{k+1}$ denotes the natural DoF interpolator in $ V^{\node}_{k+1}$ (which consists in the space $V^\node_{k+1,0}$ without the zero average condition).
  Therefore, 
  reasoning as in Remark \ref{rem:ddr:pressure.robustness}, it can be checked that irrotational perturbations of the loading term have no influence on the velocity, and that the method is \emph{pressure robust} in the sense of \cite{Linke.Merdon:16}.
This will also be reflected in the convergence result of Theorem \ref{thm:vem:convergence} below, where the right-hand side of the error estimate does not depend on the pressure.
Such property is obtained here on general meshes, for arbitrary polynomial degrees, and without resorting to submeshing.
\end{remark}

To close this section, we state the main convergence result for the VEM scheme.
As for the DDR scheme, the following theorem requires a high (piecewise) regularity for the exact velocity $\uu$, but the adopted approach allows to obtain full independence of the velocity error from the pressure solution.
\begin{theorem}[Error estimate for the VEM scheme \eqref{eq:discr-pbl}]\label{thm:vem:convergence}
Denote by $\bvec{u}\in\Hcurl{\Omega}\cap\Hdiv{\Omega}$ and $p\in H^1(\Omega)\cap L^2_0(\Omega)$, respectively, the velocity and pressure fields solution of the weak formulation \eqref{eq:variational}, and by $\bvec{u}_h\in V^{\edge}_{k} $ and $p_h\in V^{\node}_{k+1,0} $ the corresponding discrete counterparts solving the VEM scheme \eqref{eq:discr-pbl}.
Assume $p \in{H}^{\gamma}(\Omega)$ with $\gamma > \frac32$, and that $\uu$, $\CURL\uu$, $\CURL\CURL\uu$, $\ff$ and $\CURL\ff$ are in $\bvec{H}^s(\Th)$ for some $\frac12 < s \le k+1$, with tangential components of $\uu$, $\CURL\CURL\uu$ and $\ff$ integrable on all edges.
Denote the interpolants $\uu_I = {\cal I}^{\edge}_k(\uu)$ and $p_I = {\cal I}^{\node}_{k+1}(p)$.
It holds
\begin{multline}\label{eq:conv.VEM}
  \norm[\Hcurl{\Omega}]{\uu_h - \uu_I} + \norm[H^1(\Omega)]{p_h - p_I}
  \\
  \lesssim h^{s} 
  \big(
  \seminorm[\Hscurl{s}{\Th}]{\uu} + \seminorm[\bvec{H}^s(\Th)]{\CURL\CURL\uu} + \seminorm[\bvec{H}^s(\Th)]{\CURL\ff}
  \big) \, .
\end{multline}
The same bound holds also for $\norm[\Hcurl{\Omega}]{\uu_h - \uu}$ and, upon the addition of the term $\,h^{s} \seminorm[{H}^{1+s}(\Th)]{p}$ to the right-hand side, also for $\norm[H^1(\Omega)]{p_h - p}$, see Corollary \ref{prop:conv2}.
\end{theorem}
\begin{proof}
  See Section \ref{sec:conv}.
\end{proof}

\begin{remark}[Regularity condition]
  By making use of the second result in Lemma \ref{prop:int-edge} below, and noticing that $\CURL\CURL\uu = {\cal P}_{\cal H}\ff$, the Helmholtz-Hodge projection of $\ff$, if $s > \frac32$ the regularity conditions can be slightly reduced to $\uu,\CURL\uu,\ff,{\cal P}_{\cal H}\ff$ in $\bvec{H}^s(\Th)$. The right-hand side in \eqref{eq:conv.VEM} changes accordingly. 
\end{remark}

\section{Numerical experiments}\label{sec:numerical.examples}

In this section we numerically test the proposed DDR and VEM approaches (the latter, for the sake of simplicity, with $\beta_F=k$ for all $F\in\Fh$). 
More specifically, we start with a numerical convergence analysis that will validate from the practical standpoint Theorems~\ref{thm:ddr:convergence} and~\ref{thm:vem:convergence} for the DDR and VEM schemes, respectively.
In Subsection~\ref{sec:test.robust} we perform a robustness analysis with respect to the strength of the pressure component. 

We consider the following discrete and continuous norms as errors indicators.
On the one hand, for the DDR scheme~\eqref{eq:discrete}, 
we compute as discrete errors the following quantities
\[
E^{\rm d}_{\bvec{u}}\coloneq\tnorm[\CURL,h]{\uvec{u}_h-\Icurl{h}\bvec{u}}
\,,\quad
E^{\rm d}_{p}\coloneq\norm[\CURL,h]{\uGh\underline{p}_h-\uGh\Igrad{h}p},
\]
where the norms $\tnorm[\CURL,h]{{\cdot}}$ and $\norm[\CURL,h]{{\cdot}}$ 
are defined in Section~\ref{sec:ddr.problem}.
Then, as continuous error indicators, we compute
\[
E^{\rm c}_{\bvec{u}}\coloneq\Big(\norm[\bvec{L}^2(\Omega)]{\Pcurl[h]\uvec{u}_h-\bvec{u}}^2+\norm[\bvec{L}^2(\Omega)]{\cCh\uvec{u}_h-\CURL\bvec{u}}^2\Big)^{\frac12}
\,,\quad
E^{\rm c}_{p}\coloneq\norm[\bvec{L}^2(\Omega)]{\cGh\underline{p}_h-\GRAD p},
\]
where $\Pcurl[h]$, $\cCh$, and $\cGh$ are the global potentials 
and operators obtained patching the local potentials and discrete vector calculus operators $\Pcurl$, $\cCT$, and $\cGT$, respectively. 

On the other hand, for the VEM scheme~\eqref{eq:discr-pbl}, 
we define the following discrete and continuous errors:
\[
E^{\rm d}_{\bvec{u}}\coloneq\Big([\![\bvec{u}_h-\bvec{u}_I]\!]_{V_k^\edge}^2+ [\![\CURL\bvec{u}_h-\CURL\bvec{u}_I]\!]_{V_k^\face}^2\Big)^{\frac12} 
\,,\quad
E^{\rm d}_{p}\coloneq[\![ \GRAD p_h - \GRAD p_I ]\!]_{V_k^\edge},
\]
where $[\![\cdot]\!]_{V_k^\bullet}$ is the global discrete $L^2$-norm corresponding to the scalar product $[\cdot,\cdot]_{V_k^\bullet}$, and
\begin{align*}
E^{\rm c}_{\bvec{u}}\coloneq{}&\Big( \norm[\bvec{L}^2(\Omega)]{\vlproj{k}{h}\bvec{u}_h-\bvec{u}}^2+\norm[\bvec{L}^2(\Omega)]{\vlproj{k}{h}\CURL\bvec{u}_h-\CURL\bvec{u}}^2 \Big)^{\frac12}
\,,\\
E^{\rm c}_{p}\coloneq{}&\norm[\bvec{L}^2(\Omega)]{\vlproj{k}{h}\GRAD p_h-\GRAD p}\,,
\end{align*}
with $\vlproj{k}{h}$ and $\lproj{k}{h}$ global projectors on the broken polynomial spaces $\vPoly{k}(\Th)$ and $\Poly{k}(\Th)$, respectively.

Independently of the scheme taken, the errors in pressure are only based on gradients.
This choice is due to the fact that the norm of the (discrete or continuous) gradient is a norm on the (discrete or continuous) pressure space.

Since the solutions we are considering are smooth,
Theorems~\ref{thm:ddr:convergence} and~\ref{thm:vem:convergence} state that 
all these errors indicators should decay as $h^{k+1}$ for both DDR and VEM.
To show this trend, we consider two families of meshes of the unit cube $\Omega=(0,1)^3$:
one composed of tetrahedra and the other one made of Voronoi cells.
Each mesh family is a sequence with decreasing meshsize.
Figure~\ref{fig:meshes} shows one mesh from each family.

\begin{figure}\centering
  \begin{minipage}{0.45\textwidth}
    \includegraphics[width=.9\textwidth, trim=2cm 2cm 2.5cm 2.4cm]{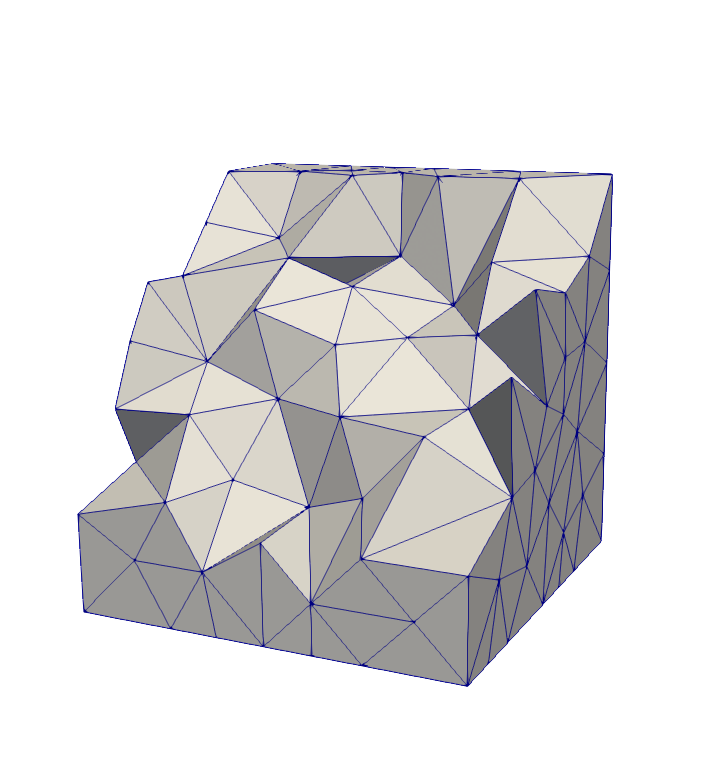}
    \subcaption{Tetrahedral mesh}
  \end{minipage}
  \hspace{0.25cm}
  \begin{minipage}{0.45\textwidth}
    \includegraphics[width=.9\textwidth, trim=2.8cm 2.4cm 2.4cm 2.4cm]{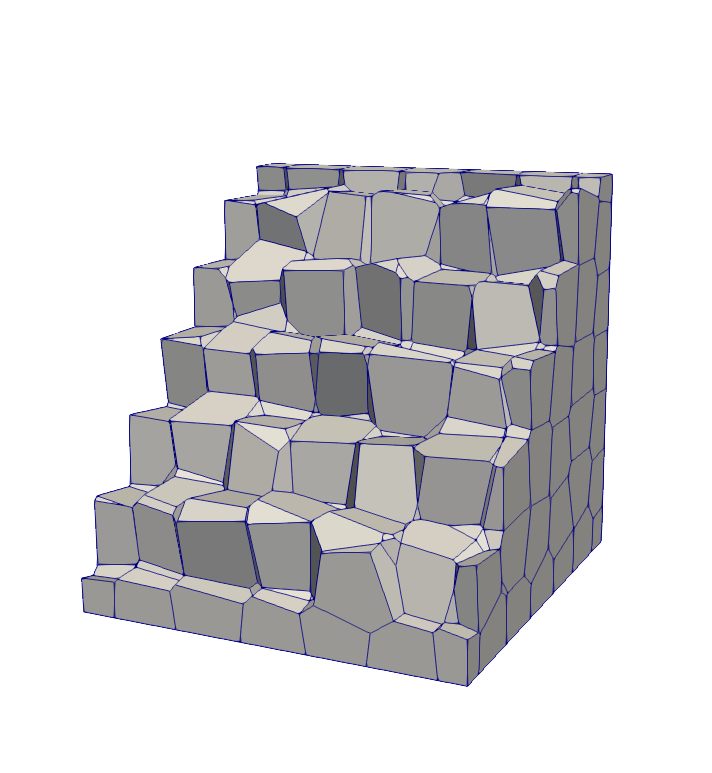}
    \subcaption{Voronoi mesh}
  \end{minipage}
  \caption{Members of mesh families used in numerical tests.}
  \label{fig:meshes}
\end{figure}

Both DDR and VEM schemes have been implemented in the open source C++ library \texttt{HArDCore3D} (see \url{https://github.com/jdroniou/HArDCore}), based on the linear algebra library \texttt{Eigen3} (see \url{http://eigen.tuxfamily.org}).
The polynomial basis functions are scaled with respect to the diameter of their associated geometric entity, and additionally orthonormalised to improve the condition number
of the resulting system in the presence of elongated faces or elements (see \cite[Section B.1.1]{Di-Pietro.Droniou:20}).
In order to solve the linear system arising from each discretization, we use the \texttt{Intel MKL PARDISO} library (see \url{https://software.intel.com/en-us/mkl})~\cite{pardiso-7.2a}. 
The VEM implementation uses the stabilisations described in Remark \ref{rem:stab.VEM}.
The zero average condition on the pressure is imposed through the introduction of a Lagrange multiplier in the space of constant functions over $\Omega$, and static condensation is applied to eliminate element degrees of freedom.
For the lowest order case of the VEM scheme, the zero-average pressure condition is enforced through a weighted node sum, since the element-wise integral is not available in that case (an alternative option would be to introduce a small modification in the discrete spaces, as in \cite{Beirao-da-Veiga.Brezzi.ea:18*1}, to make this quantity also computable for $k=0$).
For both the DDR and VEM schemes, we applied a multiplicative factor of $0.1$ to the stabilisations. This factor only impacts the magnitude of the errors (not the rates of convergences); we found by trial-and-error that a factor of $0.1$ was leading for both methods to reasonable magnitudes; deeper analysis of the optimal choice of stabilisation and/or multiplicative factor is the subject of future work.

\subsection{Convergence analysis}\label{sec:test.conv}

We build the right-hand side of problem \eqref{eq:strong} in such a way that the exact solution is 
\[
p(x,y,z)=\sin(2\pi x)\sin(2\pi y)\sin(2\pi z)\quad\mbox{ and }\quad
\bvec{u}(x,y,z)=\begin{bmatrix}
  \frac12 \sin(2\pi x)\cos(2\pi y)\cos(2\pi z)\\[.5em]
  \frac12 \cos(2\pi x)\sin(2\pi y)\cos(2\pi z)\\[.5em]
  -\cos(2\pi x)\cos(2\pi y)\sin(2\pi z)
  \end{bmatrix}.
\]
In Figures~\ref{fig:conv.tetra} and~\ref{fig:conv.voro}  we show convergence plots for both DDR and VEM schemes.
We observe an initial pre-asymptotic behavior, probably due to the fact that the first meshes are indeed quite coarse, then the convergence rates stabilize to the ones predicted by the theoretical analysis.
It should also be noted that the regularity factor of the Voronoi meshes increases quite strongly along the family (see the discussion in \cite[Section 5.1.8.2]{Di-Pietro.Droniou:20}), but that it does not seem to affect the proposed schemes, which display an apparent good robustness with respect to this factor.

Whilst the continuous errors for both schemes yield very similar values, the discrete errors are higher for the VEM scheme than the DDR scheme.
Contrary to the continuous errors, comparing the discrete errors is more tricky as they measure different quantities for each scheme, since they are based on different sets of DoFs (for instance, the VEM set of DoFs involves function derivatives as opposed to function values only for DDR).

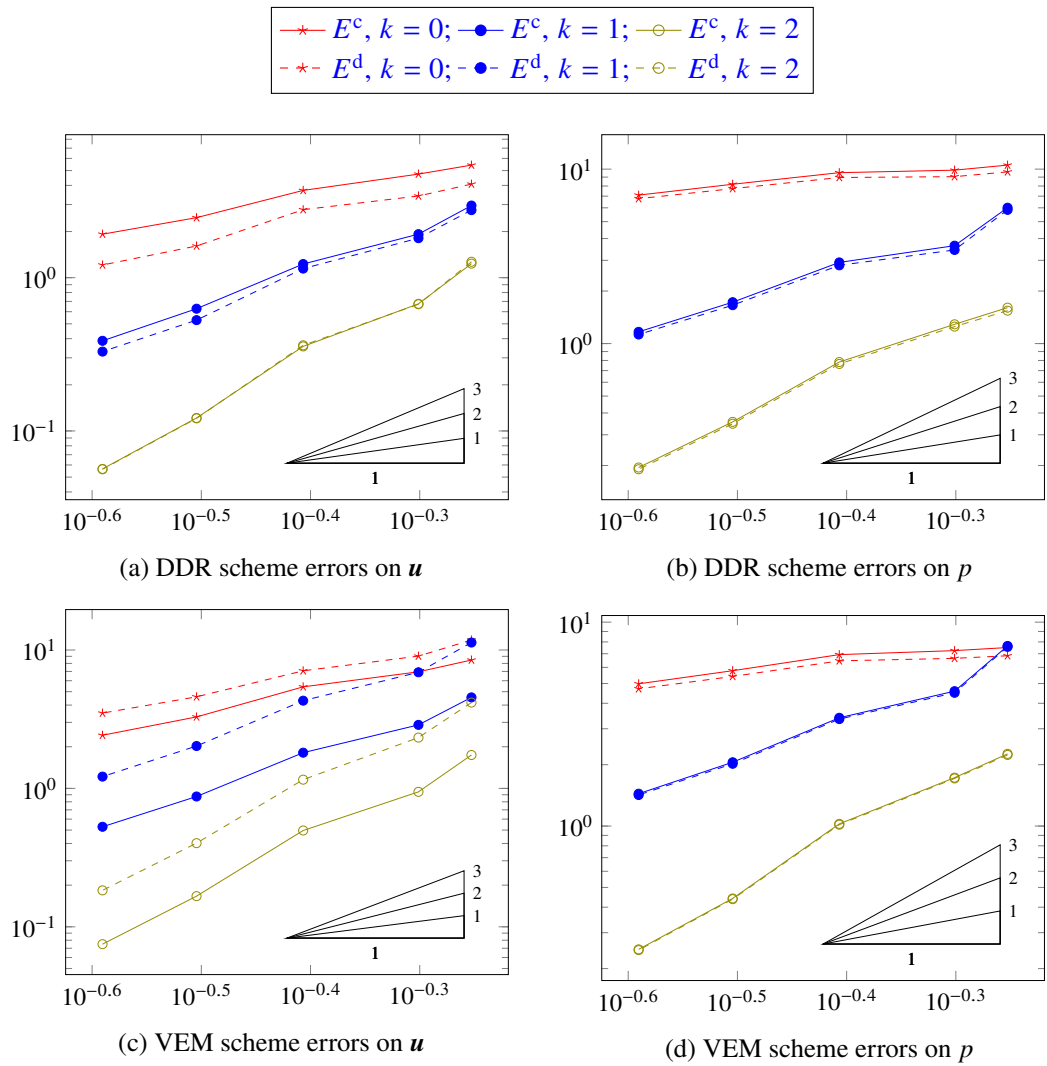
\begin{figure}\centering
  \ref{ddr.conv.tets}
  \vspace{0.50cm}\\
  \begin{minipage}{0.45\textwidth}
    \begin{tikzpicture}[scale=0.85]
      \begin{loglogaxis} [legend columns=3, legend to name=ddr.conv.tets]  
        \logLogSlopeTriangle{0.90}{0.4}{0.1}{1}{black};
        \addplot [mark=star, red] table[x=MeshSize,y=absE_cHcurlVel] {outputs/ddr/tets-trigo-k0-sca1/data_rates.dat};
        \addlegendentry{$E^{\rm c}$, $k=0$;}
        \logLogSlopeTriangle{0.90}{0.4}{0.1}{2}{black};
        \addplot [mark=*, blue] table[x=MeshSize,y=absE_cHcurlVel] {outputs/ddr/tets-trigo-k1-sca1/data_rates.dat};
        \addlegendentry{$E^{\rm c}$, $k=1$;}
        \logLogSlopeTriangle{0.90}{0.4}{0.1}{3}{black};
        \addplot [mark=o, olive] table[x=MeshSize,y=absE_cHcurlVel] {outputs/ddr/tets-trigo-k2-sca1/data_rates.dat};
        \addlegendentry{$E^{\rm c}$, $k=2$}
        \addplot [mark=star, mark options=solid, red, dashed] table[x=MeshSize,y=absE_HcurlVel] {outputs/ddr/tets-trigo-k0-sca1/data_rates.dat};
        \addlegendentry{$E^{\rm d}$, $k=0$;}
        \addplot [mark=*, mark options=solid, blue, dashed] table[x=MeshSize,y=absE_HcurlVel] {outputs/ddr/tets-trigo-k1-sca1/data_rates.dat};
        \addlegendentry{$E^{\rm d}$, $k=1$;}
        \addplot [mark=o, mark options=solid, olive, dashed] table[x=MeshSize,y=absE_HcurlVel] {outputs/ddr/tets-trigo-k2-sca1/data_rates.dat};
        \addlegendentry{$E^{\rm d}$, $k=2$}
      \end{loglogaxis}            
    \end{tikzpicture}
    \subcaption{DDR scheme errors on $\bvec{u}$}
  \end{minipage}
  \begin{minipage}{0.45\textwidth}
    \begin{tikzpicture}[scale=0.85]
      \begin{loglogaxis} 
        \logLogSlopeTriangle{0.90}{0.4}{0.1}{1}{black};
        \addplot [mark=star, red] table[x=MeshSize,y=absE_cL2GradPre] {outputs/ddr/tets-trigo-k0-sca1/data_rates.dat};
        \logLogSlopeTriangle{0.90}{0.4}{0.1}{2}{black};
        \addplot [mark=*, blue] table[x=MeshSize,y=absE_cL2GradPre] {outputs/ddr/tets-trigo-k1-sca1/data_rates.dat};
        \logLogSlopeTriangle{0.90}{0.4}{0.1}{3}{black};
        \addplot [mark=o, olive] table[x=MeshSize,y=absE_cL2GradPre] {outputs/ddr/tets-trigo-k2-sca1/data_rates.dat};
        \addplot [mark=star, mark options=solid, red, dashed] table[x=MeshSize,y=absE_L2GradPre] {outputs/ddr/tets-trigo-k0-sca1/data_rates.dat};
        \addplot [mark=*, mark options=solid, blue, dashed] table[x=MeshSize,y=absE_L2GradPre] {outputs/ddr/tets-trigo-k1-sca1/data_rates.dat};
        \addplot [mark=o, mark options=solid, olive, dashed] table[x=MeshSize,y=absE_L2GradPre] {outputs/ddr/tets-trigo-k2-sca1/data_rates.dat};
      \end{loglogaxis}            
    \end{tikzpicture}
    \subcaption{DDR scheme errors on $p$}
  \end{minipage}\\[0.5em]
  \begin{minipage}{0.45\textwidth}
    \begin{tikzpicture}[scale=0.85]
      \begin{loglogaxis} [legend columns=3, legend to name=vem.conv.tets]  
        \logLogSlopeTriangle{0.90}{0.4}{0.1}{1}{black};
        \addplot [mark=star, red] table[x=MeshSize,y=absE_cHcurlVel] {outputs/vem/tets-trigo-k0-sca1/data_rates.dat};
        \addlegendentry{$E^{\rm c}$, $k=0$;}
        \logLogSlopeTriangle{0.90}{0.4}{0.1}{2}{black};
        \addplot [mark=*, blue] table[x=MeshSize,y=absE_cHcurlVel] {outputs/vem/tets-trigo-k1-sca1/data_rates.dat};
        \addlegendentry{$E^{\rm c}$, $k=1$;}
        \logLogSlopeTriangle{0.90}{0.4}{0.1}{3}{black};
        \addplot [mark=o, olive] table[x=MeshSize,y=absE_cHcurlVel] {outputs/vem/tets-trigo-k2-sca1/data_rates.dat};
        \addlegendentry{$E^{\rm c}$, $k=2$}
        \addplot [mark=star, mark options=solid, red, dashed] table[x=MeshSize,y=absE_HcurlVel] {outputs/vem/tets-trigo-k0-sca1/data_rates.dat};
        \addlegendentry{$E^{\rm d}$, $k=0$;}
        \addplot [mark=*, mark options=solid, blue, dashed] table[x=MeshSize,y=absE_HcurlVel] {outputs/vem/tets-trigo-k1-sca1/data_rates.dat};
        \addlegendentry{$E^{\rm d}$, $k=1$;}
        \addplot [mark=o, mark options=solid, olive, dashed] table[x=MeshSize,y=absE_HcurlVel] {outputs/vem/tets-trigo-k2-sca1/data_rates.dat};
        \addlegendentry{$E^{\rm d}$, $k=2$}
      \end{loglogaxis}            
    \end{tikzpicture}
    \subcaption{VEM scheme errors on $\bvec{u}$}
  \end{minipage}
  \begin{minipage}{0.45\textwidth}
    \begin{tikzpicture}[scale=0.85]
      \begin{loglogaxis} 
        \logLogSlopeTriangle{0.90}{0.4}{0.1}{1}{black};
        \addplot [mark=star, red] table[x=MeshSize,y=absE_cL2GradPre] {outputs/vem/tets-trigo-k0-sca1/data_rates.dat};
        \logLogSlopeTriangle{0.90}{0.4}{0.1}{2}{black};
        \addplot [mark=*, blue] table[x=MeshSize,y=absE_cL2GradPre] {outputs/vem/tets-trigo-k1-sca1/data_rates.dat};
        \logLogSlopeTriangle{0.90}{0.4}{0.1}{3}{black};
        \addplot [mark=o, olive] table[x=MeshSize,y=absE_cL2GradPre] {outputs/vem/tets-trigo-k2-sca1/data_rates.dat};
        \addplot [mark=star, mark options=solid, red, dashed] table[x=MeshSize,y=absE_L2GradPre] {outputs/vem/tets-trigo-k0-sca1/data_rates.dat};
        \addplot [mark=*, mark options=solid, blue, dashed] table[x=MeshSize,y=absE_L2GradPre] {outputs/vem/tets-trigo-k1-sca1/data_rates.dat};
        \addplot [mark=o, mark options=solid, olive, dashed] table[x=MeshSize,y=absE_L2GradPre] {outputs/vem/tets-trigo-k2-sca1/data_rates.dat};
      \end{loglogaxis}            
    \end{tikzpicture}
    \subcaption{VEM scheme errors on $p$}
  \end{minipage}
  \caption{Tetrahedral meshes: errors with respect to  $h$}
  \label{fig:conv.tetra}
\end{figure}

\begin{figure}\centering
  \ref{ddr.conv.voro}
  \vspace{0.50cm}\\
  \begin{minipage}{0.45\textwidth}
    \begin{tikzpicture}[scale=0.85]
      \begin{loglogaxis} [legend columns=3, legend to name=ddr.conv.voro]  
        \logLogSlopeTriangle{0.90}{0.4}{0.1}{1}{black};
        \addplot [mark=star, red] table[x=MeshSize,y=absE_cHcurlVel] {outputs/ddr/voro-trigo-k0-sca1/data_rates.dat};
        \addlegendentry{$E^{\rm c}$, $k=0$;}
        \logLogSlopeTriangle{0.90}{0.4}{0.1}{2}{black};
        \addplot [mark=*, blue] table[x=MeshSize,y=absE_cHcurlVel] {outputs/ddr/voro-trigo-k1-sca1/data_rates.dat};
        \addlegendentry{$E^{\rm c}$, $k=1$;}
        \logLogSlopeTriangle{0.90}{0.4}{0.1}{3}{black};
        \addplot [mark=o, olive] table[x=MeshSize,y=absE_cHcurlVel] {outputs/ddr/voro-trigo-k2-sca1/data_rates.dat};
        \addlegendentry{$E^{\rm c}$, $k=2$}
        \addplot [mark=star, mark options=solid, red, dashed] table[x=MeshSize,y=absE_HcurlVel] {outputs/ddr/voro-trigo-k0-sca1/data_rates.dat};
        \addlegendentry{$E^{\rm d}$, $k=0$;}
        \addplot [mark=*, mark options=solid, blue, dashed] table[x=MeshSize,y=absE_HcurlVel] {outputs/ddr/voro-trigo-k1-sca1/data_rates.dat};
        \addlegendentry{$E^{\rm d}$, $k=1$;}
        \addplot [mark=o, mark options=solid, olive, dashed] table[x=MeshSize,y=absE_HcurlVel] {outputs/ddr/voro-trigo-k2-sca1/data_rates.dat};
        \addlegendentry{$E^{\rm d}$, $k=2$}
      \end{loglogaxis}            
    \end{tikzpicture}
    \subcaption{DDR scheme errors on $\bvec{u}$}
  \end{minipage}
  \begin{minipage}{0.45\textwidth}
    \begin{tikzpicture}[scale=0.85]
      \begin{loglogaxis} 
        \logLogSlopeTriangle{0.90}{0.4}{0.1}{1}{black};
        \addplot [mark=star, red] table[x=MeshSize,y=absE_cL2GradPre] {outputs/ddr/voro-trigo-k0-sca1/data_rates.dat};
        \logLogSlopeTriangle{0.90}{0.4}{0.1}{2}{black};
        \addplot [mark=*, blue] table[x=MeshSize,y=absE_cL2GradPre] {outputs/ddr/voro-trigo-k1-sca1/data_rates.dat};
        \logLogSlopeTriangle{0.90}{0.4}{0.1}{3}{black};
        \addplot [mark=o, olive] table[x=MeshSize,y=absE_cL2GradPre] {outputs/ddr/voro-trigo-k2-sca1/data_rates.dat};
        \addplot [mark=star, mark options=solid, red, dashed] table[x=MeshSize,y=absE_L2GradPre] {outputs/ddr/voro-trigo-k0-sca1/data_rates.dat};
        \addplot [mark=*, mark options=solid, blue, dashed] table[x=MeshSize,y=absE_L2GradPre] {outputs/ddr/voro-trigo-k1-sca1/data_rates.dat};
        \addplot [mark=o, mark options=solid, olive, dashed] table[x=MeshSize,y=absE_L2GradPre] {outputs/ddr/voro-trigo-k2-sca1/data_rates.dat};
      \end{loglogaxis}            
    \end{tikzpicture}
    \subcaption{DDR scheme errors on $p$}
  \end{minipage}\\[0.5em]
  \begin{minipage}{0.45\textwidth}
    \begin{tikzpicture}[scale=0.85]
      \begin{loglogaxis} [legend columns=3, legend to name=vem.conv.voro]  
        \logLogSlopeTriangle{0.90}{0.4}{0.1}{1}{black};
        \addplot [mark=star, red] table[x=MeshSize,y=absE_cHcurlVel] {outputs/vem/voro-trigo-k0-sca1/data_rates.dat};
        \addlegendentry{$E^{\rm c}$, $k=0$;}
        \logLogSlopeTriangle{0.90}{0.4}{0.1}{2}{black};
        \addplot [mark=*, blue] table[x=MeshSize,y=absE_cHcurlVel] {outputs/vem/voro-trigo-k1-sca1/data_rates.dat};
        \addlegendentry{$E^{\rm c}$, $k=1$;}
        \logLogSlopeTriangle{0.90}{0.4}{0.1}{3}{black};
        \addplot [mark=o, olive] table[x=MeshSize,y=absE_cHcurlVel] {outputs/vem/voro-trigo-k2-sca1/data_rates.dat};
        \addlegendentry{$E^{\rm c}$, $k=2$}
        \addplot [mark=star, mark options=solid, red, dashed] table[x=MeshSize,y=absE_HcurlVel] {outputs/vem/voro-trigo-k0-sca1/data_rates.dat};
        \addlegendentry{$E^{\rm d}$, $k=0$;}
        \addplot [mark=*, mark options=solid, blue, dashed] table[x=MeshSize,y=absE_HcurlVel] {outputs/vem/voro-trigo-k1-sca1/data_rates.dat};
        \addlegendentry{$E^{\rm d}$, $k=1$;}
        \addplot [mark=o, mark options=solid, olive, dashed] table[x=MeshSize,y=absE_HcurlVel] {outputs/vem/voro-trigo-k2-sca1/data_rates.dat};
        \addlegendentry{$E^{\rm d}$, $k=2$}
      \end{loglogaxis}            
    \end{tikzpicture}
    \subcaption{VEM scheme errors on $\bvec{u}$}
  \end{minipage}
  \begin{minipage}{0.45\textwidth}
    \begin{tikzpicture}[scale=0.85]
      \begin{loglogaxis} 
        \logLogSlopeTriangle{0.90}{0.4}{0.1}{1}{black};
        \addplot [mark=star, red] table[x=MeshSize,y=absE_cL2GradPre] {outputs/vem/voro-trigo-k0-sca1/data_rates.dat};
        \logLogSlopeTriangle{0.90}{0.4}{0.1}{2}{black};
        \addplot [mark=*, blue] table[x=MeshSize,y=absE_cL2GradPre] {outputs/vem/voro-trigo-k1-sca1/data_rates.dat};
        \logLogSlopeTriangle{0.90}{0.4}{0.1}{3}{black};
        \addplot [mark=o, olive] table[x=MeshSize,y=absE_cL2GradPre] {outputs/vem/voro-trigo-k2-sca1/data_rates.dat};
        \addplot [mark=star, mark options=solid, red, dashed] table[x=MeshSize,y=absE_L2GradPre] {outputs/vem/voro-trigo-k0-sca1/data_rates.dat};
        \addplot [mark=*, mark options=solid, blue, dashed] table[x=MeshSize,y=absE_L2GradPre] {outputs/vem/voro-trigo-k1-sca1/data_rates.dat};
        \addplot [mark=o, mark options=solid, olive, dashed] table[x=MeshSize,y=absE_L2GradPre] {outputs/vem/voro-trigo-k2-sca1/data_rates.dat};
      \end{loglogaxis}            
    \end{tikzpicture}
    \subcaption{VEM scheme errors on $p$}
  \end{minipage}
  \caption{Voronoi meshes: errors with respect to  $h$}
  \label{fig:conv.voro}
\end{figure}
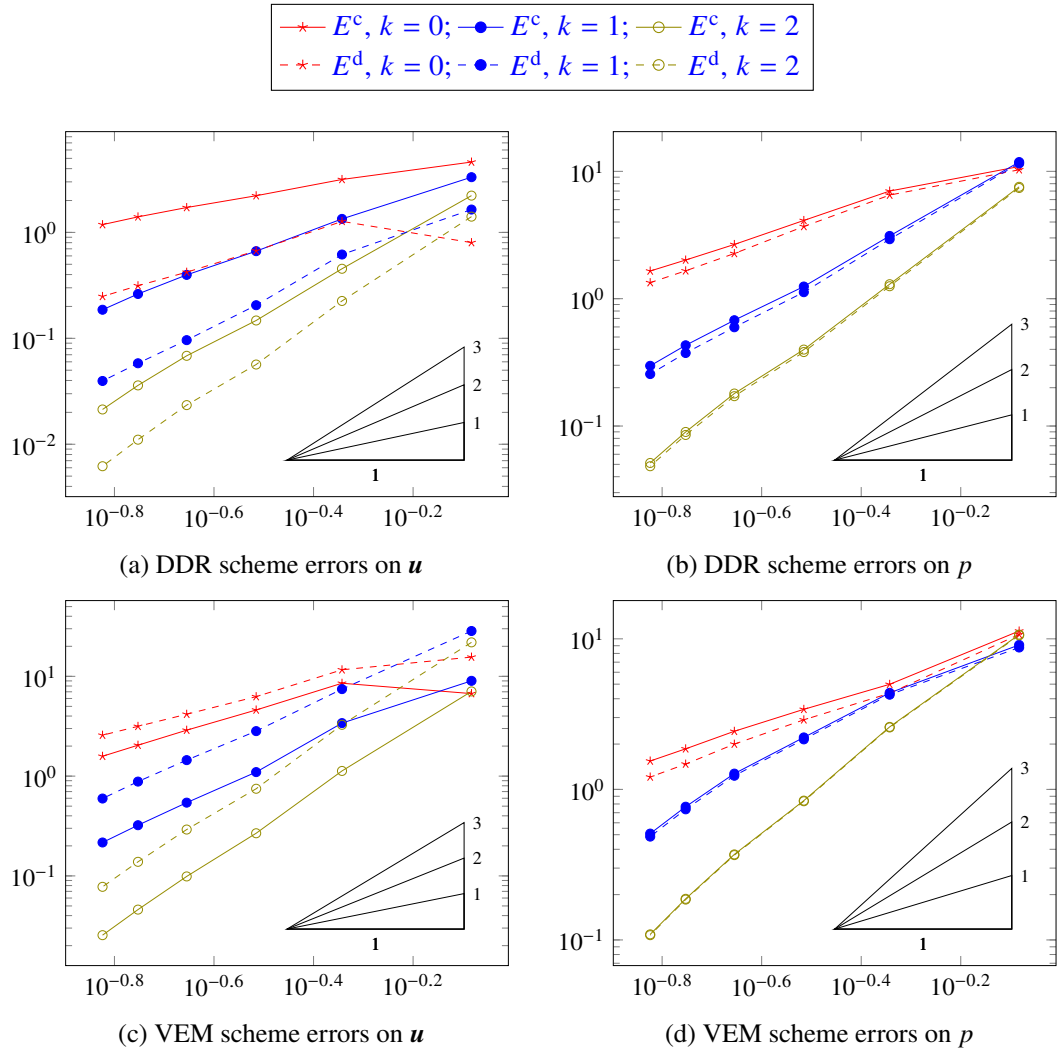

\subsection{Robustness test}\label{sec:test.robust}

In this section we are interested in showing that the proposed schemes are robust with respect to the magnitude of the pressure field $p$.
To achieve this goal, we build the right-hand side of the problem defined in \eqref{eq:strong} in such a way that the velocity field is the same as in the previous example, while the pressure field is
\[
p(x,y,z)=\lambda\sin(2\pi x)\sin(2\pi y)\sin(2\pi z)\,,
\]
where the parameter $\lambda\in\Real^+$ controls the magnitude of the irrotational part of the source term.
We fix $\lambda=10^5$ and we run a convergence test for both discretization schemes on each type of meshes. 
Figures~\ref{fig:conv.tetra5} and~\ref{fig:conv.voro5} present the results of this analysis.
First of all, we observe that the expected convergence rates are recovered for all measured quantities.
Furthermore, if comparing the figures with the corresponding ones for $\lambda=1$, i.e., 
Figure~\ref{fig:conv.tetra} with~\ref{fig:conv.tetra5} and Figure~\ref{fig:conv.voro} with~\ref{fig:conv.voro5},
we can appreciate the robustness of the proposed schemes with respect to the magnitude of $p$.
Specifically, the errors on the velocity field $\bvec{u}$ seem to be unaffected by the magnitude of $p$, as expected due to the pressure robustness of the schemes. Such observation applies also to the discrete pressure error $E^{\rm d}_{p}$, again in line with the theoretical results.
On the other hand, $E^{\rm c}_{p}$ grows approximately by a factor corresponding to the value of $\lambda$, which is again expected since we consider absolute errors and the continuous error estimates depend (linearly) on $p$.

\begin{figure}\centering
  \ref{robustness.tets}
  \vspace{0.50cm}\\
  \begin{minipage}{0.45\textwidth}
    \begin{tikzpicture}[scale=0.85]
      \begin{loglogaxis} [legend columns=3, legend to name=robustness.tets]  
        \logLogSlopeTriangle{0.90}{0.4}{0.1}{1}{black};
        \addplot [mark=star, red] table[x=MeshSize,y=absE_cHcurlVel] {outputs/ddr/tets-trigo-k0-sca1e5/data_rates.dat};
        \addlegendentry{$E^{\rm c}$, $k=0$;}
        \logLogSlopeTriangle{0.90}{0.4}{0.1}{2}{black};
        \addplot [mark=*, blue] table[x=MeshSize,y=absE_cHcurlVel] {outputs/ddr/tets-trigo-k1-sca1e5/data_rates.dat};
        \addlegendentry{$E^{\rm c}$, $k=1$;}
        \logLogSlopeTriangle{0.90}{0.4}{0.1}{3}{black};
        \addplot [mark=o, olive] table[x=MeshSize,y=absE_cHcurlVel] {outputs/ddr/tets-trigo-k2-sca1e5/data_rates.dat};
        \addlegendentry{$E^{\rm c}$, $k=2$}
        \addplot [mark=star, mark options=solid, red, dashed] table[x=MeshSize,y=absE_HcurlVel] {outputs/ddr/tets-trigo-k0-sca1e5/data_rates.dat};
        \addlegendentry{$E^{\rm d}$, $k=0$;}
        \addplot [mark=*, mark options=solid, blue, dashed] table[x=MeshSize,y=absE_HcurlVel] {outputs/ddr/tets-trigo-k1-sca1e5/data_rates.dat};
        \addlegendentry{$E^{\rm d}$, $k=1$;}
        \addplot [mark=o, mark options=solid, olive, dashed] table[x=MeshSize,y=absE_HcurlVel] {outputs/ddr/tets-trigo-k2-sca1e5/data_rates.dat};
        \addlegendentry{$E^{\rm d}$, $k=2$}
      \end{loglogaxis}            
    \end{tikzpicture}
    \subcaption{DDR scheme errors on $\bvec{u}$}
  \end{minipage}
  \begin{minipage}{0.45\textwidth}
    \begin{tikzpicture}[scale=0.85]
      \begin{loglogaxis} 
        \logLogSlopeTriangle{0.90}{0.4}{0.1}{1}{black};
        \addplot [mark=star, red] table[x=MeshSize,y=absE_cL2GradPre] {outputs/ddr/tets-trigo-k0-sca1e5/data_rates.dat};
        \logLogSlopeTriangle{0.90}{0.4}{0.1}{2}{black};
        \addplot [mark=*, blue] table[x=MeshSize,y=absE_cL2GradPre] {outputs/ddr/tets-trigo-k1-sca1e5/data_rates.dat};
        \logLogSlopeTriangle{0.90}{0.4}{0.1}{3}{black};
        \addplot [mark=o, olive] table[x=MeshSize,y=absE_cL2GradPre] {outputs/ddr/tets-trigo-k2-sca1e5/data_rates.dat};
        \addplot [mark=star, mark options=solid, red, dashed] table[x=MeshSize,y=absE_L2GradPre] {outputs/ddr/tets-trigo-k0-sca1e5/data_rates.dat};
        \addplot [mark=*,    mark options=solid, blue, dashed] table[x=MeshSize,y=absE_L2GradPre] {outputs/ddr/tets-trigo-k1-sca1e5/data_rates.dat};
        \addplot [mark=o,    mark options=solid, olive, dashed] table[x=MeshSize,y=absE_L2GradPre] {outputs/ddr/tets-trigo-k2-sca1e5/data_rates.dat};
      \end{loglogaxis}            
    \end{tikzpicture}
    \subcaption{DDR scheme errors on $p$}
  \end{minipage}\\[0.5em]
  \begin{minipage}{0.45\textwidth}
    \begin{tikzpicture}[scale=0.85]
      \begin{loglogaxis} [legend columns=3, legend to name=vem.robustness.tets]  
        \logLogSlopeTriangle{0.90}{0.4}{0.1}{1}{black};
        \addplot [mark=star, red] table[x=MeshSize,y=absE_cHcurlVel] {outputs/vem/tets-trigo-k0-sca1e5/data_rates.dat};
        \addlegendentry{$E^{\rm c}$, $k=0$;}
        \logLogSlopeTriangle{0.90}{0.4}{0.1}{2}{black};
        \addplot [mark=*, blue] table[x=MeshSize,y=absE_cHcurlVel] {outputs/vem/tets-trigo-k1-sca1e5/data_rates.dat};
        \addlegendentry{$E^{\rm c}$, $k=1$;}
        \logLogSlopeTriangle{0.90}{0.4}{0.1}{3}{black};
        \addplot [mark=o, olive] table[x=MeshSize,y=absE_cHcurlVel] {outputs/vem/tets-trigo-k2-sca1e5/data_rates.dat};
        \addlegendentry{$E^{\rm c}$, $k=2$}
        \addplot [mark=star, mark options=solid, red, dashed] table[x=MeshSize,y=absE_HcurlVel] {outputs/vem/tets-trigo-k0-sca1e5/data_rates.dat};
        \addlegendentry{$E^{\rm d}$, $k=0$;}
        \addplot [mark=*, mark options=solid, blue, dashed] table[x=MeshSize,y=absE_HcurlVel] {outputs/vem/tets-trigo-k1-sca1e5/data_rates.dat};
        \addlegendentry{$E^{\rm d}$, $k=1$;}
        \addplot [mark=o, mark options=solid, olive, dashed] table[x=MeshSize,y=absE_HcurlVel] {outputs/vem/tets-trigo-k2-sca1e5/data_rates.dat};
        \addlegendentry{$E^{\rm d}$, $k=2$}
      \end{loglogaxis}            
    \end{tikzpicture}
    \subcaption{VEM scheme errors on $\bvec{u}$}
  \end{minipage}
  \begin{minipage}{0.45\textwidth}
    \begin{tikzpicture}[scale=0.85]
      \begin{loglogaxis} 
        \logLogSlopeTriangle{0.90}{0.4}{0.1}{1}{black};
        \addplot [mark=star, red] table[x=MeshSize,y=absE_cL2GradPre] {outputs/vem/tets-trigo-k0-sca1e5/data_rates.dat};
        \logLogSlopeTriangle{0.90}{0.4}{0.1}{2}{black};
        \addplot [mark=*, blue] table[x=MeshSize,y=absE_cL2GradPre] {outputs/vem/tets-trigo-k1-sca1e5/data_rates.dat};
        \logLogSlopeTriangle{0.90}{0.4}{0.1}{3}{black};
        \addplot [mark=o, olive] table[x=MeshSize,y=absE_cL2GradPre] {outputs/vem/tets-trigo-k2-sca1e5/data_rates.dat};
        \addplot [mark=star, mark options=solid, red,   dashed] table[x=MeshSize,y=absE_L2GradPre] {outputs/vem/tets-trigo-k0-sca1e5/data_rates.dat};
        \addplot [mark=*,    mark options=solid, blue,  dashed] table[x=MeshSize,y=absE_L2GradPre] {outputs/vem/tets-trigo-k1-sca1e5/data_rates.dat};
        \addplot [mark=o,    mark options=solid, olive, dashed] table[x=MeshSize,y=absE_L2GradPre] {outputs/vem/tets-trigo-k2-sca1e5/data_rates.dat};
      \end{loglogaxis}            
    \end{tikzpicture}
    \subcaption{VEM scheme errors on $p$}
  \end{minipage}
  \caption{Tetrahedral meshes, $\lambda=10^5$: errors with respect to  $h$}
  \label{fig:conv.tetra5}
\end{figure}
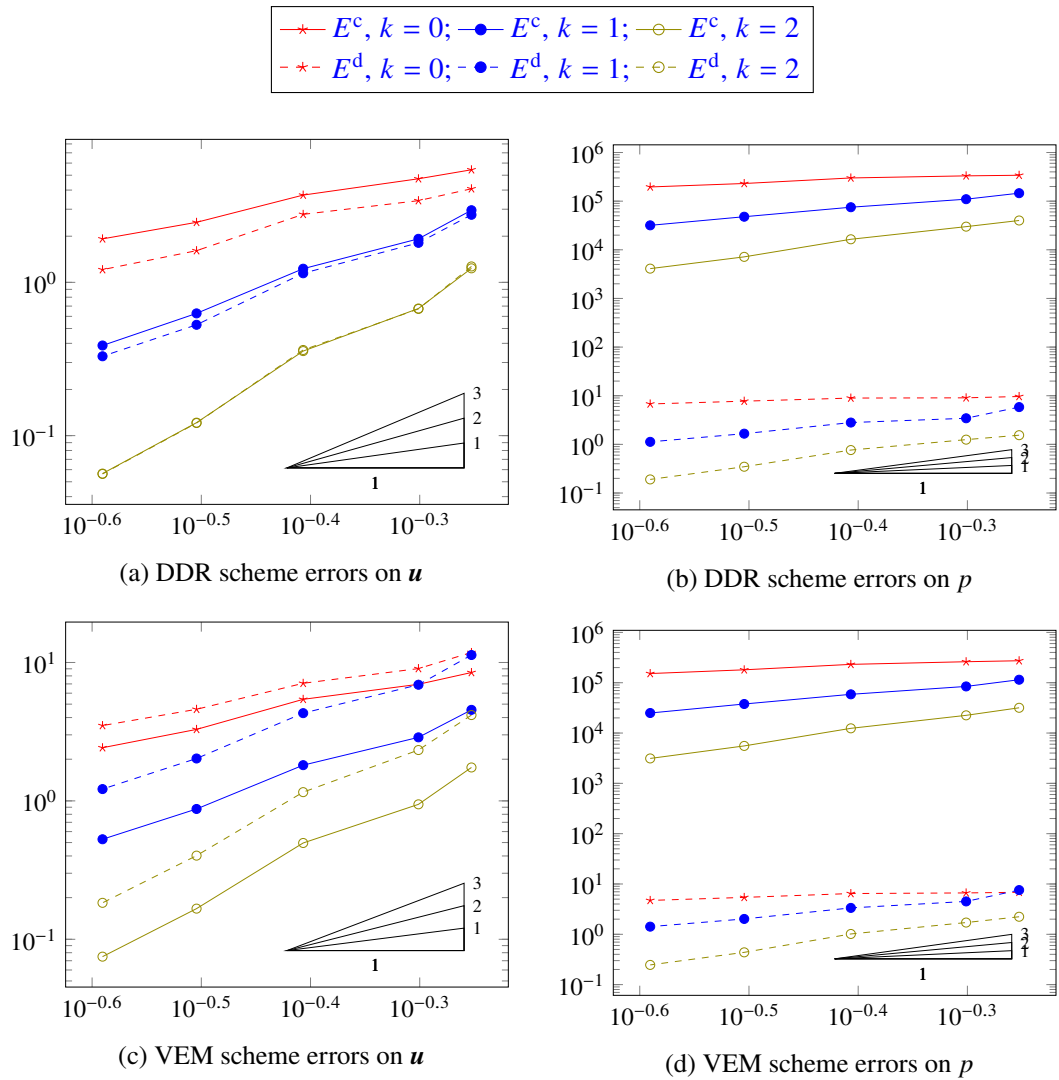

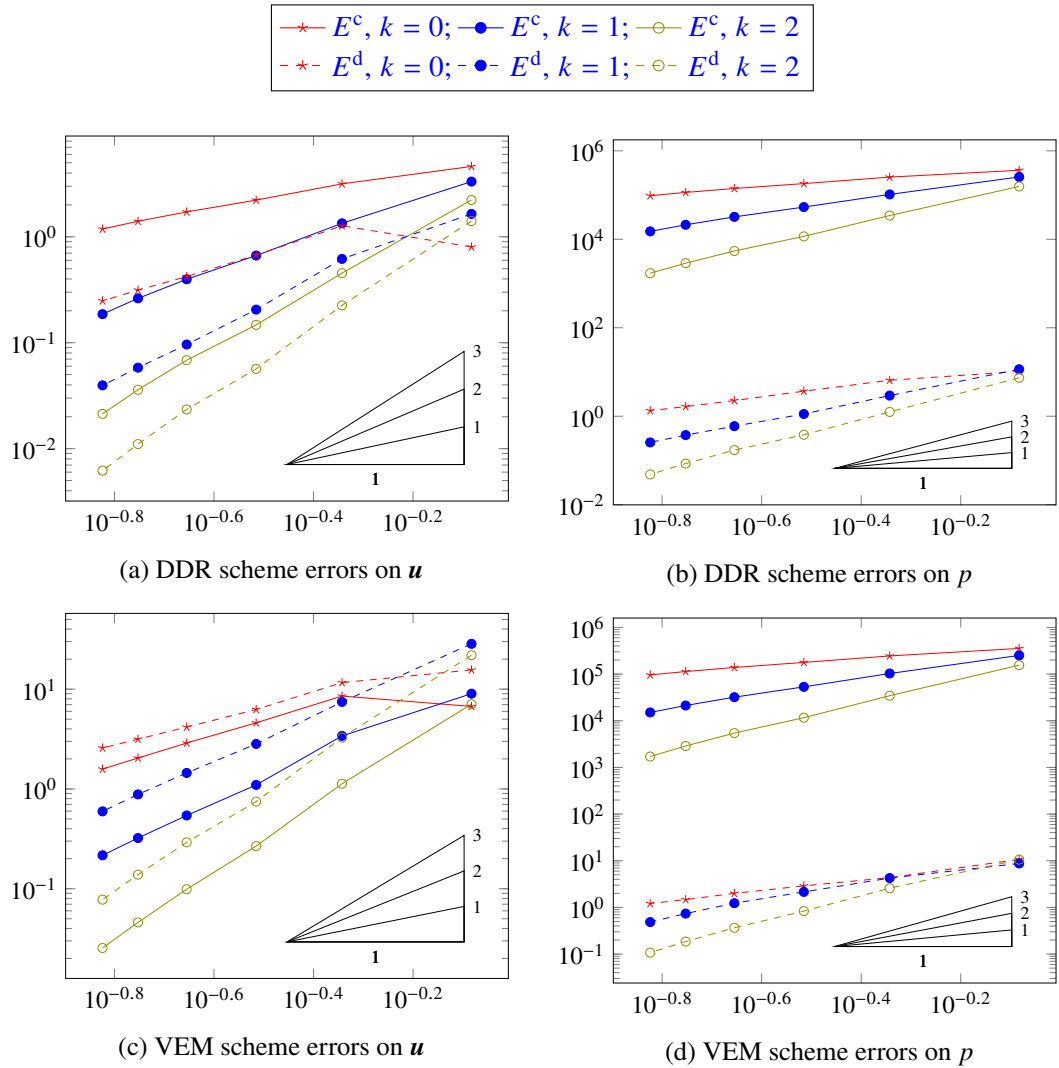
\begin{figure}\centering
  \ref{robustness.voro}
  \vspace{0.50cm}\\
  \begin{minipage}{0.45\textwidth}
    \begin{tikzpicture}[scale=0.85]
      \begin{loglogaxis} [legend columns=3, legend to name=robustness.voro]  
        \logLogSlopeTriangle{0.90}{0.4}{0.1}{1}{black};
        \addplot [mark=star, red] table[x=MeshSize,y=absE_cHcurlVel] {outputs/ddr/voro-trigo-k0-sca1e5/data_rates.dat};
        \addlegendentry{$E^{\rm c}$, $k=0$;}
        \logLogSlopeTriangle{0.90}{0.4}{0.1}{2}{black};
        \addplot [mark=*, blue] table[x=MeshSize,y=absE_cHcurlVel] {outputs/ddr/voro-trigo-k1-sca1e5/data_rates.dat};
        \addlegendentry{$E^{\rm c}$, $k=1$;}
        \logLogSlopeTriangle{0.90}{0.4}{0.1}{3}{black};
        \addplot [mark=o, olive] table[x=MeshSize,y=absE_cHcurlVel] {outputs/ddr/voro-trigo-k2-sca1e5/data_rates.dat};
        \addlegendentry{$E^{\rm c}$, $k=2$}
        \addplot [mark=star, mark options=solid, red, dashed] table[x=MeshSize,y=absE_HcurlVel] {outputs/ddr/voro-trigo-k0-sca1e5/data_rates.dat};
        \addlegendentry{$E^{\rm d}$, $k=0$;}
        \addplot [mark=*, mark options=solid, blue, dashed] table[x=MeshSize,y=absE_HcurlVel] {outputs/ddr/voro-trigo-k1-sca1e5/data_rates.dat};
        \addlegendentry{$E^{\rm d}$, $k=1$;}
        \addplot [mark=o, mark options=solid, olive, dashed] table[x=MeshSize,y=absE_HcurlVel] {outputs/ddr/voro-trigo-k2-sca1e5/data_rates.dat};
        \addlegendentry{$E^{\rm d}$, $k=2$}
      \end{loglogaxis}            
    \end{tikzpicture}
    \subcaption{DDR scheme errors on $\bvec{u}$}
  \end{minipage}
  \begin{minipage}{0.45\textwidth}
    \begin{tikzpicture}[scale=0.85]
      \begin{loglogaxis} 
        \logLogSlopeTriangle{0.90}{0.4}{0.1}{1}{black};
        \addplot [mark=star, red] table[x=MeshSize,y=absE_cL2GradPre] {outputs/ddr/voro-trigo-k0-sca1e5/data_rates.dat};
        \logLogSlopeTriangle{0.90}{0.4}{0.1}{2}{black};
        \addplot [mark=*, blue] table[x=MeshSize,y=absE_cL2GradPre] {outputs/ddr/voro-trigo-k1-sca1e5/data_rates.dat};
        \logLogSlopeTriangle{0.90}{0.4}{0.1}{3}{black};
        \addplot [mark=o, olive] table[x=MeshSize,y=absE_cL2GradPre] {outputs/ddr/voro-trigo-k2-sca1e5/data_rates.dat};
        \addplot [mark=star, mark options=solid, red,   dashed] table[x=MeshSize,y=absE_L2GradPre] {outputs/ddr/voro-trigo-k0-sca1e5/data_rates.dat};
        \addplot [mark=*,    mark options=solid, blue,  dashed] table[x=MeshSize,y=absE_L2GradPre] {outputs/ddr/voro-trigo-k1-sca1e5/data_rates.dat};
        \addplot [mark=o,    mark options=solid, olive, dashed] table[x=MeshSize,y=absE_L2GradPre] {outputs/ddr/voro-trigo-k2-sca1e5/data_rates.dat};
      \end{loglogaxis}            
    \end{tikzpicture}
    \subcaption{DDR scheme errors on $p$}
  \end{minipage}\\[0.5em]
  \begin{minipage}{0.45\textwidth}
    \begin{tikzpicture}[scale=0.85]
      \begin{loglogaxis} [legend columns=3, legend to name=vem.robustness.voro]  
        \logLogSlopeTriangle{0.90}{0.4}{0.1}{1}{black};
        \addplot [mark=star, red] table[x=MeshSize,y=absE_cHcurlVel] {outputs/vem/voro-trigo-k0-sca1e5/data_rates.dat};
        \addlegendentry{$E^{\rm c}$, $k=0$;}
        \logLogSlopeTriangle{0.90}{0.4}{0.1}{2}{black};
        \addplot [mark=*, blue] table[x=MeshSize,y=absE_cHcurlVel] {outputs/vem/voro-trigo-k1-sca1e5/data_rates.dat};
        \addlegendentry{$E^{\rm c}$, $k=1$;}
        \logLogSlopeTriangle{0.90}{0.4}{0.1}{3}{black};
        \addplot [mark=o, olive] table[x=MeshSize,y=absE_cHcurlVel] {outputs/vem/voro-trigo-k2-sca1e5/data_rates.dat};
        \addlegendentry{$E^{\rm c}$, $k=2$}
        \addplot [mark=star, mark options=solid, red, dashed] table[x=MeshSize,y=absE_HcurlVel] {outputs/vem/voro-trigo-k0-sca1e5/data_rates.dat};
        \addlegendentry{$E^{\rm d}$, $k=0$;}
        \addplot [mark=*, mark options=solid, blue, dashed] table[x=MeshSize,y=absE_HcurlVel] {outputs/vem/voro-trigo-k1-sca1e5/data_rates.dat};
        \addlegendentry{$E^{\rm d}$, $k=1$;}
        \addplot [mark=o, mark options=solid, olive, dashed] table[x=MeshSize,y=absE_HcurlVel] {outputs/vem/voro-trigo-k2-sca1e5/data_rates.dat};
        \addlegendentry{$E^{\rm d}$, $k=2$}
      \end{loglogaxis}            
    \end{tikzpicture}
    \subcaption{VEM scheme errors on $\bvec{u}$}
  \end{minipage}
  \begin{minipage}{0.45\textwidth}
    \begin{tikzpicture}[scale=0.85]
      \begin{loglogaxis} 
        \logLogSlopeTriangle{0.90}{0.4}{0.1}{1}{black};
        \addplot [mark=star, red] table[x=MeshSize,y=absE_cL2GradPre] {outputs/vem/voro-trigo-k0-sca1e5/data_rates.dat};
        \logLogSlopeTriangle{0.90}{0.4}{0.1}{2}{black};
        \addplot [mark=*, blue] table[x=MeshSize,y=absE_cL2GradPre] {outputs/vem/voro-trigo-k1-sca1e5/data_rates.dat};
        \logLogSlopeTriangle{0.90}{0.4}{0.1}{3}{black};
        \addplot [mark=o, olive] table[x=MeshSize,y=absE_cL2GradPre] {outputs/vem/voro-trigo-k2-sca1e5/data_rates.dat};
        \addplot [mark=star, mark options=solid, red,   dashed] table[x=MeshSize,y=absE_L2GradPre] {outputs/vem/voro-trigo-k0-sca1e5/data_rates.dat};
        \addplot [mark=*,    mark options=solid, blue,  dashed] table[x=MeshSize,y=absE_L2GradPre] {outputs/vem/voro-trigo-k1-sca1e5/data_rates.dat};
        \addplot [mark=o,    mark options=solid, olive, dashed] table[x=MeshSize,y=absE_L2GradPre] {outputs/vem/voro-trigo-k2-sca1e5/data_rates.dat};
      \end{loglogaxis}            
    \end{tikzpicture}
    \subcaption{VEM scheme errors on $p$}
  \end{minipage}
  \caption{Voronoi meshes, $\lambda=10^5$: errors with respect to  $h$}
  \label{fig:conv.voro5}
\end{figure}

\section{Bridging the VEM and DDR approaches}
\label{sec:bridge}

In this section we bridge the VEM and DDR approaches. The bridge VEM$\to$DDR consists in constructing, for each virtual space $V^\bullet_l$, $(\bullet,l)\in\{(\node,k+1),(\edge,k),(\face,k)\}$, a fully discrete space $\underline{V}^\bullet_l$ of vectors of polynomials, with the DoFs (interpreted as providing polynomials through the moments they describe) creating isomorphisms $V^\bullet_l\stackrel{\approx}{\to}\underline{V}^\bullet_l$; we also construct discrete operators, between the various $\underline{V}^\bullet_l$ spaces, that commute with DoF maps and the corresponding operators between the virtual spaces $V^\bullet_l$.
Similarly, to create the bridge DDR$\to$VEM, we construct for each DDR space $\underline{X}_\bullet^k$, $\bullet\in\{\GRAD,\CURL,\DIV\}$, a space of virtual functions $X_\bullet^k$ with DoFs that create isomorphisms $X_\bullet^k\stackrel{\approx}{\to}\underline{X}_\bullet^k$, and such that the discrete DDR operators commute, through these isomorphisms, with the corresponding continuous differential operators between the virtual spaces.

\subsection{DDR interpretation of the VEM scheme}

The DDR interpretation of the VEM approach requires us to identify the fully discrete spaces, vector calculus operators, and potentials corresponding, respectively, to the VEM spaces and projections. This identification leads, in particular, to the commutative diagram \eqref{eq:commutatif.diagram.vem_to_ddr}, in which the vertical arrows are the isomorphisms defined by the DoFs in each virtual space, while $\uvec{G}_T^{\node,k}$ and $\uvec{C}_T^{\edge,k}$ are, respectively, the restrictions to a mesh element $T\in\Th$ of $\uvec{G}_h^{\node,k}$ and $\uvec{C}_h^{\edge,k}$ (see \eqref{eq:vem.Gh} and \eqref{eq:vem.Ch} below).
\begin{equation}\label{eq:commutatif.diagram.vem_to_ddr}
  \begin{tikzcd}
    V^\node_{k+1}(T)\arrow{r}{\GRAD}\arrow{d}{\rotatebox{90}{$\approx$}~({\rm DoF})}
    & V^\edge_k(T)\arrow{r}{\CURL}\arrow{d}{\rotatebox{90}{$\approx$}~({\rm DoF})}
    & V^\face_k(T)\arrow{r}{\DIV}\arrow{d}{\rotatebox{90}{$\approx$}~({\rm DoF})}
    & \Poly{k}(T)\arrow{d}{\rm Id}
    \\
    \underline{V}^\node_{k+1}(T)\arrow{r}{\uvec{G}_T^{\node,k}} & \underline{V}^\edge_k(T)\arrow{r}{\uvec{C}_T^{\edge,k}} & \underline{V}^\face_k(T)\arrow{r}{D^{\face,k}_T} & \Poly{k}(T)
  \end{tikzcd}
\end{equation}
Notice that the DoF maps are also cochain maps (i.e., they commute with the continuous/discrete differential operators), which trivially yields an isomorphism between the cohomologies of the virtual and discrete complexes.

\subsubsection{Spaces}

The fully discrete counterparts of the global nodal, edge, and face VEM spaces are, respectively,
\begin{equation}\label{eq:uV.node}
  \underline{V}_{k+1}^{\node}
  \coloneq\Big\{
  \begin{aligned}[t]
    &\underline{q}_h = \big(
    (\bvec{G}_{q,T})_{T\in\Th}, (\bvec{G}_{q,F})_{F\in\Fh}, (q_E)_{E\in\Eh}, (q_\nu)_{\nu\in\Vh}
    \big)\st
    \\
    &\qquad\text{$\bvec{G}_{q,T}\in\cRoly{k}(T)$ for all $T\in\Th$,
    $\bvec{G}_{q,F}\in\cRoly{\beta_F+1}(F)$ for all $F\in\Fh$,}
    \\
    &\qquad\text{$q_E\in\Poly{k-1}(E)$ for all $E\in\Eh$,
      and $q_\nu\in\Real$ for all $\nu\in\Vh$}
    \Big\},
  \end{aligned}
\end{equation}
\begin{equation*} 
  \underline{V}_k^{\edge}
  \coloneq\Big\{
  \begin{aligned}[t]
    &\uvec{v}_h = \big(
    (\bvec{C}_{\bvec{v},T},\bvec{v}_{\cvec{R},T}^\compl)_{T\in\Th},
    (C_{\bvec{v},F},\bvec{v}_{\cvec{R},F}^\compl)_{F\in\Fh},
    (v_E)_{E\in\Eh}
    \big)\st
    \\
    &\qquad\text{$\bvec{C}_{\bvec{v},T}\in\cGoly{k+1}(T)$ and $\bvec{v}_{\cvec{R},T}^\compl\in\cRoly{k}(T)$ for all $T\in\Th$,}
    \\
    &\qquad\text{$C_{\bvec{v},F}\in\Poly[0]{k}(F)$ and $\bvec{v}_{\cvec{R},F}^\compl\in\cRoly{\beta_F+1}(F)$ for all $F\in\Fh$},
    \\
    &\qquad\text{and $v_E\in\Poly{k}(E)$ for all $E\in\Eh$}
    \Big\},
  \end{aligned}
\end{equation*}
and
\begin{equation} \label{eq:uV.face}
  \underline{V}_k^{\face}
  \coloneq\Big\{
  \begin{aligned}[t]
    &\uvec{w}_h = \big(
    (D_{\bvec{w},T},\bvec{w}_{\cvec{G},T}^\compl)_{T\in\Th},
    (w_F)_{F\in\Fh}
    \big)\st
    \\
    &\qquad\text{$D_{\bvec{w},T}\in\Poly[0]{k}(T)$ and $\bvec{w}_{\cvec{G},T}^\compl\in\cGoly{k+1}(T)$ for all $T\in\Th$,}
    \\
    &\qquad\text{and $w_F\in\Poly{k}(F)$ for all $F\in\Fh$}
    \Big\},
  \end{aligned}
\end{equation}
where we have used the customary underlined notation to recall the fact that these spaces are spanned by vectors of polynomials.
Notice that, unlike \eqref{glo-n}, \eqref{eq:uV.node} does not incorporate the zero-average condition over $\Omega$.
For a comparison of the discrete spaces corresponding to the DDR and VEM complexes, see Table \ref{tab:discrete.spaces}.
\begin{remark}[Link between polynomial components and degrees of freedom]
  In $\underline{V}^\node_{k+1}$, the component $\bvec{G}_{q,T}$ is associated to the DoFs \eqref{dof-3dnk-4},
  $\bvec{G}_{q,F}$ corresponds to \eqref{dof-3dnk-3}, $q_E$ to \eqref{dof-3dnk-2} and $q_\nu$ to \eqref{dof-3dnk-1}.
  The link between components in $\underline{V}^\edge_k$ and DoFs is as follows: $\bvec{C}_{\bvec{v},T}$ comes from \eqref{dof-3dek-5},
  $\bvec{v}^\compl_{\cvec{R},T}$ from \eqref{dof-3dek-4}, $C_{\bvec{v},F}$ from \eqref{dof-3dek-3}, $\bvec{v}^\compl_{\cvec{R},F}$ from \eqref{dof-3dek-2},
  and $v_E$ from \eqref{dof-3dek-1}.
  Finally, for the face space: $D_{\bvec{w},T}$ is generated by \eqref{dof-3dfk-2}, $\bvec{w}_{\cvec{G},T}^\compl$ from \eqref{dof-3dfk-3}, and $w_F$ from \eqref{dof-3dfk-1}.
\end{remark}

For $(\bullet,l)\in\left\{(\node,k+1),(\edge,k),(\face,k)\right\}$ and any geometric entity $Y\in\Th\cup\Fh\cup\Eh$ appearing in the definition of a fully discrete VEM space $\underline{V}_l^\bullet$, we denote by $\underline{V}_l^\bullet(Y)$ the restriction of $\underline{V}_l^\bullet$ to $Y$ collecting the polynomial spaces attached to $Y$ and its boundary.

\begin{table}
  \renewcommand*{\arraystretch}{1.2}
  \centering
  \begin{tabular}{ccccc}
    \toprule
    Space & $V\in\Vh$ & $E\in\Eh$ & $F\in\Fh$ & $T\in\Th$ \\
    \midrule
    \multicolumn{5}{c}{DDR} \\
    \midrule
    $\Xgrad{h}$ & $\Real$ & $\Poly{k-1}(E)$ & $\Poly{k-1}(F)$ & $\Poly{k-1}(T)$ \\
    $\Xcurl{h}$ & & $\Poly{k}(E)$ & $\Roly{k-1}(F)\oplus\cRoly{k}(F)$ & $\Roly{k-1}(T)\oplus\cRoly{k}(T)$ \\
    $\Xdiv{h}$ & & & $\Poly{k}(F)$ & $\Goly{k-1}(T)\oplus\cGoly{k}(T)$ \\
    $\Poly{k}(\Th)$ & & & & $\Poly{k}(T)$ \\
    \midrule
    \multicolumn{5}{c}{VEM} \\
    \midrule
    $\underline{V}_{k+1}^{\node}$ & $\Real$ & $\Poly{k-1}(E)$ & $\cRoly{\beta_F+1}(F)$ & $\cRoly{k}(T)$ \\
    $\underline{V}_k^{\edge}$ & & $\Poly{k}(E)$ & $\Poly[0]{k}(F)\times\cRoly{\beta_F+1}(F)$ & $\cGoly{k+1}(T)\times\cRoly{k}(T)$ \\
    $\underline{V}_k^{\face}$ & & & $\Poly{k}(F)$ & $\Poly[0]{k}(T)\times\cGoly{k+1}(T)$ \\
    $\Poly{k}(\Th)$ & & & & $\Poly{k}(T)$ \\
    \bottomrule
  \end{tabular}
  \caption{Comparison of the polynomial components of the discrete spaces for the DDR and VEM discrete complexes.
  The direct sum symbol ($\oplus$) replaces the Cartesian product ($\times$) when the polynomial components attached to a mesh entity refer to homogeneous quantities and not to, e.g., functions and derivatives.\label{tab:discrete.spaces}}
\end{table}

\subsubsection{Discrete vector calculus operators and potentials}

\paragraph{Nodal space}
Given $\underline{q}_h\in\underline{V}_{k+1}^{\node}$, we denote by $q_{\Eh}\in\Poly[\rm c]{k+1}(\Eh)$ the unique function on the edge skeleton of the mesh such that $\lproj{k-1}{E}q_{\Eh} = q_E$ for all $E\in\Eh$ and $q_{\Eh}(\bvec{c}_\nu) = q_\nu$ for all $\nu\in\Vh$.
We then define the \emph{discrete gradient operator} $\uvec{G}_h^{\node,k}:\underline{V}_{k+1}^{\node}\to\underline{V}_k^{\edge}$ such that, for all $\underline{q}_h\in\underline{V}_{k+1}^{\node}$,
\begin{equation}\label{eq:vem.Gh}    
  \uvec{G}_h^{\node,k}\underline{q}_h
  \coloneq
  \big(  
  (\bvec{0},\bvec{G}_{q,T})_{T\in\Th},
  (0,\bvec{G}_{q,F}) )_{F\in\Fh},
  ((q_{\Eh})_{|E}')_{E\in\Eh}
  \big).
\end{equation}

\paragraph{Edge space} Given a mesh face $F\in\Fh$, the \emph{edge serendipity operator} $\bvec{\Pi}_{S,F}^{\edge,k}:\underline{V}_k^{\edge}(F)\to\vPoly{k}(F)$ is such that, for all $\uvec{v}_F\in\underline{V}_k^{\edge}(F)$,
\[
\begin{alignedat}{2}
  \sum_{E\in\EF}\omega_{FE}\int_E(\bvec{\Pi}_{S,F}^{\edge,k}\uvec{v}_F\cdot\tangent_E)(\bvec{w}\cdot\tangent_E)
  &=\sum_{E\in\EF}\omega_{FE}\int_E v_E(\bvec{w}\cdot\tangent_E)
  &\qquad&\forall\bvec{w}\in\Goly{k}(F),
  \\
  \sum_{E\in\EF}\omega_{FE}\int_E\bvec{\Pi}_{S,F}^{\edge,k}\uvec{v}_F\cdot\tangent_E
  &=\sum_{E\in\EF}\omega_{FE}\int_E v_E.
  \\
  \int_F\ROT_F(\bvec{\Pi}_{S,F}^{\edge,k}\uvec{v}_F)~q
  &= \int_F C_{\bvec{v},F}~q
  &\qquad&\forall q\in\Poly[0]{k-1}(F),
  \\
  \int_F\bvec{\Pi}_{S,F}^{\edge,k}\uvec{v}_F\cdot\bvec{w}
  &= \int_F\bvec{v}_{\cvec{R},F}^\compl\cdot\bvec{w}
  &\qquad&\forall\bvec{w}\in\cRoly{\beta_F+1}(F),
\end{alignedat}
\]
Recalling the face Raviart--Thomas space $\RT{k+1}(F)\coloneq\Roly{k}(F)\oplus\cRoly{k+1}(F)$, the \emph{tangent trace} $\bvec{\gamma}_F^{\edge,k+1}:\underline{V}_k^{\edge}(F)\to\RT{k+1}(F)$ is such that, for all $\uvec{v}_F\in\underline{V}_k^{\edge}(F)$ and all $(r_F,\bvec{w}_F)\in\Poly[0]{k+1}(F)\times\cRoly{k+1}(F)$,
\begin{equation}\label{eq:vem:gammaFek}
  \int_F\bvec{\gamma}_F^{\edge,k+1}\uvec{v}_F\cdot(\VROT_F r_F + \bvec{w}_F)
  = \int_F C_{\bvec{v},F}~r_F
  + \sum_{E\in\EF}\omega_{FE}\int_E v_E~r_F
  + \int_F\bvec{\Pi}_{\cvec{S},F}^{\edge,k}\uvec{v}_F\cdot\bvec{w}_F.
\end{equation}
Since $C_{\bvec{v},F}$ only encodes the zero-averaged component of the discrete face curl, we reconstruct a complete \emph{face curl} $C_F^{\edge,k}:\underline{V}^\edge_k(F)\to \Poly{k}(F)$ by using the tangential components to the edges: For all $\uvec{v}_F\in\underline{V}_k^\edge(F)$,
\[
\int_F C_F^{\edge,k}\uvec{v}_F\, r_F = \int_F C_{\bvec{v},F} r_F + \frac{1}{|F|}\bigg(\int_F r_F\bigg) \bigg(\sum_{E\in\EF}\omega_{FE}\int_E v_E\bigg)\qquad\forall r_F\in\Poly{k}(F).
\]

For all $T\in\Th$, the \emph{element potential} $\bvec{P}_T^{\edge,k}:\underline{V}_k^{\edge}(T)\to\vPoly{k}(T)$ is such that, for all $\uvec{v}_T\in\underline{V}_k^{\edge}(T)$ and all $(\bvec{w}_T,\bvec{z}_T)\in\cGoly{k+1}(T)\times\cRoly{k}(T)$,
\[
\int_T\bvec{P}_T^{\edge,k}\uvec{v}_T\cdot(\CURL\bvec{w}_T + \bvec{z}_T)
= \int_T\bvec{C}_{\bvec{v},T}\cdot\bvec{w}_T
- \sum_{F\in\FT}\omega_{TF}\int_F\bvec{\gamma}_F^{\edge,k+1}\uvec{v}_T\cdot(\bvec{w}_T\times\normal_F)
+ \int_T\bvec{v}_{\cvec{R},T}^\compl\cdot\bvec{z}_T.
\]
Finally, we define the \emph{discrete curl} $\uvec{C}_h^{\edge,k}:\underline{V}_k^{\edge}\to\underline{V}_k^{\face}$ such that, for all $\uvec{v}_h\in\underline{V}_k^{\edge}$,
\begin{equation}\label{eq:vem.Ch}
  \uvec{C}_h^{\edge,k}\uvec{v}_h\coloneq\big(
  (\bvec{0}, \bvec{C}_{\bvec{v},T})_{T\in\Th},
  (C_F^{\edge,k}\uvec{v}_F)_{F\in\Fh}
  \big).
\end{equation}

\begin{remark}[Face and element gradients]
  Face and element gradients on full polynomial spaces in the spirit of \eqref{eq:cGF} and \eqref{new:X1} can be obtained setting $\bvec{G}_F^{\node,k}\coloneq\vlproj{k}{F}\circ\bvec{\gamma}_F^{\edge,k+1}\circ\uvec{G}_F^{\node,k}$ for all $F\in\FT$ and $\bvec{G}_T^{\node,k}\coloneq\bvec{P}_T^{\edge,k}\circ\uvec{G}_T^{\node,k}$ for all $T\in\Th$ where, given $Y\in\Th\cup\Fh$, $\uvec{G}_Y^{\node,k}$ denotes the restriction to $Y$ of $\uvec{G}_h^{\node,k}$ (see \eqref{eq:vem.Gh}).
  For $Y\in\Th\cup\Fh$, $\bvec{G}_Y^{\node,k}$ yields the exact gradient when applied to the interpolants on $\underline{V}_k^{\node}(Y)$ of functions in $\Poly{k+1}(Y)$.
\end{remark}

\paragraph{Face space}

The \emph{element potential} is $\bvec{P}_T^{\face,k}:\underline{V}_k^{\face}(T)\to\vPoly{k}(T)$ such that, for  all $\uvec{w}_T\in\underline{V}_k^{\face}(T)$ and all $(r_T,\bvec{v}_T)\in\Poly[0]{k+1}(T)\times\cGoly{k}(T)$,
\[
\int_T\bvec{P}_T^{\face,k}\uvec{w}_T\cdot(\GRAD r_T + \bvec{v}_T)
= -\int_T D_{\bvec{w},T}~r_T
+ \sum_{F\in\FT}\omega_{TF}\int_Fw_F~r_T
+ \int_T\bvec{w}_{\cvec{G},T}^\compl\cdot\bvec{v}_T.
\]
The \emph{discrete divergence} is the operator $D^{\face,k}_h:\underline{V}^\face_k\to \Poly{k}(\Th)$ such that, for all $\uvec{w}_h\in\underline{V}^\face_k$ and all $T\in\Th$, $(D^{\face,k}_h\uvec{w}_h)_{|T}=D^{\face,k}_T\uvec{w}_T\in\Poly{k}(T)$ where, recalling that $D_{\bvec{w},T}$ only encodes the zero-average component of the discrete divergence and, following the same idea as for $C_F^{\edge,k}$, we define $D^{\face,k}_T\uvec{w}_T$ by
\[
\int_T D^{\face,k}_T\uvec{w}_T\,q = \int_T D_{\bvec{w},T} q + \frac{1}{|T|}\bigg(\int_T q\bigg) \bigg(\sum_{F\in\FT}\omega_{TF}\int_F w_F\bigg)\qquad\forall q\in\Poly{k}(T).
\]

\subsubsection{Discrete $L^2$-products}

For $\bullet\in\{\edge,\face\}$, the discrete $L^2$-product $(\cdot,\cdot)_{\bullet,h}$ in $\underline{V}_k^{\bullet}$ is defined as follows:
For all $\uvec{v}_h,\uvec{w}_h\in\underline{V}_k^\bullet$,
\[
(\uvec{v}_h,\uvec{w}_h)_{\bullet,h}\coloneq\sum_{T\in\Th}(\uvec{v}_T,\uvec{w}_T)_{\bullet,T}
\]
where, for all $T\in\Th$,
\[
(\uvec{v}_T,\uvec{w}_T)_{\bullet,T}
\coloneq\int_T\bvec{P}_T^{\bullet,k}\uvec{v}_T\cdot\bvec{P}_T^{\bullet,k}\uvec{w}_T
+ s_{\bullet,T}\big(\uvec{v}_T - \uvec{I}_{k,T}^\bullet(\bvec{P}_T^{\bullet,k}\uvec{v}_T), \uvec{w}_T - \uvec{I}_{k,T}^\bullet(\bvec{P}_T^{\bullet,k}\uvec{w}_T)\big),
\]
where $s_{\bullet,T}:\underline{V}_k^\bullet(T)\times\underline{V}_k^\bullet(T)\to\Real$ is, for example, the stabilization bilinear form corresponding to the one defined by \eqref{eq:vem:sT} while $\uvec{I}_{k,T}^\bullet$ is the natural interpolator on $\underline{V}_k^\bullet(T)$ obtained assembling the $L^2$-orthogonal projections onto each component space.

\subsection{VEM interpretation of the DDR scheme}

The main steps in interpreting the DDR as a VEM scheme is the introduction of local virtual element spaces, the associated DoFs, and the projectors corresponding to the DDR discrete operators and potentials. The virtual spaces $X^\bullet_l$, $(\bullet,l)\in\{(\node,k+1),(\edge,k),(\face,k)\}$, and continuous differential operators are then linked to the DDR spaces $\underline{X}^\bullet_l$, $\bullet\in\{\GRAD,\CURL,\DIV\}$, and discrete operators through the commuting diagram \eqref{eq:commutatif.diagram.ddr_to_vem}, in which the vertical arrows are the isomorphisms defined by the DoFs in each virtual space, and $\uGT$ and $\uCT$ are, respectively, the restrictions to $T\in\Th$ of $\uGh$ and $\uCh$ (see \eqref{eq:uGh} and \eqref{eq:uCh}).
\begin{equation}\label{eq:commutatif.diagram.ddr_to_vem}
  \begin{tikzcd}
    X^\node_{k+1}(T)\arrow{r}{\GRAD}\arrow{d}{\rotatebox{90}{$\approx$}~({\rm DoF})}
    & X^\edge_k(T)\arrow{r}{\CURL}\arrow{d}{\rotatebox{90}{$\approx$}~({\rm DoF})}
    & X^\face_k(T)\arrow{r}{\DIV}\arrow{d}{\rotatebox{90}{$\approx$}~({\rm DoF})}
    & \Poly{k}(T)\arrow{d}{\rm Id}
    \\
    \Xgrad{T}\arrow{r}{\uGT} & \Xcurl{T}\arrow{r}{\uCT} & \Xdiv{T}\arrow{r}{\DT} & \Poly{k}(T)
  \end{tikzcd}
\end{equation}
As for diagram \eqref{eq:commutatif.diagram.vem_to_ddr}, since the DoF maps are cochain maps, the cohomologies of the virtual and discrete complexes in diagram \eqref{eq:commutatif.diagram.ddr_to_vem} are isomorphic.

\subsubsection{Virtual spaces and degrees of freedom}

\paragraph{Nodal space.}  
We start by defining the nodal space on faces
\begin{equation}\label{DDR2VEM1}
\begin{aligned}
{X}_{k+1}^{\node}(F) \coloneq 
\Big\{ q \in H^1(F) : \: & q_{|E} \in \Poly{k+1}(E) \quad\forall E \in \EF, \: \Delta_F q \in \Poly{k+1}(F) , \\
& \int_{F} (\GRAD_F q - \vlproj{k}{F} (\GRAD_F q) )\cdot \xxf\,p\df=0 \quad\forall p\in \Poly{k-1 | k+1}(F) \Big\} .
\end{aligned}
\end{equation}
where $\Poly{k-1 | k+1}(F)$ is any space such that $\Poly{k+1}(F) = \Poly{k-1}(F) \oplus \Poly{k-1 | k+1}(F)$.
The local space on an element $T$ is defined by
$$
\begin{aligned}
X^{\node}_{k+1}(T) \coloneq \Big\{ q \in H^1(T) \: : 
\:& q_{|F}\in X^{\node}_{k+1}(F)\quad\forall F\in\FT , \: \,\Delta\,q\in\Poly{k+1}(T) , \\
& \int_{T} (\GRAD q - \vlproj{k}{T} (\GRAD q) )\cdot \xxE \,p\df=0 \quad\forall p\in \Poly{k-1 | k+1}(T)
\Big\} .
\end{aligned}
$$
The DoFs are chosen as follows:
\begin{align*}
& \bullet \mbox{for any vertex $\nu\in\VT$, the nodal value $q(\bvec{c}_\nu),$ } \\
& \bullet \mbox{if $k\ge 1$: \ for each edge $E\in\ET$,  $ \int_E q\, p\ds\quad\forall p\in\Poly{k-1}(E),$} \\
&\bullet
\mbox{if $k\ge 1$: \  for each face $F\in\FT$,  $\int_{F} q \: p \df \quad \forall p \in\Poly{k-1}(F)$,}\\
&\bullet\mbox{if $k\ge 1$: \  $\int_{T} q \: p \dPP \quad \forall p \in\Poly{k-1}(T).$}
\end{align*}

\paragraph{Edge space.} 
The edge spaces on faces and elements are given by 
\begin{multline*}
  \bvec{X}_{k}^{\edge}(F) \coloneq\! \Big\{
  \vv\in \bvec{L}^2(F)\st
  \DIV_F\vv \in \Poly{k+1}(F), \, \ROT_F\vv \in \Poly{k}(F),\,
  \vv \cdot \tt_E \in \Poly{k}(E) \quad \forall E \in \EF  \\
  \int_{F} (\vv - \vlproj{k}{F} \vv )  \cdot \xxf \,p \df = 0 \quad\forall p\in \Poly{k-1 | k+1}(F)
  \Big\} .
\end{multline*}
$$
\begin{aligned}
\bvec{X}^{\edge}_{k}(T) \coloneq \Big\{
    \vv\in\bvec{L}^2(T)\::\:&\DIV\vv\in \Poly{k+1}(T), \: \CURL(\CURL\vv) \in \vPoly{k}(T), \\
    &  \vv_{{\rm t},F} \in \bvec{X}^{\edge}_{k}(F)\ \forall F\in\FT,\; \vv\cdot\tt_E \mbox{ single valued on each edge } E\in\ET, \\
    & \int_T (\CURL\vv - \vlproj{k}{T}(\CURL\vv)) \cdot (\xxE \times \bvec{p}) = 0   \quad\forall \bvec{p} \in \vPoly{k-1 | k}(T) , \\
    & \int_{T} (\vv - \vlproj{k}{T} \vv )\cdot \xxE \,p\df=0 \quad\forall p\in \Poly{k-1 | k+1}(T)
    \Big\} .
\end{aligned}
$$
The DoFs are chosen as follows
\begin{align}\label{eq:DDR:dofs.Xek}
& \bullet\mbox{for each edge $E\in\ET$, $\int_E (\vv\cdot\tt_E) p\ds \quad \forall p \in \Poly{k}(E) ,$ } \\
& \bullet \mbox{if $k\ge 1$: \ for each face $F\in\FT$,  $\int_{F}({\vv_{{\rm t},F}}\cdot\xxf) \: p\df \quad \forall p \in\Poly{k-1}(F),$} \\
& \bullet \mbox{if $k \ge1$: \ for each face $F\in\FT$,  $\int_{F}{\vv_{{\rm t},F}\cdot\VROT_Fp} \df \quad \forall p \in \Poly[0]{k}(F)$ }, \\
&
 \bullet \mbox{if $k\ge 1$: \ $ \int_{T}
(\vv\cdot\xx_{T}) p\, \dPP \quad \forall p \in \Poly{\kdP}(T)$} , \\
&\label{eq:DDR:dofs.Xek:Gck}
 \bullet\mbox{if $k\ge 1$: \ $\int_{T} 
\vv\cdot\CURL (\bvec{x}_T\times\bvec{p}) \,\dPP \quad \forall \bvec{p} \in\vPoly{k-1}(T)$} .
\end{align}

\paragraph{Face space.} 
The face space on an element $T$ of the mesh is given by
\begin{multline*}
\bvec{X}^{\face}_{k}(T) \!\coloneq
\Big\{ \ww\in\bvec{L}^2(T)\st {\DIV\ww\in \Poly{k}}(T), \, \CURL\ww\in\vPoly{k}(T),\:   \ww_{|F}\cdot\nn_F\in\Poly{k}(F)\quad\forall F \in \FT , \\
\int_T (\ww - \vlproj{k}{T}\ww) \cdot (\xxE \times \bvec{p}) = 0   \quad\forall \bvec{p} \in \vPoly{k-1 | k}(F) ,
\Big\}.
\end{multline*}
with DoFs
\begin{align*}
& \bullet\mbox{for each face $F\in\FT$, $\int_F
(\ww\cdot\nn_F) p \df \quad \forall p \in \Poly{k}(F), $} \\
& \bullet \mbox{if $k\ge 1$: \ $\int_{T}
{\ww\cdot(\GRAD p) \dPP \quad \forall p \in \Poly[0]{k}(T) }$}, \\
& \bullet \mbox{if $k\ge 1$: \ $\int_{T} \ww\cdot (\xx_{T}\times \bvec{p}) \dPP \quad \forall \bvec{p}\in\vPoly{k-1}(T)$}.
\end{align*}

\subsubsection{$L^2$-orthogonal projectors}

\paragraph{Nodal space.}
By definition of the $L^2$-orthogonal projector and an integration by parts, 
for all $q_F \in {X}_{k+1}^{\node}(F)$ and all $\bvec{p} \in \vPoly{k}(F)$,
$$
\int_{F} \vlproj{k}{F} (\GRAD_F q_F) \cdot \bvec{p} = \int_{F} (\GRAD_F q_F) \cdot \bvec{p}
= - \int_F q_F \DIV_F \bvec{p} \, + \! \sum_{E\in\EF} \omega_{FE} \int_E q_F (\bvec{p}\cdot\normal_{FE}), 
$$
which is computable from the DoFs. 
This relation and the chosen DoFs show that $\vlproj{k}{F}\circ\GRAD_F$ is computable on $X_{k+1}^\node(F)$ from the DoFs and corresponds to the face gradient \eqref{eq:cGF}.
The scalar trace defined in \eqref{eq:trF} corresponds to the $L^2$-orthogonal projector from ${X}_{k+1}^{\node}(F)$ onto $\Poly{k+1}(F)$. This is seen first expressing any $r \in \Poly{k+1}(F)$ as $\DIV_F (\xx_F p)$ with 
$p\in\Poly{k+1}(F)$ (see, e.g., \cite[Remark 2]{Di-Pietro.Droniou:21*1}), and then integrating by parts to get, for all $q_F \in  {X}_{k+1}^{\node}(F)$,
$$
\int_F (\lproj{k+1}{F} q_F) \, r = \int_F q_F \, r = - \int_F (\GRAD_F q_F) \cdot (\xx_F p) + 
 \sum_{E\in\EF} \omega_{FE} \int_E q_F \, p (\xx_F\cdot\normal_{FE}).
$$
Thanks to the second line in definition \eqref{DDR2VEM1}, the term $\GRAD_F q_F$ in the right-hand side can be replaced with $\vlproj{k}{F} (\GRAD_F q_F)$. This shows that $\lproj{k+1}{F} q_F$ is computable from the DoFs and corresponds to the scalar trace $\trF \underline{q}_F$.
Analogously, the element gradient in \eqref{new:X1} corresponds to the $L^2$-projection of the gradient 
on $\vPoly{k}(T)$ (and this projection is therefore computable from the DoFs) since, for all $q_T \in {X}_{k+1}^{\node}(T)$ and all $\bvec{p} \in \vPoly{k}(T)$,
$$
\int_{F} \vlproj{k}{T} (\GRAD q_T) \cdot \bvec{p} = \int_{T} (\GRAD q_T) \cdot \bvec{p}
= - \int_T q_T \DIV \bvec{p}  +  \sum_{F\in\FT} \omega_{TF} \int_F (\lproj{k+1}{F} q_T) (\bvec{p}\cdot\normal_{F}).
$$
The discrete scalar potential \eqref{eq:PgradT} corresponds to the $L^2$-orthogonal projector from 
${X}_{k+1}^{\node}(T)$ onto $\Poly{k+1}(T)$. Indeed, writing any $r \in \Poly{k+1}(T)$ as $\DIV (\xx_T p)$ with 
$p\in\Poly{k+1}(T)$, for all $q_T\in X_{k+1}^\node(T)$,
$$
\begin{aligned}
\int_T (\lproj{k+1}{T} q_T) r = \int_T q_T \, r 
& = - \int_T (\GRAD q_T) \cdot (\xx_T p) + 
 \sum_{F\in\FT} \omega_{TF} \int_F  q_T \, p (\xx_T\cdot\normal_{F}) \\
& = - \int_T \vlproj{k}{T} (\GRAD q_T) \cdot (\xx_T p) + 
 \sum_{F\in\FT} \omega_{TF} \int_F (\lproj{k+1}{F} q_T) \, p (\xx_T\cdot\normal_{F}) \, ,
\end{aligned}
$$
where, in the last step, we also used the constraint appearing in the definition of ${X}_{k+1}^{\node}(T)$, together with the fact that $\xx_T\cdot\normal_F$ is constant over $F$.

\paragraph{Edge space.}  
For all $F\in\Fh$ and $\vv_F \in \bvec{X}^{\edge}_{k}(F)$, the face curl \eqref{eq:CF} corresponds to $\ROT_F\vv_F \in \Poly{k}(F)$, without the need for any projection. This can be checked by noticing that
\begin{equation*}
  \int_F (\ROT_F\vv_F) p
  = \int_F \vv_F \cdot \VROT_F p
  - \sum_{E\in\EF}\omega_{FE}\int_E (\vv_F\cdot\tangent_{E}) p \qquad
  \forall p \in\Poly{k}(F)
\end{equation*}
and by recalling the choice of DoFs in $\bvec{X}^\edge_k(F)$.
The tangential trace \eqref{eq:trFt} becomes the $L^2$-orthogonal projector $\vlproj{k}{F} : \bvec{X}^{\edge}_{k}(F)\to\vPoly{k}(F)$ (which is thus computable
from the DoFs). Indeed, for all $\vv_F \in \bvec{X}^{\edge}_{k}(F)$, writing any $\bvec{p} \in \vPoly{k}(F)$ as $\VROT_F p + \xxf q$ with $p \in \Poly{k+1}(F)$
and $q \in \Poly{k-1}(F)$, we have
\begin{align*}
  \int_F \vlproj{k}{F}\vv_F \cdot(\VROT_F p + \xxf q) 
  ={}& \int_F \vv_F \cdot(\VROT_F p + \xxf q)\\ 
  ={}& \int_F p (\ROT_F \vv_F) 
  + \sum_{E\in\EF}\omega_{FE} \int_E (\vv_F\cdot\tangent_E) p
  + \int_F (\vv_F\cdot\xxf) \, q,
\end{align*}
with, as noticed above, $\ROT_F \vv_F$ corresponding to the discrete face curl appearing in \eqref{eq:trFt}.

For all $T\in\Th$ and $\vv_T\in \bvec{X}^{\edge}_{k}(T)$, the element curl \eqref{eq:cCT} is the $L^2$-orthogonal projector of $\CURL \vv$ onto $\vPoly{k}(T)$. To see this, we simply write
\begin{equation*}
\begin{aligned}
  \int_T  (\vlproj{k}{T} \CURL \vv_T) \cdot \bvec{p}
  & = \int_T \vv_T \cdot \CURL\bvec{p}
  + \sum_{F\in\FT} \omega_{TF} \int_F  \vv_T \cdot(\bvec{p}\times\normal_F) \\
  & = \int_T \Rproj{k-1}{T} \vv_T \cdot \CURL\bvec{p}
  + \sum_{F\in\FT} \omega_{TF} \int_F  \vlproj{k}{F}(\vv_T)_{{\rm t},F} \cdot(\bvec{p}\times\normal_F) 
  \quad\forall \bvec{p} \in\vPoly{k}(T),
  \end{aligned}
\end{equation*}
the introduction of the tangential component denoted by ${\rm t},F$ and of the projector $\vlproj{k}{F}$ being justified by the fact that $\bvec{p}\times\normal_F$ is tangential to $F$ and of degree $\le k$.
We conclude the equivalence with \eqref{eq:cCT} by recalling that the DoFs \eqref{eq:DDR:dofs.Xek:Gck} provide $\Rproj{k-1}{T} \vv_T$ (since $\Roly{k-1}(T)=\CURL\vPoly{k}(T) = \CURL\cGoly{k}(T)$) and that $\vlproj{k}{F}(\vv_T)_{{\rm t},F}$ corresponds through the DoFs to the tangential trace $\trFt\uvec{v}_F$.
Finally, the discrete vector potential \eqref{eq:PcurlT}  corresponds to the $L^2$-orthogonal projector $\vlproj{k}{T} : \bvec{X}^{\edge}_{k}(T) \to\vPoly{k}(T)$ since a generic test function $\bvec{p} \in \vPoly{k}(T)$ can be decomposed as $\bvec{p} = \CURL \bvec{q} + \xxE p$ with $\bvec{q} \in \cGoly{k+1}(T)$ and $p \in \Poly{k-1}(T)$ so that, for all $\vv_T \in \bvec{X}^{\edge}_{k}(T)$,
$$
\int_T \vlproj{k}{T}\vv_T \cdot \bvec{p}
= \int_T (\vlproj{k}{T}\CURL\vv_T) \cdot \bvec{q}
- \sum_{F\in\FT}\omega_{TF} \int_F  \vlproj{k}{F}(\vv_T)_{{\rm t},F} \cdot(\bvec{q}\times\normal_F)
+ \int_T (\vv_T\cdot\xxE)\, p,
$$
where we used again the constraints appearing in the definition of $\bvec{X}^{\edge}_{k}(T)$ and $\bvec{X}^{\edge}_{k}(F)$,
the latter combined with the fact that 
$\bvec{q}_{|F}\times \normal_F \in \vPoly{k}(F)+\xxf\Poly{k}(F)$, see for instance \cite[Eq.~(A.5)]{Di-Pietro.Droniou:21*1}.

\paragraph{Face space.}  
For all $\ww_T \in \bvec{X}^{\face}_{k}(T)$, the discrete divergence \eqref{eq:DT} corresponds to $\DIV\ww_T \in \Poly{k}(T)$ (without any projection) since
$$
\int_T (\DIV\bvec{w}_T) q = - \int_T \bvec{w}_T \cdot\GRAD q
+ \sum_{F\in\FT}\omega_{TF} \int_F (\ww_T\cdot\normal_F) q \qquad\forall q\in\Poly{k}(T),
$$
each of the terms involving $\bvec{w}_T$ in the right-hand side corresponding to some DoFs on $\bvec{X}^\face_k(T)$.
The discrete vector potential \eqref{eq:PdivT} corresponds to the $L^2$-orthogonal projector 
$\vlproj{k}{T}: \bvec{X}^{\face}_{k}:\to\vPoly{k}(T)$: For all $\ww\in\bvec{X}^{\face}_{k}(T)$, writing a generic $\bvec{p}\in \vPoly{k}(T)$ as 
$\bvec{p} = \GRAD r + \xxE \times \bvec{q}$ with $r \in \Poly{k+1}(T)$ and $\bvec{q} \in \Poly{k-1}(T)$, we have
$$
\int_T (\vlproj{k}{T}\ww) \cdot(\GRAD r + \xxE\times\bvec{q})
= -\int_T (\DIV\ww)  r + \sum_{F\in\FT}\omega_{TF} \int_F (\ww\cdot\normal_F) r
+ \int_T \ww \cdot (\xxE\times\bvec{q}).
$$

\subsubsection{Discrete scalar products}

Having established the relation, through the DoFs, between discrete potentials and projections, the translation of the DDR scalar products of Section \ref{sec:DDRpots} into the VEM setting is simply a substitution of symbols. For example, in the nodal case, the scalar product is given by: For all $q,w \in {X}_{k+1}^{\node}(T)$,
$$
\begin{aligned}{}
[q,w]_{{X}_{k+1}^{\node}(T)} ={}& \int_T (\lproj{k+1}{T} q) \, (\lproj{k+1}{T} w) + S_{\GRAD,T}(q,w) \, ,
\\
\mbox{ where }S_{\GRAD,T}(q,w)
={}&
\sum_{F\in\FT} h_T\int_F\big( \lproj{k+1}{T} q - \lproj{k+1}{F} q \big) \big( \lproj{k+1}{T} w - \lproj{k+1}{F} w \big)
\\
& + \sum_{E\in\ET} h_T^2 \int_E \big( \lproj{k+1}{T} q - q) \big(\lproj{k+1}{T} w - w \big) \, .
\end{aligned}
$$

\subsection{Comparison and further developments}\label{sec:bridge:comparison}

Due to the use of the serendipity approach on element faces, the Virtual Element Method has fewer DoFs for certain DoF sets. On the other hand, thanks to the systematic adoption of what, in the VEM language, would be called an enhancement approach, the DDR has fewer DoFs in other instances, such as element volumes.
Enhancement and serendipity have a similar goal, which is to use some DoFs that allow the computation of a suitable projection operator in order to ``enslave'' some other DoFs and thus reduce the local space dimension. They however have some important technical differences:
\begin{compactitem}
\item enhancement incorporates information made available by differential operators reconstructions into a potential of higher order than the projection directly computable from the DoFs;
\item serendipity hinges on the fact that fixing the value of a polynomial on the boundary of an element fixes also its value inside the element, at least for elements with a sufficiently large number of edges/faces (or otherwise with the help of some internal moments).
\end{compactitem}
This difference is reflected in the fact that serendipity is a geometry-dependent approach (as can be noticed recalling the definition of the parameter $\eta_F$ in Section \ref{sec:local}), while higher-order reconstructions/enhancement are independent of the element shape (but potentially less efficient as they ``enslave'' a smaller number of DoFs).
  
Thanks to the bridges developed above, such DoF-reduction strategies can be easily combined
resulting in a more efficient and highly competitive construction that can be interpreted both as DDR and VEM, as detailed in the recent work \cite{Di-Pietro.Droniou:22}. 
Drawing the details of such construction is beyond the scope of the present article; nevertheless, to shed some more light onto the idea, we show very briefly the following short example of cross-fertilization. 
Consider the Virtual Element space $V^{\face}_{k}(T)$. It can be checked that lowering by one the degree $k$ appearing in the DoF set \eqref{dof-3dfk-3} (that is, testing only on $\bvec{p}\in\vPoly{k-1}(T)$) would still allow to compute the $L^2$-projection 
${\vlproj{k}{T}}: V^{\face}_{k}(T) \rightarrow \vPoly{k}(T)$. Therefore one could use such projection to introduce an ``enhancement'' following the spirit in \cite{Projectors} and mimicking the analogous construction in the DDR approach:
decrease by one the degree $k$ in \eqref{dof-3dfk-3} and introduce the corresponding constraint in the definition of the space $V^{\face}_{k}(T)$. 
Analogously, the serendipity idea from VEM could be injected in the DDR setting, allowing to lower the DoF count on faces.
For instance, one could introduce in the spaces $\Xgrad{T}$ and $\Xcurl{T}$ a construction leveraging serendipity operators in order to fix the values of certain polynomial components on faces.

\section{Theoretical analysis of the methods}\label{sec:theoretical}

In this section we prove the error estimates for the proposed schemes stated in Theorems \ref{thm:ddr:convergence} (DDR) and \ref{thm:vem:convergence} (VEM).

\subsection{Analysis of the DDR scheme}\label{sec:theo-DDR}

After recasting problem \eqref{eq:discrete} in variational form and proving stability for the bilinear form in the left-hand side, we give a proof of Theorem \ref{thm:ddr:convergence}.

\subsubsection{Variational formulation and stability}\label{sec:theo-DDR:stability}

The variational formulation of problem \eqref{eq:discrete} reads
\begin{equation}\label{eq:discrete:variational}
  \left\{~
  \begin{aligned}
    &\text{Find $(\uvec{u}_h,\underline{p}_h)\in\Xcurl{h}\times\XgradO{h}$ such that}
    \\
    &\mathrm{A}_h((\uvec{u}_h,\underline{p}_h),(\uvec{v}_h,\underline{q}_h))
    = \ell_h(\bvec{f},\uvec{v}_h)\quad
    \forall(\uvec{v}_h,\underline{q}_h)\in\Xcurl{h}\times\XgradO{h},
  \end{aligned}
  \right.
\end{equation}
where the bilinear form $\mathrm{A}_h:\big(\Xcurl{h}\times\Xgrad{h}\big)^2\to\Real$ is such that, for all $(\uvec{w}_h,\underline{r}_h),(\uvec{v}_h,\underline{q}_h)\in\Xcurl{h}\times\Xgrad{h}$,
\begin{equation}\label{eq:Ah}
  \mathrm{A}_h((\uvec{w}_h,\underline{r}_h),(\uvec{v}_h,\underline{q}_h))
  \coloneq
  \mathrm{a}_h(\uvec{w}_h,\uvec{v}_h) + \mathrm{b}_h(\underline{r}_h,\uvec{v}_h) -\mathrm{b}_h(\underline{q}_h,\uvec{w}_h).
\end{equation}
Well-posedness is then a classical consequence of the following inf-sup condition on $\mathrm{A}_h$.
\begin{lemma}[Inf-sup condition for $\mathrm{A}_h$]\label{lem:inf-sup}
  For all $(\uvec{w}_h,\underline{r}_h)\in\Xcurl{h}\times\XgradO{h}$, it holds
  \begin{equation}\label{eq:inf-sup}
    \tnorm[h]{(\uvec{w}_h,\underline{r}_h)}
    \lesssim\sup_{(\uvec{v}_h,\underline{q}_h)\in\Xcurl{h}\times\XgradO{h}\setminus\{(\uvec{0},\underline{0})\}}
    \frac{\mathrm{A}_h((\uvec{w}_h,\underline{r}_h),(\uvec{v}_h,\underline{q}_h))}{\tnorm[h]{(\uvec{v}_h,\underline{q}_h)}}.
  \end{equation}
\end{lemma}

\begin{proof}
  Denote by $\$$ the supremum in the right-hand side of \eqref{eq:inf-sup}.
  Taking $(\uvec{v}_h,\underline{q}_h) = (\uvec{w}_h + \uGh\underline{r}_h,\underline{r}_h)$ in \eqref{eq:Ah},
  recalling the definitions \eqref{eq:ah.bh.lh} of $\mathrm{a}_h$ and $\mathrm{b}_h$,
  and using the relation \eqref{eq:Im.uGh.subset.Ker.uCh}, we have
  \begin{equation}\label{eq:inf-sup:1}
    \begin{aligned}
    \norm[\DIV,h]{\uCh\uvec{w}_h}^2
    + \norm[\CURL,h]{\uGh\underline{r}_h}^2
    &= \mathrm{A}_h((\uvec{w}_h,\underline{r}_h),(\uvec{w}_h + \uGh\underline{r}_h,\underline{r}_h))
    \\
    &\le\$\tnorm[h]{(\uvec{w}_h + \uGh\underline{r}_h,\underline{r}_h)}
    \lesssim
    \$\tnorm[h]{(\uvec{w}_h,\underline{r}_h)},
  \end{aligned}
  \end{equation}
  where we have used a triangle inequality along with the definitions \eqref{eq:tnorm.h} of $\tnorm[h]{{\cdot}}$ and \eqref{eq:tnorm.CURL.GRAD.h} of $\tnorm[\GRAD,h]{{\cdot}}$ for the pair $(\uGh\underline{r}_h,\underline{0})$ together with \eqref{eq:Im.uGh.subset.Ker.uCh} to infer $\tnorm[h]{(\uvec{w}_h + \uGh\underline{r}_h,\underline{r}_h)}\lesssim\tnorm[h]{(\uvec{w}_h,\underline{r}_h)}$.
  The Poincar\'e--Wirtinger inequality of \cite[Theorem 3]{Di-Pietro.Droniou:21*1} combined with \eqref{eq:inf-sup:1} yields
  \begin{equation}\label{eq:inf-sup:2}
    \norm[\GRAD,h]{\underline{r}_h}^2\lesssim\$\tnorm[h]{(\uvec{w}_h,\underline{r}_h)}.
  \end{equation}
  To estimate $\norm[\CURL,h]{\uvec{w}_h}$, we infer from the exactness property of \cite[Theorem 2]{Di-Pietro.Droniou:21*1} the $(\cdot,\cdot)_{\CURL,h}$-orthogonal decomposition $\uvec{w}_h = \uvec{\psi}_h + \uGh\underline{\phi}_h$, where $\uvec{\psi}_h\in(\Ker\uCh)^\perp$ and $\underline{\phi}_h\in\XgradO{h}$.
  By the Poincar\'e inequality of \cite[Theorem 4]{Di-Pietro.Droniou:21*1}, it holds
  \begin{equation}\label{eq:inf-sup:3}
    \norm[\CURL,h]{\uvec{\psi}_h}^2
    \lesssim\norm[\DIV,h]{\uCh\uvec{\psi}_h}^2
    =\norm[\DIV,h]{\uCh\uvec{w}_h}^2
    \lesssim\$\tnorm[h]{(\uvec{w}_h,\underline{r}_h)},
  \end{equation}
  where the equality follows from \eqref{eq:Im.uGh.subset.Ker.uCh} and we have used \eqref{eq:inf-sup:1} to conclude.
  On the other hand, taking $(\uvec{v}_h,\underline{q}_h) = (\uvec{0},-\underline{\phi}_h)$ in \eqref{eq:Ah} and using the definition \eqref{eq:ah.bh.lh} of $\mathrm{b}_h$ along with the $(\cdot,\cdot)_{\CURL,h}$-orthogonality of the decomposition $\uvec{w}_h = \uvec{\psi}_h + \uGh\underline{\phi}_h$ and $\uGh\underline{\phi}_h\in\Ker\uCh$ (see \eqref{eq:Im.uGh.subset.Ker.uCh}) to infer $\mathrm{b}_h(\underline{\phi}_h,\uvec{w}_h) = \cancel{(\uGh\underline{\phi}_h,\uvec{\psi}_h)_{\CURL,h}} + \norm[\CURL,h]{\uGh\underline{\phi}_h}^2$, we get
  \begin{equation}\label{eq:inf-sup:4}
    \norm[\CURL,h]{\uGh\underline{\phi}_h}^2
    = \mathrm{A}_h((\uvec{w}_h,\underline{r}_h),(\uvec{0},-\underline{\phi}_h))
    \le\$\tnorm[h]{(\uvec{0},\underline{\phi}_h)}
    \le\$\norm[\CURL,h]{\uvec{w}_h}
    \le\$\tnorm[h]{(\uvec{w}_h,\underline{r}_h)}.
  \end{equation}
  The second inequality in \eqref{eq:inf-sup:4} is obtained writing
  $$
  \tnorm[h]{(\uvec{0},\underline{\phi}_h)}^2
  = \norm[\GRAD,h]{\underline{\phi}_h}^2 + \norm[\CURL,h]{\uGh\underline{\phi}_h}^2
  \lesssim \norm[\CURL,h]{\uGh\underline{\phi}_h}^2
  \le \norm[\CURL,h]{\uvec{w}_h}^2,
  $$
  where we have used, in this order, the definitions \eqref{eq:tnorm.h} of $\tnorm[h]{{\cdot}}$ and \eqref{eq:tnorm.CURL.GRAD.h} of $\tnorm[\GRAD,h]{{\cdot}}$,
  the discrete Poincar\'e inequality of \cite[Theorem 3]{Di-Pietro.Droniou:21*1},
  and the $(\cdot,\cdot)_{\CURL,h}$-orthogonality of the decomposition $\uvec{w}_h = \uvec{\psi}_h + \uGh\underline{\phi}_h$.
  Summing \eqref{eq:inf-sup:3} and \eqref{eq:inf-sup:4}, it is inferred that $\norm[\CURL,h]{\uvec{w}_h}^2
    \lesssim\$\tnorm[h]{(\uvec{w}_h,\underline{r}_h)}$ which, together with \eqref{eq:inf-sup:1} and \eqref{eq:inf-sup:2}, gives \eqref{eq:inf-sup} after simplification.
\end{proof}

\subsubsection{Convergence}\label{sec:theo-DDR:convergence}

\begin{proof}[Proof of Theorem \ref{thm:ddr:convergence}]
  By the Third Strang lemma \cite[Theorem 10]{Di-Pietro.Droniou:18} and Lemma \ref{lem:inf-sup} (with a slight modification for $k=0$, not detailed here, to make sure that $\underline{p}_h - \Igrad{h}p$ has zero average), it holds
  \begin{equation}\label{eq:err.est:basic}
    \tnorm[h]{(\uvec{u}_h - \Icurl{h}\bvec{u}, \underline{p}_h - \Igrad{h} p)}
    \lesssim\sup_{(\uvec{v}_h,\underline{q}_h)\in\Xcurl{h}\times\XgradO{h}\setminus\{(\uvec{0},\underline{0})\}}
    \frac{\mathcal{E}_h(\uvec{v}_h,\underline{q}_h)}{\tnorm[h]{(\uvec{v}_h,\underline{q}_h)}},
  \end{equation}
  where $\mathcal{E}_h:\Xcurl{h}\times\XgradO{h}\to\Real$ is the consistency error linear form such that
  \[
  \mathcal{E}_h(\uvec{v}_h,\underline{q}_h)
  \coloneq\ell_h(\bvec{f},\uvec{v}_h) - \mathrm{A}_h((\Icurl{h}\bvec{u},\Igrad{h}p),(\uvec{v}_h,\underline{q}_h)).
  \]
  To prove \eqref{eq:err.est}, it suffices to bound the right-hand side of \eqref{eq:err.est:basic}.
  The rest of the proof relies on consistency results established in \cite{Di-Pietro.Droniou:21*1} in the case of maximal regularity $s=k+1$; their adaptation to the generic case $1\le s\le k+1$, used here, is straightforward.
  Recalling the definitions \eqref{eq:ah.bh.lh} and \eqref{eq:Ah} of the discrete bilinear forms, and since $\bvec{f} = \CURL(\CURL\bvec{u}) + \GRAD p$ almost everywhere in $\Omega$, we have
  \begin{equation}\label{eq:err.est:decomposition}
  \begin{aligned}
    \mathcal{E}_h(\uvec{v}_h,\underline{q}_h)
    &=
    \underbrace{%
      \ell_h(\CURL(\CURL\bvec{u}),\uvec{v}_h)
      - (\uCh(\Icurl{h}\bvec{u}),\uCh\uvec{v}_h)_{\DIV,h}
    }_{\term_1}
    \\
    &\quad
    + \underbrace{%
      \ell_h(\GRAD p,\uvec{v}_h)
      - (\uGh(\Igrad{h}p),\uvec{v}_h)_{\CURL,h}
    }_{\term_2}
    + \underbrace{%
      (\uGh\underline{q}_h,\Icurl{h}\bvec{u})_{\CURL,h}.
    }_{\term_3}
  \end{aligned}
  \end{equation}
  Recalling the commutation property $\uCh(\Icurl{h}\bvec{w}) = \Idiv{h}(\CURL\bvec{w})$ valid for all $\bvec{w}\in\bvec{H}^2(\Omega)$,
  which is an easy consequence of the corresponding local relation proved in \cite[Lemma 4]{Di-Pietro.Droniou:21*1}, and replacing $\ell_h$ with its definition \eqref{eq:ah.bh.lh}, we have
  \[
  \term_1 = (\Icurl{h}(\CURL(\CURL\bvec{u})),\uvec{v}_h)_{\CURL,h}
  - (\Idiv{h}(\CURL\bvec{u}),\uCh\uvec{v}_h)_{\DIV,h}.
  \]
  Then, leveraging \cite[Theorem 10]{Di-Pietro.Droniou:21*1} as in the bound of the component $\mathcal{E}_{h,3}$ in the proof of \cite[Theorem 12]{Di-Pietro.Droniou:21*1}, we obtain
  \begin{equation}\label{eq:err.est:T1}
    |\term_1|\lesssim h^s\left(
    \seminorm[\bvec{H}^s(\Th)]{\CURL\bvec{u}}
    + \seminorm[\bvec{H}^{s+1}(\Th)]{\CURL\bvec{u}}
    + \seminorm[\bvec{H}^{(s,2)}(\Th)]{\CURL\CURL\bvec{u}}
    \right)\tnorm[\CURL,h]{\uvec{v}_h}.
  \end{equation}
  By \eqref{eq:rhs:irrotational.source} with $\psi=p$, we immediately have for the second term
  \begin{equation}\label{eq:err.est:T2}
    \term_2 = 0.
  \end{equation}
  Finally, for the third term, the adjoint consistency result for the gradient proved in \cite[Theorem 9]{Di-Pietro.Droniou:21*1} along with the fact that $\DIV\bvec{u}=0$ in $\Omega$ and $\bvec{u}\cdot\bvec{n}=0$ on $\partial\Omega$ readily yields
  \begin{equation}\label{eq:err.est:T3}
    |\term_3|\lesssim h^s\seminorm[\bvec{H}^{(s,2)}(\Th)]{\bvec{u}}\norm[\CURL,h]{\uGh\underline{q}_h}.
  \end{equation}
  Using \eqref{eq:err.est:T1}, \eqref{eq:err.est:T2}, and \eqref{eq:err.est:T3} to estimate the right-hand side of \eqref{eq:err.est:decomposition} and plugging the resulting bound into \eqref{eq:err.est:basic} proves \eqref{eq:err.est}.
\end{proof}

\subsection{Analysis of the VEM scheme}\label{sec:theo:VEM}

In this section we tackle the convergence analysis of the Virtual Element approach presented in Section \ref{sec:discre1}.

\subsubsection{Preliminary results}\label{sec:prel}

Let us start by reviewing some results that will be useful in the following. The proof of Lemma \ref{prop:int-nodal} can be found in \cite{Brenner-Sung}, see also 
\cite{Beirao-da-Veiga.Lovadina.ea:17*1,CangianiAPOS}, for any order $k$. The proofs of the other interpolation and scalar product stability results below can be found in \cite{edgeface,edgefacegeneral}, while the proof of Lemma \ref{prop:coerc} is provided in \cite{maxwellVEM}. 
We classically assume, in what follows, the following mesh property.

\begin{assumption}[Star-shaped property]\label{ass:star-shaped}
For all meshsize $h$, all the elements $T\in\Th$ and all the faces $F\in\Fh$ are uniformly star shaped with respect to a ball, and there exists a uniform positive constant $\gamma$ such that $h_E \ge \gamma h_T$ for all $E \in \ET$.
\end{assumption}

\smallskip\noindent
The discrete spaces proposed in Section \ref{sec:discre1} have optimal approximation properties, in terms of the associated polynomial degree, as shown below. We state the next three results only at the local level, the global counterpart following trivially by summing on all elements.
\begin{lemma}\label{prop:int-nodal}
Let $v \in H^s(\Omega)$, $\frac32 < s \le k+2$. 
Then, the nodal interpolant $v_I = {\cal I}^{\node}_{k+1} (v) \in V^{\node}_{k+1}$ (which, we recall, is the space $V^\node_{k+1,0}$ without the zero average condition) satisfies
$$
\norm[L^2(T)]{v - v_I} + h_T \seminorm[H^1(T)]{v - v_I} \lesssim h_T^{s} \seminorm[H^s(T)]{v} \quad \forall T \in \Th.
$$
\end{lemma}

\begin{lemma}\label{prop:int-edge}
  Let $\vv \in\Hscurl{s}{\Omega}$, $\frac12 < s \le k+1$, with tangential component $\vv_{|E}\cdot\tt_E$ integrable on each edge $E\in\Eh$.
  Then, the edge interpolant $\vv_I = {\cal I}^{\edge}_k (\vv)\in V^{\edge}_{k}$ satisfies
$$
\norm[\Hcurl{T}]{\vv - \vv_I} \lesssim C h_T^{s} \, \Big(
\norm[\bvec{H}^s(T)]{\vv} + \seminorm[\bvec{H}^s(T)]{\CURL\vv}
\Big)
\quad \forall T \in \Th.
$$
If $s > \frac32$, then the right-hand side can be substituted with $C h_T^{s} \, \seminorm[\bvec{H}^s(T)]{\vv}$.
\end{lemma}

We furthermore have the following stability result for the discrete scalar products.
\begin{lemma}\label{prop:stab}
It holds
$$
\norm[\bvec{L}^2(T)]{\vv}^2 \,\lesssim\, [\vv , \vv]_{V^\bullet_k(T)} \,\lesssim \,  \norm[\bvec{L}^2(T)]{\vv}^2 \quad \forall T \in \Th, \ \forall \vv \in V^\bullet_k(T),
$$
where, as usual, the symbol $V^\bullet_k(T)$ denotes either $V^{\edge}_{k}(T)$ or $V^{\face}_{k}(T)$.
\end{lemma}

Finally, we state a Poincar\'e inequality for the curl.
\begin{lemma}\label{prop:coerc}
Let $Z_h$ be the orthogonal in $V^\edge_k$ of the image of the gradient, that is
\[
Z_h = \left\{ \vv \in V^{\edge}_{k}\st [\vv,\GRAD w]_{V^{\edge}_{k}} = 0 \quad \forall w \in V^{\node}_{k+1,0} \right\} .
\]
Then, it holds
$$
\norm[\bvec{L}^2(\Omega)]{\vv} \lesssim \norm[\bvec{L}^2(\Omega)]{\CURL \vv} \qquad \forall \vv \in Z_h .
$$
\end{lemma}

\subsubsection{Stability and convergence}\label{sec:conv}

We start by a simple lemma stating the consistency of the scalar product on the virtual spaces. 

\begin{lemma}[Consistency of the scalar products]\label{lem:scal} 
  Let $T\in\Th$ and the symbol $V^\bullet_k(T)$ represent either $V^{\edge}_{k}(T)$ or $V^{\face}_{k}(T)$.
  Let $\var \in\bvec{H}^s(T)$, $0 \le s \le k+1$, and $\var_h \in V^\bullet_k(T)$.
  Then, for all $\ww_h \in V^\bullet_k(T)$, it holds
  $$
  \int_T \var\cdot\ww_h - [\var_h,\ww_h]_{V^\bullet_k(T)} \lesssim \Big(
  \norm[\bvec{L}^2(T)]{\var - \var_h}
  + h_T^s \seminorm[\bvec{H}^s(T)]{\var}
  \Big) \norm[\bvec{L}^2(T)]{\ww_h}.
  $$
\end{lemma}
\begin{proof} 
  Let $\var_\pi$ be the best $L^2$ approximation in $\vPoly{k}(T)$ of $\var$.
  Then, using first property \eqref{consiE3k} and then Lemma \ref{prop:stab}, we find 
$$
\begin{aligned}
  \int_T \var\cdot\ww_h - [\var_h,\ww_h]_{V^\bullet_k(T)} &= 
  \int_T (\var-\var_\pi)\cdot\ww_h + [\var_\pi - \var_h,\ww_h]_{V^\bullet_k(T)} \\
  & \lesssim \Big( \norm[\bvec{L}^2(T)]{\var-\var_\pi} + \norm[\bvec{L}^2(T)]{\var_\pi - \var_h} \Big) 
  \norm[\bvec{L}^2(T)]{\ww_h} \\
  & \lesssim \Big( \norm[\bvec{L}^2(T)]{\var-\var_\pi} + \norm[\bvec{L}^2(T)]{\var - \var_h} \Big) \norm[\bvec{L}^2(T)]{\ww_h},
\end{aligned}
$$
where the last line is obtained introducing $\var$ in $\norm[\bvec{L}^2(T)]{\var_\pi - \var_h}$ and using a triangle inequality.
The lemma is concluded by standard polynomial approximation results on star shaped polytopes, see for instance \cite[Theorem 1.45]{Di-Pietro.Droniou:20}.
\end{proof}

We are now ready to prove the convergence result stated in Theorem \ref{thm:vem:convergence}.
\begin{proof}[Proof of Theorem \ref{thm:vem:convergence}]
Introducing the linear form ${\cal A}_h$ 
\begin{multline*}
  {\cal A}_h(\ww_h,r_h;\vv_h,q_h)
  \coloneq [\CURL \ww_h , \CURL \vv_h]_{V^{\face}_{k}} + [\GRAD r_h,\vv_h]_{V^{\edge}_{k}} - [\GRAD q_h,\ww_h]_{V^{\edge}_{k}} \\
  \forall(\ww_h,r_h), (\vv_h,q_h) \in \left(V^{\edge}_{k}\times V^{\node}_{k+1,0}\right)^2,
\end{multline*}
problem \eqref{eq:discr-pbl} can be equivalently written as: 
Find $(\uu_h, p_h) \in V^{\edge}_{k}\times V^{\node}_{k+1,0}$ such that
\begin{equation}\label{eq:P2}
{\cal A}_h(\uu_h,p_h;\vv_h,q_h) = [\ff_I,\vv_h]_{V^{\edge}_{k}}
\qquad \forall (\vv_h,q_h) \in V^{\edge}_{k}\times V^{\node}_{k+1,0}.
\end{equation}
The stability of problem \eqref{eq:P2} in the natural norms associated to this formulation ($\Hcurl{\Omega}$ norm for $V^{\edge}_{k}$ and $H^1(\Omega)$ norm for $V^{\node}_{k+1,0}$) follows from the standard theory of mixed methods \cite{Boffi.Brezzi.ea:13}.
The coercivity on the discrete kernel follows immediately from Lemma \ref{prop:coerc} and the stability of the scalar products (i.e., Lemma \ref{prop:stab}).
The inf-sup condition (in the natural norms of the problem) is a simple consequence of the exact complex property and Lemma \ref{prop:stab}:
For each $q_h \in V^{\node}_{k+1,0}$,
$$
\sup_{\vv_h \in V^{\edge}_{k}} \frac{[\GRAD q_h,\vv_h]_{V^{\edge}_{k}}}{\| \vv_h \|_{\Hcurl{\Omega}}}
\ge \frac{[\GRAD q_h,\GRAD q_h]_{V^{\edge}_{k}}}{\| \GRAD q_h \|_{L^2(\Omega)}}
\gtrsim \| \GRAD q_h \|_{L^2(\Omega)}.
$$
Therefore, given $\uu_I \in V^{\edge}_{k}$, $p_I \in V^{\node}_{k+1,0}$ the interpolants of $\uu$ and $p$ 
(with a slight modification for $\beta_F<0$ not detailed here to make sure that $p_I$ has zero average), respectively, we have the existence of $\vv_h \in V^{\edge}_{k}$, $q_h \in V^{\node}_{k+1,0}$ such that
\begin{equation}\label{eq:init:stab}
\left\{
\begin{aligned}
  & \norm[\Hcurl{\Omega}]{\uu_h - \uu_I}
  + \norm[H^1(\Omega)]{p_h - p_I}
  \le {\cal A}_h(\uu_h-\uu_I,p_h-p_I;\vv_h,q_h), \\
  &  \norm[\Hcurl{\Omega}]{\vv_h}
  + \norm[H^1(\Omega)]{q_h} \lesssim 1.
\end{aligned}
\right.
\end{equation}
We start from \eqref{eq:init:stab} and apply the discrete equation \eqref{eq:P2}. Afterwards, we recall the continuous equation \eqref{eq:strong} and substitute $\ff$ in terms of $\uu$ and $p$. We obtain
\begin{equation}\label{eq:main}
  \norm[\Hcurl{\Omega}]{\uu_h - \uu_I} + \norm[H^1(\Omega)]{p_h - p_I} \le 
       [\ff_I,\vv_h]_{V^{\edge}_{k}} - {\cal A}_h(\uu_I,p_I;\vv_h,q_h)
       = \term_1 + \term_2 + \term_3,
\end{equation}
where
$$
\begin{aligned}
& \term_1 \coloneq [(\CURL\CURL\uu)_I,\vv_h]_{V^{\edge}_{k}} - [\CURL \uu_I , \CURL \vv_h]_{V^{\face}_{k}} \, , \\
& \term_2 \coloneq [(\GRAD p)_I,\vv_h]_{V^{\edge}_{k}} - [\GRAD p_I,\vv_h]_{V^{\edge}_{k}} \, , \\
& \term_3 \coloneq - [\GRAD q_h,\uu_I]_{V^{\edge}_{k}} \, .
\end{aligned}
$$
We deal with the three terms separately. Introducing $\int_\Omega (\CURL \uu)\cdot(\CURL \vv_h)$ and integrating by parts, recalling that $\CURL\uu \times \nn$ vanishes on the boundary, gives $\term_1 = \term_{1,1} + \term_{1,2}$ where
$$
\begin{aligned}
\term_{1,1} &\coloneq [(\CURL\CURL\uu)_I,\vv_h]_{V^{\edge}_{k}} - \int_\Omega (\CURL\CURL\uu)\cdot\vv_h \, , \\
\term_{1,2} &\coloneq \int_\Omega (\CURL \uu)\cdot(\CURL \vv_h) - [\CURL \uu_I , \CURL \vv_h]_{V^{\face}_{k}} \, .
\end{aligned}
$$
We bound the square of $\term_{1,1}$ first by applying Lemma \ref{lem:scal} and recalling \eqref{eq:init:stab}, 
then using Lemma \ref{prop:int-edge}:
$$
\begin{aligned}
  \term_{1,1}^2 & \lesssim
  \sum_{T\in\Th} \Big(
  \norm[\bvec{L}^2(T)]{\CURL\CURL\uu - (\CURL\CURL\uu)_I}^2
  + h_T^{2s} \seminorm[\bvec{H}^s(T)]{\CURL\CURL\uu}^2
  \Big) \\
  & \lesssim \sum_{T\in\Th}  h_T^{2s} \big(
    \norm[\bvec{H}^s(T)]{\CURL\CURL\uu}^2
  + \seminorm[\bvec{H}^s(T)]{\CURL\ff}^2
  \big) \, , 
\end{aligned}
$$
where we also used $\CURL \CURL\CURL \bvec{u}=\CURL \ff$.
Term $\term_{1,2}$ is bounded following the same steps, leading to
$$
\term_{1,2}^2 \lesssim \sum_{T\in\Th}  h_T^{2s} 
\big( \norm[\bvec{H}^s(T)]{\CURL\uu}^2 + \seminorm[\bvec{H}^s(T)]{\CURL\CURL\uu}^2
\big) \, .
$$
By the commuting diagram property (see Remark \ref{rem:PBgrad}), it is trivial to check that $\term_2=0$, which is a key point in the present analysis and essential to the pressure-robustness of the scheme. Finally, the term $\term_3$ is bounded by using the continuous equation
$$
\term_3 = 
[\GRAD q_h,\uu_I]_{V^{\edge}_{k}} - \int_\Omega (\GRAD q_h)\cdot\uu
$$
and then again by Lemmas \ref{lem:scal} and \ref{prop:int-edge}, similarly to the previous terms:
$$
\term_3 \lesssim \sum_{T\in\Th} \Big(
\norm[\bvec{L}^2(T)]{\uu - \uu_I}^2
+ h_T^{2s}\seminorm[\bvec{H}^s(T)]{\uu}^2
\Big)
\lesssim \sum_{T\in\Th}  h_T^{2s} 
\left(\norm[\bvec{H}^s(T)]{\uu}^2 + \seminorm[\bvec{H}^s(T)]{\CURL\uu}^2\right) \, .
$$
The final result follows by combining all the bounds into \eqref{eq:main} and using $h_T \le h$ for all $T\in\Th$.
\end{proof}

By combining the above proposition with the interpolation estimates in Lemmas \ref{prop:int-edge} and \ref{prop:int-nodal}, we obtain the following result, which is more akin to the error estimates in the VEM literature.

\begin{corollary}[Error with respect to the continuous solution]\label{prop:conv2}
Let the solution $(\uu,p)$ of the continuous problem and the datum $\ff$ satisfy the following.
The functions $\uu$, $\CURL\uu$, $\CURL\CURL\uu$, $\ff$ and $\CURL\ff$ are in $\bvec{H}^s(\Th)$, $p \in H^{s+1}(\Th)$, $\frac12 < s < k+1$, and the tangential components of $\uu$, $\CURL\CURL\uu$, and $\ff$ on all edges are integrable.
Then it holds
$$
\begin{aligned}
& \norm[\Hcurl{\Omega}]{\uu - \uu_h} \lesssim
h^{s} 
\big(
\seminorm[\Hscurl{s}{\Th}]{\uu}
+ \seminorm[\bvec{H}^s(\Th)]{\CURL\CURL\uu}
+ \seminorm[\bvec{H}^s(\Th)]{\CURL\ff}
\big) \, , \\
& \seminorm[H^1(\Omega)]{p - p_h} \lesssim h^{s} 
\big(
\seminorm[\Hscurl{s}{\Th}]{\uu}
+ \seminorm[\bvec{H}^s(\Th)]{\CURL\CURL\uu}
+ \seminorm[H^{s+1}(\Th)]{p}
\big) \, .
\end{aligned}
$$ 
\end{corollary}

\section*{Acknowledgements}

The authors acknowledge the support of ANR ``NEMESIS'' (ANR-20-MRS2-0004).
LBDV also acknowledges the partial support of the PRIN 2017 grant ``Virtual Element Methods:  Analysis and Applications'' and the PRIN 2020 grant ``Advanced polyhedral discretizations  of  heterogeneous  PDEs  for  multi-physics  problems''.
DDP gratefully acknowledges the partial support of I-Site MUSE grant ``RHAMNUS'' (ANR-16-IDEX-0006).


\printbibliography

\end{document}